\documentclass{article}
\usepackage[utf8]{inputenc}
\usepackage[T1]{fontenc}
\usepackage[margin=1.2in]{geometry}
\usepackage{lmodern}
\usepackage{float,subcaption,comment}
\usepackage{mathtools}
\usepackage{amsthm}
\usepackage[svgnames]{xcolor}
\usepackage{mathrsfs} 
\usepackage{tikz,pgf,tkz-graph}
\usetikzlibrary{positioning, shapes.geometric, arrows.meta, calc, positioning}
\usepackage{chngcntr}
\usepackage{algorithm}
\usepackage{algpseudocode}
\usepackage{amsmath}
\allowdisplaybreaks
\usepackage{listings}

\usepackage{amssymb}
\usepackage{graphicx}
\usepackage{hyperref}
\usepackage[T1]{fontenc}
\usepackage[english]{babel}
\usepackage{longtable}
\usepackage{maple}

\usepackage{breqn}
\usepackage{textcomp}
\usepackage{booktabs}
\usepackage{natbib}

\usetikzlibrary{angles, quotes}
\usepackage{longtable}

\usetikzlibrary{arrows,shapes}
\usetikzlibrary{decorations.pathreplacing}
\usepackage[shortlabels]{enumitem}

\usepackage[capitalise]{cleveref}

\DeclareMathOperator{\diam}{diam}
\DeclareMathOperator{\disc}{disc}

\newtheorem{thr}{Theorem}

\newtheorem{claim}[thr]{Claim}

\newtheorem{proposition}[thr]{Proposition}
\newtheorem{lem}[thr]{Lemma}
\newtheorem{rem}[thr]{Remark}

\newtheorem{conj}[thr]{Conjecture}

\theoremstyle{definition}
\newtheorem*{definition*}{Definition}
\crefname{lem}{Lemma}{Lemmas}
\crefname{conj}{Conjecture}{Conjectures}

\newcommand{\C}{\mathbb{C}}
\newcommand{\R}{\mathbb{R}}

\newcommand*{\myproofname}{Proof}

\let\le\leqslant
\let\ge\geqslant
\let\leq\leqslant

\let\geq\geqslant
\let\nleq\nleqslant
\let\ngeq\ngeqslant

\usepackage{bookmark}

\newcommand{\abs}[1]{\left|#1\right|}
\DeclareMathOperator{\Lip}{Lip}

\newcommand*{\oD}{\overline \Delta}
\renewcommand*{\phi}{\varphi}
\renewcommand*{\emptyset}{\varnothing}
\renewcommand*{\Im}{\mathrm{Im}}

\def\P{\mathcal{P}}

\newcounter{row}
\newcounter{col}

\usetikzlibrary{positioning, shapes.geometric}

\title{On the maximum product of distances of diameter $2$ point sets}

\author{
	Stijn Cambie
 \thanks{Department of Computer Science, KU Leuven Campus Kulak-Kortrijk, 8500 Kortrijk, Belgium. 
 Supported by a FWO grant with grant number 1225224N. Email: \protect\href{mailto:stijn.cambie@hotmail.com}{\protect\nolinkurl{stijn.cambie@hotmail.com}}.}	
	\and Arne Decadt
 \thanks{Department of Electronics and Information Systems, Ghent University, 9000 Ghent, Belgium. Email: \protect\href{mailto:arne.decadt@gmail.com}{\protect\nolinkurl{arne.decadt@gmail.com}}.}	
    \and Yanni Dong
 \thanks{Key Laboratory of System Software, Institute of Software, Chinese Academy of Sciences, Beijing 100190,  P. R. China . Email: \protect\href{mailto:yannidong@outlook.com}{\protect\nolinkurl{yannidong@outlook.com}}.}	
    \and Tao Hu
    \thanks{School of Mathematics and Statistics, Xi'an Jiaotong University, Xi'an 710049, P. R. China. Email: \protect\href{mailto:hu_tao@stu.xjtu.edu.cn}{\protect\nolinkurl{hu_tao@stu.xjtu.edu.cn}}.}
\and Quanyu Tang
 \thanks{School of Mathematics and Statistics, Xi'an Jiaotong University, Xi'an 710049, P. R. China. Email: \protect\href{mailto:tang_quanyu@163.com}{\protect\nolinkurl{tang_quanyu@163.com}}.}
 }

 \date{\today}

\begin{document}
\lstset{basicstyle=\ttfamily,breaklines=true,columns=flexible}
\DefineParaStyle{Maple Bullet Item}
\DefineParaStyle{Maple Heading 1}
\DefineParaStyle{Maple Warning}
\DefineParaStyle{Maple Heading 4}
\DefineParaStyle{Maple Heading 2}
\DefineParaStyle{Maple Heading 3}
\DefineParaStyle{Maple Dash Item}
\DefineParaStyle{Maple Error}
\DefineParaStyle{Maple Title}
\DefineParaStyle{Maple Ordered List 1}
\DefineParaStyle{Maple Text Output}
\DefineParaStyle{Maple Ordered List 2}
\DefineParaStyle{Maple Ordered List 3}
\DefineParaStyle{Maple Normal}
\DefineParaStyle{Maple Ordered List 4}
\DefineParaStyle{Maple Ordered List 5}
\DefineCharStyle{Maple 2D Output}
\DefineCharStyle{Maple 2D Input}
\DefineCharStyle{Maple Maple Input}
\DefineCharStyle{Maple 2D Math}
\DefineCharStyle{Maple Hyperlink}

\maketitle

\begin{abstract}
    We consider a problem posed by Erd\H{o}s, Herzog and Piranian on the maximum product of distances of a point set of order $n$ with a given diameter. We prove that it is sufficient to consider convex polygons and obtain results on the structure of the diameter graph.
    We also give constructions that drastically improve on the regular $n$-gons, sketching what the extremal polygons should look like, while presenting results indicating that one cannot hope to characterize the extremal polygons in general for even orders.
\end{abstract}

\section{Introduction}

\subsection{Motivation}
Let $\mathbf{z}=(z_1,\dots,z_n)\in\C^n$ be a configuration of points in the plane and set
\[
\Delta(\mathbf{z})
= \prod_{i\neq j}|z_i-z_j|
= \prod_{1\le i<j\le n}|z_i-z_j|^2.
\]

A convenient way to view this functional is through polynomial discriminants. If
\[
p(z)=\prod_{k=1}^n (z-z_k)
\]
is the monic polynomial with zeros $\{z_k\}_{k=1}^n$, then its discriminant satisfies
\[
\disc(p)=\prod_{1\le i<j\le n}(z_i-z_j)^2,
\qquad\text{hence}\qquad
|\disc(p)|=\prod_{1\le i<j\le n}|z_i-z_j|^2=\Delta(z_1,\dots,z_n).
\]
In other words, maximizing $\Delta$ under geometric constraints on the roots is exactly an
absolute discriminant maximization problem for monic polynomials. This point of view goes
back to Erd\H{o}s, Herzog and Piranian~\cite{ErdosHerzogPiranian1958} and Pommerenke's subsequent work~\cite{Pommerenke1961} on metric properties of complex polynomials.

For a finite point set $X\subset\C$, we write \(\diam(X)=\max\{|x-y|: x,y\in X\}\). Following Erd\H{o}s Problem \#1045~\cite{BloomErdosProblem1045}, we work under the standard diameter constraint
\[
|z_i-z_j|\le 2 \qquad \text{for all } i,j.
\]
Equivalently, writing $P$ for a configuration whose vertex set is $\{z_1,\dots,z_n\}$,
we assume $\diam(P)=2$ (if not, rescale it to have diameter $2$).
This rescaling is essential: if we multiply all points by a factor $s>0$, then every distance
$|z_i-z_j|$ scales by $s$, and therefore
\[
\Delta(sz_1,\dots,sz_n)=s^{n(n-1)}\Delta(z_1,\dots,z_n),
\]
so any maximizer must saturate the diameter constraint.
We will use the terms ``maximizer'', ``extremal polygon'' (a maximizer turns out to be a convex polygon, Proposition~\ref{prop:extreme_point_of_conv_1}) and ``extremal configuration'' to refer to a set of vertices that maximizes \(\Delta\) under this standard diameter constraint.

\subsection{Normalization and benchmark configurations}

For a polygon $P$ of diameter $2$ (otherwise we rescale), we let $\oD(P)=\frac{\Delta(P)}{n^n}$ be the normalized discriminant of the polygon $P$. We also define $\oD_{\max}(n)=\frac{\Delta_{\max}(n)}{n^n}$, where
\[
\Delta_{\max}(n)=\sup\Bigl\{\Delta(z_1,\dots,z_n):\ \max_{i,j}|z_i-z_j|\le2\Bigr\}\footnote{By translation invariance we may assume $z_1=0$, hence $|z_j|\le 2$ for all $j$ and the feasible set is compact.
Since $\Delta$ is continuous, the supremum is attained; moreover any maximizer must satisfy $\diam(P)=2$ by scaling.}.
\]

The quotient by $n^n$ is a natural benchmark because it matches the exact value of $\Delta$ for regular even $n$-gons of diameter $2$,
and it captures the leading asymptotics in the odd case. Concretely, as summarized in the Erd\H{o}s Problems exposition of Problem \#1045~\cite{BloomErdosProblem1045},
when the $z_i$ are the vertices of a regular $n$-gon (rescaled to have diameter $2$) one has
\[
\Delta(\text{regular }n\text{-gon})=n^n\quad\text{if }n\text{ is even},
\]
while for odd $n$,
\[
\Delta(\text{regular }n\text{-gon})
=\cos(\pi/2n)^{-n(n-1)}\,n^n
\sim e^{\pi^2/8}n^n.
\]
After normalization, this means $\oD(\text{regular }n\text{-gon})=1$ for even $n$ and
$\oD(\text{regular }n\text{-gon})\sim e^{\pi^2/8}$ for odd $n$.
The conjectural picture originating in~\cite{ErdosHerzogPiranian1958} predicts that regular polygons should be extremal at least for odd $n$,
and more broadly suggests that $\oD(P)\le \exp(\pi^2/8)$ might hold universally.
On the other hand, even obtaining sharp bounds of the correct order for general configurations is nontrivial:
Pommerenke~\cite{Pommerenke1961} proved that under the diameter constraint $|z_i-z_j|\le 2$ one has
$\Delta \le 2^{O(n)}n^n$, showing that the normalization by $n^n$ isolates the genuinely geometric (and subtle) part of the problem.

\subsection{Background and challenges}

Our objective couples all $\binom{n}{2}$ pairwise distances multiplicatively.
Compared to many classical isodiametric and extremal polygon questions, this creates a far more rigid and rugged optimization landscape:
small local perturbations can improve some distances while degrading many others, and the logarithm of the objective
becomes a dense sum of pairwise interaction terms. This typically leads to a combination of
(i) strong combinatorial constraints (which distances equal the diameter at an extremum) and
(ii) genuinely nonconvex analytic optimization in high dimension (see, for example, standard discussions of nonconvexity and
global optimization barriers in~\cite{nocedal2006numerical}).

It is instructive that even problems optimizing functionals that are simpler than $\Delta$ remain only partially understood.
For instance, discrete isoperimetric and small-polygon questions related to perimeter, width, or area
often require substantial geometric classification and, for specific $n$, heavy computational or symbolic assistance.
Audet, Hansen and Messine~\cite{AUDET2007135} resolved the convex small octagon with maximum perimeter by combining geometric reasoning
with interval-arithmetic-based global optimization.
Audet et al.~\cite{Audet2011SmallHexagonHeptagonSumDistances} treated small hexagons and heptagons for a sum-of-distances objective,
again through nonconvex programs with a delicate global analysis.
Audet, Hansen and Svrtan~\cite{AudetHansenSvrtan2021Symbolic} used symbolic elimination to obtain exact algebraic characterizations in the octagonal area setting, but their methodology
explicitly relies on an axial symmetry conjecture to make even this case tractable. Bingane~\cite{Bingane2021MaxPerimWidthSmallPolygon} constructed families of convex small $n$-gons with $n=2^s$ whose perimeters and widths are within high-order error of the (unknown) optima,
and in doing so disproved a natural diameter-graph conjecture for $s\ge 4$.
These works emphasize a common theme: even when the objective involves only a subset of distances or a lower-order functional,
complete solutions quickly become case-dependent, and some of the sharpest statements remain conditional on symmetry assumptions
or yield only near-optimal constructions for infinite families of $n$.

In our setting, these difficulties are amplified rather than alleviated: maximizing $\Delta$ forces one to control all pairwise distances simultaneously,
and the extremizers appear to depend sensitively on the global combinatorics of diameter pairs.
This is precisely why it is important, from a conceptual standpoint, to (a) identify structural constraints that any extremizer must satisfy,
(b) develop asymptotic constructions that show the conjectural landscape is rich and, importantly, falsifiable,
and (c) formulate refined conjectures that are compatible with both the combinatorics and the analysis.
Our results are aimed at these goals, and they provide evidence that a complete resolution of the original conjectural picture
is genuinely hard rather than merely technically incomplete.

\subsection{Main contributions}

Our first contribution is a structural theorem restricting the diameter graph of an extremal configuration.
Here the diameter graph is the graph on the vertex set where edges correspond to pairs at distance equal to the diameter.
Diameter graphs of polygons have a rich history; for instance, Foster and Szab\'o~\cite{Foster2007DiameterGO} proved a conjecture of Graham
by translating geometric extremality into strong combinatorial constraints on diameter graphs.
In the present paper, we show that extremizers for $\Delta$ must have a constrained diameter-graph shape.

\begin{thr}\label{thr:structure_diametergraph}
Let $P=\{z_1,\dots,z_n\}\subset\mathbb C$ satisfy $\diam(P)\le 2$ and attain $\Delta(P)=\Delta_{\max}(n)$.
Then the diameter graph on vertex set $P$ is either unicyclic or a caterpillar.
\end{thr}

Second, we provide explicit extremal and conjectured extremal constructions for small $n$ and extract numerical evidence
about the values of $\oD(P)$ and the likely combinatorial types that occur at optimality. In Section~\ref{sec:constructionsforsmalln}, we present some extremal and conjectured extremal polygons for $3 \le n \le 12$ and list the expected values of $\oD(P)$ for small (even) order $n$. 

Third, motivated by the Erd\H{o}s Problems Forum discussion of \#1045~\cite{BloomErdosProblem1045}, we address an asymptotic subquestion
that captures the behavior along a natural infinite subsequence of even orders. While the exact even-$n$ extremizers
appear too complex to characterize uniformly, one can still obtain robust asymptotic lower bounds that rule out overly naive conjectural pictures.

From the complex nature of the constructions, we are convinced that determining the extremal polygon for even $n$ in general seems out of reach with current techniques. Still, the following theorem addresses a subquestion raised by Thomas Bloom\footnote{This subquestion was raised in the comments of the Erd\H{o}s Problems forum thread \url{https://www.erdosproblems.com/forum/thread/1045}, in a post by Thomas Bloom (13:57 on 03 Oct 2025; forum timestamp).}, may provide a good approximation to the limsup of $\oD$ over even values of $n$.

\begin{thr}\label{thr:6mult}
For every $n$ divisible by $6$, there exists a set $P=\{z_1,\dots,z_n\}\subset\mathbb C$ with $\diam(P)\le 2$
such that, as $n\to\infty$ along multiples of $6$,
\[
\oD(P) \to C_{*} = \frac{3^{9/4}}{2^3}\exp\left( \frac{\pi^2-2 \sqrt 3 \pi}{8} \right)\approx 1.304457.
\]
We also have $\liminf_{n\to\infty} \oD_{\max}(n) \ge C_{*}^{1/9}>1$.
\end{thr}

Sothanaphan~\cite{Natso25} obtained simultaneously (inspired by our progress) a first explicit constant-factor improvement for $\liminf_{n\to\infty}\oD_{\max}(n)$ along even $n$.
In Section~\ref{sec:unif_lower_bnd} we refine and extend the approach of~\cite{Natso25} to obtain the following
uniform lower bound.

\begin{thr}\label{thm:main_unifom_even_n_1}
Along even integers $n\to\infty$ one has
\[
\liminf_{\substack{n\to\infty\\ n\ \mathrm{even}}}\oD_{\max}(n)
\ge
\exp\left(\frac{7}{24}\zeta(3)-\frac{\pi^4}{864}\right)
\approx 1.26853.
\]
\end{thr}

We expect the true value of this $\liminf$ to be larger, but determining it appears difficult,
likely due to the complexity of extremal configurations for finite~$n$.

Finally, we propose a refined conjectural description of extremizers that is compatible with the structural restrictions above and with the observed small-$n$ behavior.
An important meta-point for the interpretation of our results is that the original conjectural landscape around $\Delta$ is falsifiable:
even if one does not settle the main conjecture, proving structural constraints and producing constructions that significantly improve on the regular even $n$-gons
already rules out broad classes of naive conjectures and helps isolate what a final resolution must look like.
This mirrors the experience in related isodiametric polygon problems discussed above, where progress often comes from a combination of structural graph restrictions,
symmetry heuristics, and carefully validated extremal constructions rather than from a single closed-form characterization for all $n$.

\subsection{Organization and open conjectures}

We now summarize the organization of the paper and state the main conjectural properties suggested by our results.
Section~\ref{sec:diam_graph} proves Theorem~\ref{thr:structure_diametergraph} and develops general constraints on extremal configurations.
Section~\ref{sec:constructionsforsmalln} presents explicit constructions and numerical evidence for $3\le n\le 12$.
Section~\ref{sec:diam_graph_6k} gives an explicit asymptotic construction along the subsequence $6\mid n$, proving Theorem~\ref{thr:6mult}.
Section~\ref{sec:unif_lower_bnd} develops a separate approach that yields an improved uniform asymptotic lower bound along all even orders, proving Theorem~\ref{thm:main_unifom_even_n_1}.

Finally, we record a conjectural description of the maximizers suggested by our results and by computations for small values of $n$.

\begin{conj}\label{conj:extrpolygonproperties}
Fix $n\ge 3$. Let $P=\{z_1,\dots,z_n\}\subset\C$ be an $n$-gon with $\diam(P)\le 2$ that attains
$\Delta(P)=\Delta_{\max}(n)$.
Then:
\begin{enumerate}
    \item[(i)] If $n$ is odd, $P$ is a regular $n$-gon.
    \item[(ii)] If $n$ is even, $P$ has an axis of symmetry. Moreover, if $6\mid n$, then $P$ is invariant under rotation by $2\pi/3$
    (and hence has dihedral symmetry compatible with this $120^\circ$ rotational symmetry).
    \item[(iii)] If $n$ is even, the diameter graph of $P$ is obtained from a cycle $C_{n-3}$ by attaching three pendant edges to vertices of the cycle.
\end{enumerate}
\end{conj}

In particular, Conjecture~\ref{conj:extrpolygonproperties}(i), or proving that $\oD(P) \le \exp( \pi^2/8)$, is the main remaining challenge.
%for every $P$ with $\diam(P)\le 2$, 

\subsection{Asymptotic notation}
Let $f,g:\mathbb{N}\to\mathbb{R}$ with $g(n)>0$ for all sufficiently large $n$. 
We write $f(n)=O(g(n))$ as $n\to\infty$ if there exist constants $C>0$ and $n_0$ such that $|f(n)|\le Cg(n)$ for all $n\ge n_0$; 
we write $f(n)=o(g(n))$ if $\lim_{n\to\infty} f(n)/g(n)=0$; 
and we write $f(n)\sim g(n)$ if $\lim_{n\to\infty} f(n)/g(n)=1$.

\subsection{Declaration of AI usage}

We used an AI assistant (ChatGPT, model: GPT-5.2 Pro) in the brainstorming process (only) for Section~\ref{sec:unif_lower_bnd}.

\section{On the structure of extremal configurations}\label{sec:diam_graph}

In this section we discuss several structural properties that any extremal configuration must satisfy.

We begin with the associated diameter graph. A classical theorem of Hopf and Pannwitz~\cite{HP33} yields an a priori upper bound on the number of edges in this graph; see also Pach's exposition~\cite[Theorem~2]{Pach2013}.

\begin{lem}[\cite{HP33}]\label{lem:hopf_Pannwitz1}
For every $n\ge 3$, any set of $n$ points in $\mathbb R^2$ determines at most $n$
pairs of points at the maximum distance.
\end{lem}

Next, we look at results that imply a lower bound on the number of edges of the diameter graph.

\begin{lem}\label{lem:degge1}
In any maximizer, for every $k$ there exists $j\ne k$ with $|z_k-z_j|=2$.
\end{lem}

\begin{proof}
Since the regular $n$-gon of diameter $2$ has $\Delta>0$,
any maximizer $(z_1,\dots,z_n)$ must satisfy $\Delta(P)>0$, hence all points are distinct. Fix all points except $z_k$. The feasible set for $z_k$ is the
intersection of closed discs $\bigcap_{j\ne k}\overline D(z_j,2)$, a compact
convex set. 
If $|z_k-z_j|<2$ for all $j\ne k$, then $z_k$ is an interior point of this set. The function $z\mapsto \sum_{j\ne k}\log|z-z_j|$ is
harmonic away from $\{z_j\}_{j\ne k}$, hence cannot attain a local maximum at an interior point. 
This is a contradiction.
\end{proof}

This allows us to prove that extremal configurations must be convex polygons with \(n\) vertices.
This considerably simplifies optimization procedures.

\begin{proposition}\label{prop:extreme_point_of_conv_1}
Let $z_1,\dots,z_n\in\mathbb C$ be a maximizer, and set
$K =\operatorname{conv}\{z_1,\dots,z_n\}$.
Then every $z_i$ is an extreme point of $K$.
\end{proposition}

\begin{proof}
By Lemma~\ref{lem:degge1}, for each $i$ there exists $j\ne i$ with $|z_i-z_j|=2$.

We first show that $\diam(K)=2$. Since $|z_r-z_s|\le 2$ for all $r,s$, it suffices to
prove that this inequality extends to all $x,y\in K$.
Write $x=\sum_r \alpha_r z_r$ and $y=\sum_s \beta_s z_s$ with $\alpha_r,\beta_s\ge 0$ and
$\sum_r\alpha_r=\sum_s\beta_s=1$. Then \(x-y=\sum_{r,s}\alpha_r\beta_s\,(z_r-z_s)\), hence by the triangle inequality,
\[
|x-y|\le \sum_{r,s}\alpha_r\beta_s\,|z_r-z_s|
\le 2\sum_{r,s}\alpha_r\beta_s=2.
\]
Thus $\diam(K)\le 2$. On the other hand, for each $i$ we can pick $j$ with $|z_i-z_j|=2$,
so $\diam(K)\ge 2$. Therefore $\diam(K)=2$, and every such pair $(z_i,z_j)$ is a
diametral pair of $K$.

It is well known that in a compact convex set $K\subset\mathbb R^2$, every diametral pair consists of exposed (hence extreme) points; see~\cite[Proposition~1.1(a)]{BrGoMe17}.
Hence $z_i$ and $z_j$ are extreme points of $K$.

Applying this to each diametral pair $(z_i,z_j)$ shows that every $z_i$ is an extreme point
of $K$.
\end{proof}

The following two lemmas will allow us to use the theory on thrackles to significantly reduce the possible  diameter graphs.
Here a (linear) thrackle is a graph drawn in the plane with edges represented with straight lines for which every pair of edges meet exactly once.

\begin{lem}\label{lem:conn}
Let \(G\) be the diameter graph of an extremal configuration.
Then \(G\) is connected.
\end{lem}

\begin{proof}
Assume for contradiction that $G$ is disconnected. Then there is a partition
$\{1,\dots,n\}=I_1\sqcup I_2$ with $I_1,I_2\neq\emptyset$ such that
$|z_i-z_j|<2$ for all $i\in I_1$, $j\in I_2$ (otherwise an edge of length $2$
would connect the parts).  Set
\[
S_1=\{z_i:\ i\in I_1\},\qquad S_2=\{z_j:\ j\in I_2\}.
\]
Fix $s_1\in S_1$. For $x\in\mathbb C$ consider the translated set
\[
S_1(x)=S_1-s_1+x=\{a-s_1+x:\ a\in S_1\},
\]and the combined configuration
\(
\mathcal Z(x)=S_2 \cup S_1(x).
\)
Distances within $S_1$ and within $S_2$ are preserved; only cross distances vary with $x$.

Define
\[
\Omega=\bigcap_{a\in S_1,\ b\in S_2}\overline{D(b+s_1-a,2)}.
\]
Then $\mathcal Z(x)$ satisfies $|u-v|\le 2$ for all cross pairs
$(u,v)\in S_1(x)\times S_2$ if and only if $x\in\Omega$.
In particular $s_1\in\Omega$, so $\Omega$ is nonempty, compact, and convex.

Let
\[
\Phi(x)=\sum_{a\in S_1,\ b\in S_2}\log|x-(b+s_1-a)|.
\]
Up to an additive constant independent of $x$ (coming from within-part distances),
we have
\[
\log\Delta(\mathcal Z(x))=2\,\Phi(x)+\mathrm{const}.
\]
Since the original configuration is a global maximizer of $\Delta$ under
$|z_i-z_j|\le 2$, the point $x=s_1$ maximizes $\Phi$ over $\Omega$.

Let $C=\{\,b+s_1-a:\ a\in S_1,\ b\in S_2\,\}$.
Because the points in an extremal configuration are pairwise distinct,
we have $a\neq b$ for $a\in S_1$, $b\in S_2$, hence $s_1\notin C$.
Moreover, since $|a-b|<2$ for all $a\in S_1$, $b\in S_2$, we have
\[
|s_1-(b+s_1-a)|=|a-b|<2,
\]
so $s_1$ lies in every open disk $D(b+s_1-a,2)$, hence
$s_1\in\operatorname{Int}(\Omega)$.

Therefore $s_1$ is an interior point of the open set
$U=\operatorname{Int}(\Omega)\setminus C$, and $\Phi$ is harmonic on $U$.
Let $U_0$ be the connected component of $U$ containing $s_1$.
Then $\Phi$ attains its maximum on $U_0$ at the interior point $s_1$.
By the strong maximum principle, $\Phi$ must be constant on $U_0$.

On the other hand, writing
\[
P(z)=\prod_{a\in S_1,\ b\in S_2}\bigl(z-(b+s_1-a)\bigr),
\]
we have $\Phi(z)=\log|P(z)|$ on $U_0$, and $P$ is a nonconstant polynomial.
Thus $P$ is holomorphic and nonconstant on $U_0$, so by the open mapping theorem
its image is open, implying $|P|$ (hence $\Phi$) cannot be constant on $U_0$.
This contradiction shows that $G$ must be connected.
\end{proof}

As a corollary of the triangle inequality, one can deduce the following elementary lemma.

\begin{lem}[adaptation~of~Lemma~2(a)~of~\cite{Datta1997ADI}]\label{lem:DiametersIntersect}
In the diameter graph of any maximizer all diameters intersect.
\end{lem}

As promised, this allows us to use the results on thrackles to reduce the possible diameter graphs to caterpillars and similar unicyclic graphs, which is done in the following more precise version of~\cref{thr:structure_diametergraph}.

\begin{thr}\label{thm:Woodall}
The diameter graph of a maximizer is a caterpillar or it consists of a cycle with an odd number of vertices together possibly with extra vertices all of which are joined to vertices of the odd cycle by edges (so every edge in the graph is incident with at least one vertex of the odd cycle).
\end{thr}

\begin{proof}
Consider the straight-line drawing of $G$ with vertices at $z_1,\dots,z_n$ and edges
as the segments joining diameter pairs.
By Lemma~\ref{lem:DiametersIntersect}, any two edges intersect.
Moreover, since every $z_i$ is a vertex of $\operatorname{conv}\{z_1,\dots,z_n\}$ by Proposition~\ref{prop:extreme_point_of_conv_1},
no edge contains a third vertex in its interior and two edges cannot overlap;
hence any two edges meet exactly once.
Therefore $G$ admits a straight thrackle.

By~\cite[Theorem~2]{woodall1971thrackles}, $G$ is either a union of disjoint
caterpillars or consists of an odd cycle together possibly with additional vertices,
each adjacent to a vertex of the odd cycle.
Finally, Lemma~\ref{lem:conn} implies that $G$ is connected, so in the first case
$G$ is a caterpillar.
\end{proof}

\begin{lem}[adaptation~of~Lemma~3~of~\cite{Datta1997ADI}]
    \label{lem:NoC4}
    The diameter graph of a maximizer does not contain an even cycle.
\end{lem}

Finally, the theory on nonlinear programming also gives us a set of equations that any extremal configuration should satisfy.
We can rephrase the optimization problem as follows.
\begin{equation}
\begin{aligned}
& \underset{\mathbf{z}\in \C^n}{\text{maximize}}
& & f(\mathbf{z}) = \sum_{1 \le j < k \le n} \log \left( |z_k - z_j|^2 \right), \\
& \text{subject to}
& & g_{j,k}(\mathbf{z}) = |z_k - z_j|^2 - 4 \leq 0, \quad \forall \, 1 \le j < k \le n.
\end{aligned}
\tag{NLP}\label{NLPKKT}
\end{equation}

For any feasible solution \(\mathbf{z}\) to \cref{NLPKKT}, we define the active set 
\[
\mathcal A(\mathbf z)=\bigl\{\{a,b\}\subseteq\{1,\dots,n\}:\ 1\le a<b\le n,\ g_{a,b}(\mathbf z)=0\bigr\}.
\]

\begin{thr}\label{thm:KKTNec}
For any local maximizer \(\mathbf{z}=(z_1,\dots,z_n)\) to \cref{NLPKKT},
there exist Lagrange multipliers \(\lambda_{j,k}\ge 0\) for all \(1\le j<k\le n\)
such that
\[
\lambda_{j,k}\, g_{j,k}(\mathbf z)=0\qquad(1\le j<k\le n),
\]
and for every \(k\in\{1,\dots,n\}\),
\begin{equation}\label{eq:summationth12}
\sum_{j\neq k}\frac{1}{z_j-z_k}
=
\sum_{j<k}\lambda_{j,k}\,(\overline{z_j}-\overline{z_k})
+
\sum_{j>k}\lambda_{k,j}\,(\overline{z_j}-\overline{z_k}).
\end{equation}
\end{thr}
% The multipliers are indexed by pairs \(1\le j<k\le n\); we set \(\lambda_{k,j}:=\lambda_{j,k}\) for convenience.
\begin{proof}
Identify \(\C^n\) with \(\R^{2n}\) by writing \(z_k=x_k+iy_k\).
At an optimal solution the points are pairwise distinct (otherwise \(f=-\infty\)),
so \(f\) is \(C^1\) in a neighborhood of the optimum.

We want to use \cite[Theorem~12.1, Definition~12.6 and around]{nocedal2006numerical},
the MFCQ constraint qualification. We now rephrase \cref{NLPKKT} as a nonlinear
minimization problem over \((x_1,y_1,\dots,x_n,y_n)\in \mathbb{R}^{2n}\):
\begin{equation*}
\begin{aligned}
& \underset{(\mathbf{x},\mathbf{y}) \in \mathbb{R}^{2n}}{\text{minimize}}
& & -f(x_1,y_1,\dots,x_n,y_n)
=-\sum_{1 \le j < k \le n} \log \left( (x_k-x_j)^2+(y_k-y_j)^2 \right), \\
& \text{subject to}
& & -g_{j,k}(x_1,\dots,y_n)
=4- (x_k-x_j)^2-(y_k-y_j)^2 \geq 0,
\quad \forall \, 1 \le j < k \le n .
\end{aligned}
\end{equation*}
The partial derivatives of \(f\) are, for each \(1\le k\le n\),
\[
\frac{\partial f}{\partial x_k}
=\sum_{j\neq k}\frac{2(x_k-x_j)}{(x_k-x_j)^2+(y_k-y_j)^2},
\qquad
\frac{\partial f}{\partial y_k}
=\sum_{j\neq k}\frac{2(y_k-y_j)}{(x_k-x_j)^2+(y_k-y_j)^2}.
\]
For the constraint \(g_{j,k}\) with \(1\le j<k\le n\), its partial derivatives are
\[
\frac{\partial g_{j,k}}{\partial x_m}=
\begin{cases}
2(x_k-x_j), & m=k,\\
2(x_j-x_k), & m=j,\\
0, & \text{otherwise,}
\end{cases}
\qquad
\frac{\partial g_{j,k}}{\partial y_m}=
\begin{cases}
2(y_k-y_j), & m=k,\\
2(y_j-y_k), & m=j,\\
0, & \text{otherwise.}
\end{cases}
\]

For MFCQ, we need to find a direction \(w\) for which
\(\nabla (-g_{j,k}) \cdot w > 0\) for all \(\{j,k\}\in \mathcal{A}(\mathbf{z})\).
Define \(w=(-\frac{x_1}2,-\frac{y_1}2,\dots,-\frac{x_n}2,-\frac{y_n}2)\),
and consider any \(\{j,k\}\in \mathcal{A}(\mathbf{z})\). Then
\[
\nabla (-g_{j,k}) \cdot w
= (x_k-x_j)^2+(y_k-y_j)^2=4>0.
\]
Hence, by \cite[Theorem~12.1, Definition~12.6 and around]{nocedal2006numerical},
there exist multipliers \(\lambda_{j,k}\ge 0\) for all \(1\le j<k\le n\) such that
\[
\lambda_{j,k}\,(-g_{j,k})(x_1,y_1,\dots,x_n,y_n)=0\qquad(1\le j<k\le n),
\]
and
\[
-\nabla f(x_1,y_1,\dots,x_n,y_n)
+\sum_{1\le j<k\le n}\lambda_{j,k}\,\nabla g_{j,k}(x_1,y_1,\dots,x_n,y_n)=0.
\]
Equivalently,
\[
\nabla f(x_1,y_1,\dots,x_n,y_n)
=\sum_{1\le j<k\le n}\lambda_{j,k}\,\nabla g_{j,k}(x_1,y_1,\dots,x_n,y_n).
\]
Now fix \(k\). Taking the \(x_k\)-component minus \(i\) times the \(y_k\)-component,
we get
\[
\frac{\partial f}{\partial x_k}- i \frac{\partial f}{\partial y_k}
=
\sum_{j<k}\lambda_{j,k}
\left(\frac{\partial g_{j,k}}{\partial x_k}
-i\frac{\partial g_{j,k}}{\partial y_k}\right)
+
\sum_{j>k}\lambda_{k,j}
\left(\frac{\partial g_{k,j}}{\partial x_k}
-i\frac{\partial g_{k,j}}{\partial y_k}\right).
\]
Substituting the derivatives gives
\[
\sum_{j\neq k}\frac{2(\overline{z_k-z_j})}{|z_k-z_j|^2}
=
\sum_{j<k} \lambda_{j,k}\cdot 2(\overline{z_k}-\overline{z_j})
+
\sum_{j>k} \lambda_{k,j}\cdot 2(\overline{z_k}-\overline{z_j}).
\]
Using \(\overline{z_k-z_j}/|z_k-z_j|^2 = 1/(z_k-z_j)\) and rearranging,
\[
\sum_{j\neq k}\frac{1}{z_j-z_k}
=
\sum_{j<k}\lambda_{j,k}\,(\overline{z_j}-\overline{z_k})
+
\sum_{j>k}\lambda_{k,j}\,(\overline{z_j}-\overline{z_k}),
\]
as claimed.
\end{proof}
Since \cref{thm:KKTNec} requires that \(\lambda_{j,k} g_{j,k}(\mathbf{z})=0\) for all \(j\) and \(k\), we know that the only Lagrange multipliers \(\lambda_{j,k}\) that can be non-zero are those for which \(\{j,k\}\in\mathcal{A}(\mathbf{z})\).

These constraints simplify the problem sufficiently that, for small values of $n$,
we can determine the extremal configurations explicitly; this is carried out in the next section.

\section{Extremal constructions for small $n$}\label{sec:constructionsforsmalln}

\subsection{Up to $4$ points}

For small integers \(n\leq 3\), the constructions are straightforward and trivial.
We include them here for completeness.
We use the convention that the empty product is equal to one.
\begin{proposition}
\(\overline{\Delta}_\max (0)= 1\),
\(\overline{\Delta}_\max (1)= 1\), \(\overline{\Delta}_\max (2)= 1\), and \(\overline{\Delta}_\max (3)= \frac{64}{27}\).
\end{proposition}
\begin{proof}
For $n=0$ and $n=1$, $\Delta$ is the empty product, which we take to be $1$; hence
$\overline\Delta_{\max}(0)=\overline\Delta_{\max}(1)=1$ by convention.

For $n=2$, we have $\Delta=|z_1-z_2||z_2-z_1|=|z_1-z_2|^2\le 2^2$, with equality when
$|z_1-z_2|=2$. Thus $\overline\Delta_{\max}(2)=2^2/2^2=1$.

For $n=3$, each factor $|z_i-z_j|\le 2$, so
$\Delta=\prod_{i\ne j}|z_i-z_j|\le 2^{6}$.
Equality holds for an equilateral triangle of side length $2$, hence
$\overline\Delta_{\max}(3)=2^6/3^3=64/27$.
\end{proof}

The first interesting case is $n=4$, where the regular square on the circle of diameter $2$ (e.g.\ $z_1=1$, $z_2=i$, $z_3=-1$, $z_4=-i$) is not optimal. For this square we have \(\overline\Delta(z_1,z_2,z_3,z_4)=\frac{\bigl(2^2(\sqrt2)^4\bigr)^2}{4^4}=1\). In fact, the optimum is attained by a kite-shaped configuration (see Fig.~\ref{fig:n=4}), with $\overline\Delta_{\max}(4)=16(7-4\sqrt3)\approx 1.1487\ldots$.
The square cannot be extremal because its diameter graph consists only of the two diagonals and is therefore disconnected, contradicting Lemma~\ref{lem:conn}.

\begin{figure}
\centering
\begin{tikzpicture}[scale=2.5, >=stealth]

    % --- Subfigure (b): n=6 (Right side) ---
    % % Shifted by 6 units to the right
    \begin{scope}[shift={(0,0)}, local bounding box=figB]
        % Grid
        \draw[step=0.5, gray!35, very thin] (-0.2,-1.2) grid (2.4,1.2);
        % Axes
        \draw[->, gray!30] (-0.2,0) -- (2.4,0);
        \draw[->, gray!30] (0,-1.2) -- (0,1.2);
        \coordinate (w1) at ({sqrt(3)},1);
        \coordinate (w2) at (0,0);
        \coordinate (w3) at ( {sqrt(3)},-1);
        \coordinate (w4) at (2,0);

        \foreach \i in {1,2,3,4} {
            \fill[black!80] (w\i) circle (0.03);
            \pgfmathanglebetweenpoints{\pgfpoint{0}{0}}{\pgfpointanchor{w\i}{center}}
            \let\myangle\pgfmathresult
            \node[black!80, font=\tiny, label={\myangle:$z_{\i}$}] at (w\i) {};
        }

        \draw[thick] (w1) -- (w2);
        \draw[thick] (w2) -- (w4);
        \draw[thick] (w2) -- (w3);
        \draw[thick] (w1) -- (w3);
    \end{scope}

\end{tikzpicture}
\caption{Construction for \(n=4\).}\label{fig:n=4}
\end{figure}
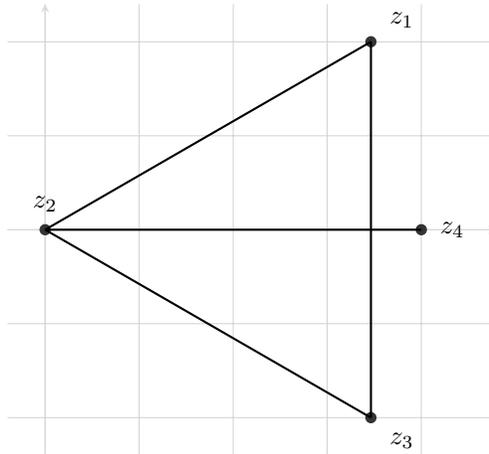

We prove the following proposition in two ways, a short one that is based on a plot, and a rigorous proof which uses~\cref{thm:KKTNec}.

\begin{proposition}\label{prop:n4}
\(\overline{\Delta}_{\max}(4)=16(7-4\sqrt3)\).
Moreover, any maximizer is congruent (up to translation and rotation, and relabeling of the points)
to the kite with vertex set \(\{0, 2, \sqrt3+i, \sqrt3-i\}\).
\end{proposition}

We first give a brief informal proof to motivate the geometry and identify the candidate maximizer. 
A complete proof is deferred to Appendix~\ref{app:proof_n4}.

\begin{proof}[Informal proof]
    As the diameter graph is connected and has no isolated vertices, it contains a spanning tree,
which is either $P_4$ or $K_{1,3}$.
    In the former case, without loss of generality, we can set $z_1=0$, $z_2=2$, $z_3=(2-x)+ i\sqrt{4x-x^2}$, and $z_4=y+ i\sqrt{4y-y^2}$ for parameters $0<x,y \le 1$ yet to be determined.
    So,
    $$
    \Delta(z_1,...,z_4)= (2^3)^2 \cdot 4x \cdot 4y \cdot \left( (2-x-y)^2 + (\sqrt{4x-x^2}-\sqrt{4y-y^2})^2\right)
    % &=32768 (-1+\cos(\alpha))\left(\cos(\alpha)+\cos(\alpha-\beta)-\cos(\beta)-\frac32\right)(1+\cos(\beta))\\
    % &=32768 (-1+\cos(\alpha))\left(\cos(\alpha)-2\sin\left(\tfrac \alpha2\right)\sin\left(\beta-\tfrac \alpha2\right))-\frac32\right)(1+\cos(\beta))
    % &=64 \left| 64 \left(2-2 \exp(\alpha i)\right)^{2} \left(-2-2 \exp(\beta i)\right)^{2} \left(-2+2 \exp(\alpha i)-2 \exp(\beta i)\right)^{2}\right|
    $$

    The maximum occurs when $\max\{x,y\}=1$, see~\cite[\texttt{Computation_P4_n4}]{githubRepo}, implying the diameter graph contains a $K_3.$ 
% Now one can notice that the maximum is attained on the boundary,
% \url{https://www.wolframalpha.com/input?i=maximum+x*y*%28+%28x-y%29%5E2+%2B+%28sqrt%284*x-x%5E2%29-+sqrt%284*y-y%5E2%29%29%5E2+%29+for+0+%3C%3D+x%3C%3D1%2C0%3C%3Dy+%3C%3D+1+exact} and thus we are in a situation where the graph has an $S_4$ as subgraph.
In the $K_{1,3}$ case, it is easy to see that the two furthest leaves need to be at distance $2$.

Thus, in both cases, the diameter graph needs to be unicyclic, a $K_3$ plus a pendent edge. A trigonometric argument implies that the pendent edge is a diagonal of the triangle. 
%(circle inside ellipse, or product of sines maximal
\end{proof}

\subsection{$n \in \{5,6\}$}

The \(n=5\) case can be seen as a consequence of the isoperimetric inequality.

\begin{proposition}
    \label{prop:n5}
    $\overline{\Delta}_\max (5) = (\frac45)^5(\sqrt 5 -1)^{10}$, the latter being attained (only) by a regular pentagon.
\end{proposition}

\begin{proof}
Let $\{z_1,\dots,z_5\}$ be a maximizer. By Lemma~\ref{lem:degge1}, we know that $\max_{i\ne j}|z_i-z_j|=2$. By Proposition~\ref{prop:extreme_point_of_conv_1}, the maximizer is in convex position.
Label the vertices cyclically as $z_0,\dots,z_4$. Then
\[
\Delta(z_0,\dots,z_4)=\Bigl(\prod_{k=0}^4 |z_{k+1}-z_k|\cdot \prod_{k=0}^4 |z_{k+2}-z_k|\Bigr)^2.
\]
Since all pairwise distances are $\le 2$, we have $|z_{k+2}-z_k|\le 2$ for all $k$, hence \(\prod_{k=0}^4 |z_{k+2}-z_k| \le 2^5\).

Next we bound the product of the side lengths. Datta~\cite{Datta1997ADI} proved that the perimeter of any convex
$n$-gon of diameter $1$ satisfies $\operatorname{per}\le 2n\sin(\pi/2n)$; scaling to diameter $2$ gives
\[
\sum_{k=0}^4 |z_{k+1}-z_k| \le 4\cdot 5\sin\Bigl(\frac{\pi}{10}\Bigr)=20\sin\Bigl(\frac{\pi}{10}\Bigr).
\]
By AM--GM,
\[
\prod_{k=0}^4 |z_{k+1}-z_k| \le \left(\frac{1}{5}\sum_{k=0}^4 |z_{k+1}-z_k|\right)^5
\le \bigl(4\sin(\pi/10)\bigr)^5.
\]
Using $\sin(\pi/10)=(\sqrt5-1)/4$, we obtain $\prod |z_{k+1}-z_k|\le (\sqrt5-1)^5$.
Combining the two bounds yields
\[
\Delta \le \bigl(2^5(\sqrt5-1)^5\bigr)^2 = 2^{10}(\sqrt5-1)^{10},
\]
and therefore
\[
\overline{\Delta}=\frac{\Delta}{5^5}\le \left(\frac45\right)^5(\sqrt5-1)^{10}.
\]

For equality, we must have $|z_{k+2}-z_k|=2$ for all $k$ and
$|z_{k+1}-z_k|$ constant in $k$ (by the strictness of AM--GM). Write this common side length as
$s=4\sin(\pi/10)=\sqrt5-1$.
Then each triangle $(z_k,z_{k+1},z_{k+2})$ has side lengths $(s,s,2)$, hence these triangles are congruent.
In particular, the interior angle at $z_{k+1}$ (equivalently, the turning angle) is the same for all $k$.
Since the turning angles sum to $2\pi$, each equals $2\pi/5$, so the pentagon is equiangular.
Being equilateral as well, it is regular, and it is the unique equality case.
\end{proof}

    \begin{figure}
        \centering
        \begin{tikzpicture}[
    % Global style configuration
    vertex/.style={circle, fill=black, inner sep=2pt, outer sep=0pt},
    edge/.style={thick},
    % Helper for layout
    node distance=1cm,
    scale=0.8
]

% ==========================================
% ROW 1: Caterpillar Graphs (1-3)
% ==========================================
\node[anchor=west] at (-1, 1.5) {\textbf{Caterpillars}};

% --- Graph 1: Path Graph (P6) ---
\begin{scope}[shift={(-1,0)}]
    \node[vertex] (p1) {};
    \node[vertex] (p2) [right=0.5cm of p1] {};
    \node[vertex] (p3) [right=0.5cm of p2] {};
    \node[vertex] (p4) [right=0.5cm of p3] {};
    \node[vertex] (p5) [right=0.5cm of p4] {};
    \node[vertex] (p6) [right=0.5cm of p5] {};
    \draw[edge] (p1) -- (p2) -- (p3) -- (p4) -- (p5) -- (p6);
\end{scope}

% --- Graph 2: Star Graph (S6) ---
\begin{scope}[shift={(5,0)}]
    \node[vertex] (c) {}; % Center
    \foreach \angle in {0, 72, 144, 216, 288} {
        \node[vertex] at (\angle:0.8cm) (l\angle) {};
        \draw[edge] (c) -- (l\angle);
    }
\end{scope}

% --- Graph 3: Double Star (2,2) ---
% Spine of length 2 (u-v), each has 2 leaves
\begin{scope}[shift={(9,0)}]
    \node[vertex] (u) at (-0.5, 0) {};
    \node[vertex] (v) at (0.5, 0) {};
    \draw[edge] (u) -- (v);
    
    % Leaves for u
    \node[vertex] (u1) [above left=0.4cm of u] {};
    \node[vertex] (u2) [below left=0.4cm of u] {};
    \draw[edge] (u1) -- (u) -- (u2);
    
    % Leaves for v
    \node[vertex] (v1) [above right=0.4cm of v] {};
    \node[vertex] (v2) [below right=0.4cm of v] {};
    \draw[edge] (v1) -- (v) -- (v2);
\end{scope}

% ==========================================
% ROW 2: Caterpillar Graphs (4-6)
% ==========================================
% \node[anchor=west] at (-1, -2) {\textbf{Caterpillars (4-6)}};

% --- Graph 4: Double Star (3,1) ---
% Spine u-v, u has 3 leaves, v has 1
\begin{scope}[shift={(1,-2.5)}]
    \node[vertex] (u) at (-0.5, 0) {};
    \node[vertex] (v) at (0.5, 0) {};
    \draw[edge] (u) -- (v);
    
    % 3 Leaves for u
    \node[vertex] (u1) [above left=0.4cm of u] {};
    \node[vertex] (u2) [left=0.5cm of u] {};
    \node[vertex] (u3) [below left=0.4cm of u] {};
    \foreach \n in {u1,u2,u3} \draw[edge] (u) -- (\n);
    
    % 1 Leaf for v
    \node[vertex] (v1) [right=0.5cm of v] {};
    \draw[edge] (v) -- (v1);
\end{scope}

% --- Graph 5: Spine-3 (1,1,1) ---
% Spine u-v-w, each has 1 leaf attached
\begin{scope}[shift={(5,-2.5)}]
    \node[vertex] (u) at (-1, 0) {};
    \node[vertex] (v) at (0, 0) {};
    \node[vertex] (w) at (1, 0) {};
    \draw[edge] (u) -- (v) -- (w);
    
    \node[vertex] (l1) [above=0.5cm of u] {};
    \node[vertex] (l2) [above=0.5cm of v] {};
    \node[vertex] (l3) [above=0.5cm of w] {};
    
    \draw[edge] (u) -- (l1);
    \draw[edge] (v) -- (l2);
    \draw[edge] (w) -- (l3);
\end{scope}

% --- Graph 6: Spine-3 (2,0,1) ---
% Spine u-v-w. u has 2 leaves, v has 0, w has 1.
\begin{scope}[shift={(9,-2.5)}]
    \node[vertex] (u) at (-1, 0) {};
    \node[vertex] (v) at (0, 0) {};
    \node[vertex] (w) at (1, 0) {};
    \draw[edge] (u) -- (v) -- (w);
    
    % 2 leaves on left
    \node[vertex] (l1) [above left=0.4cm of u] {};
    \node[vertex] (l2) [below left=0.4cm of u] {};
    \draw[edge] (l1) -- (u) -- (l2);
    
    % 1 leaf on right
    \node[vertex] (l3) [right=0.5cm of w] {};
    \draw[edge] (w) -- (l3);
\end{scope}

% ==========================================
% ROW 3: Unicyclic Graphs
% ==========================================
\node[anchor=west] at (-1, -4) {\textbf{Unicyclic (Odd Cycle)}};

% --- Graph 7: C5 + 1 Pendant ---
\begin{scope}[shift={(0,-6.5)}]
    % Draw Pentagon
    \node[vertex] (c1) at (90:0.6) {};
    \node[vertex] (c2) at (162:0.6) {};
    \node[vertex] (c3) at (234:0.6) {};
    \node[vertex] (c4) at (306:0.6) {};
    \node[vertex] (c5) at (18:0.6) {};
    \draw[edge] (c1)--(c2)--(c3)--(c4)--(c5)--(c1);
    
    % Pendant
    \node[vertex] (p) [above=0.5cm of c1] {};
    \draw[edge] (c1) -- (p);
    
    % \node[align=center, below=1cm] at (0,0) {$C_5 + 1$};
\end{scope}

% --- Graph 8: C3 + (1,1,1) ---
% Triangle, 1 pendant on each vertex
\begin{scope}[shift={(3.5,-6.5)}]
    \node[vertex] (t1) at (90:0.5) {};
    \node[vertex] (t2) at (210:0.5) {};
    \node[vertex] (t3) at (330:0.5) {};
    \draw[edge] (t1) -- (t2) -- (t3) -- (t1);
    
    % Pendants radiating out
    \node[vertex] (p1) at (90:1.0) {};
    \node[vertex] (p2) at (210:1.0) {};
    \node[vertex] (p3) at (330:1.0) {};
    
    \draw[edge] (t1)--(p1);
    \draw[edge] (t2)--(p2);
    \draw[edge] (t3)--(p3);
    
    % \node[align=center, below=1cm] at (0,0) {Dist: $(1,1,1)$};
\end{scope}

% --- Graph 9: C3 + (2,1,0) ---
% Triangle, top has 2 pendants, left has 1, right has 0
\begin{scope}[shift={(7,-6.5)}]
    \node[vertex] (t1) at (90:0.5) {};
    \node[vertex] (t2) at (210:0.5) {};
    \node[vertex] (t3) at (330:0.5) {};
    \draw[edge] (t1) -- (t2) -- (t3) -- (t1);
    
    % 2 on top
    \node[vertex] (p1a) at (60:1.0) {};
    \node[vertex] (p1b) at (120:1.0) {};
    \draw[edge] (p1a)--(t1)--(p1b);
    
    % 1 on left
    \node[vertex] (p2) at (210:1.0) {};
    \draw[edge] (t2)--(p2);
    
    % \node[align=center, below=1cm] at (0,0) {Dist: $(2,1,0)$};
\end{scope}

% --- Graph 10: C3 + (3,0,0) ---
% Triangle, top has all 3 pendants
\begin{scope}[shift={(10.5,-6.5)}]
    \node[vertex] (t1) at (90:0.5) {};
    \node[vertex] (t2) at (210:0.5) {};
    \node[vertex] (t3) at (330:0.5) {};
    \draw[edge] (t1) -- (t2) -- (t3) -- (t1);
    
    % 3 on top
    \node[vertex] (p1) at (90:1.0) {};
    \node[vertex] (p2) at (45:1.0) {};
    \node[vertex] (p3) at (135:1.0) {};
    
    \draw[edge] (p1)--(t1) (p2)--(t1) (p3)--(t1);
    
    % \node[align=center, below=1cm] at (0,0) {Dist: $(3,0,0)$};
\end{scope}

\end{tikzpicture}
        \caption{All diameter graphs for \(n=6\) allowed by \cref{thm:Woodall}.}
        \label{fig:caterpillarsAndUnicylic6}
    \end{figure}
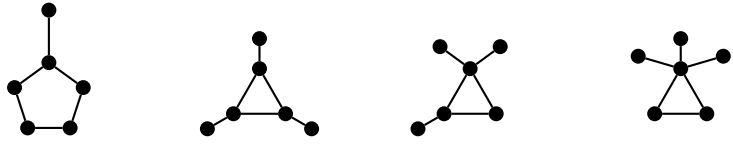

% For higher \(n\), with multiple computer assisted approaches, the maxima for $\oD(n)$ were numerically computed. 
For higher \(n\), we numerically estimated \(\overline{\Delta}_{\max}(n)\) by enumerating the admissible diameter-graph types and searching for local maximizers within each type. This was done by listing all caterpillars and unicyclic graphs allowed and for each one computing the local maxima.
The local maxima were found by a gradient descent algorithm, and by finding solutions to the equations of \cref{thm:KKTNec}.
%  The obtained results were consistent and hence we conjecture that our numerical approximation presents the extremal $n$-gons.
For various values of \(n\), we investigated the possible diameter graphs, but we cannot rule out the possibility that additional local maxima exist.
For the \(n=6\) case, the list is depicted in~\cref{fig:caterpillarsAndUnicylic6}. Here we found that 
\[\overline{\Delta}_\max (6) = \left(\frac{2^{4} (2-\sqrt{3}) (\sqrt{3}-1)}{3}\right)^6 = \frac{(2\sqrt{3}-2)^{18}}{3^6}\]
achieved by the same points as the \(n=4\) case, i.e. \(z_1= \sqrt3 + i\), \(z_2=0\), \(z_3= \sqrt3 - i\),  and \(z_4=2\), in addition to \(z_5=(\sqrt{3}-1)(1+i)\)  and \(z_6=(\sqrt{3}-1)(1-i)\).
For odd $n \le 11$, the regular $n$-gon was extremal.
For even values $n\leq 12$, we present constructions in the following subsections.

\subsection{The extremal octagon and decagon ($n \in \{8,10\}$)}

There are $20$ caterpillars of order $8$.
For each caterpillar used as the diameter graph, a computer program can search for (approximate) local maximizers of $\Delta$.
In this way, we observed that multiple local maxima can occur even for a fixed underlying caterpillar.
Only three caterpillars achieved values of $\oD$ exceeding $\frac54$.
These three caterpillars are precisely those that can be obtained by deleting one edge from a single unicyclic graph, since in each case the maximum is attained when eight distances saturate the diameter constraint.
An exact analysis of one of these caterpillars led to a configuration from which symmetry could be inferred.
Finally, under the same symmetry assumption, the conjectured optimum was computed to high precision.

The same result was obtained using the nonlinear optimization library IPopt with JuMP in Julia.
This alternative approach is limited to Float64 precision (about 15--17 significant decimal digits), so further refinement required other methods or optimizers.

Both approaches pointed to a single globally optimal construction, up to translation and rotation.

For $n=10$, there are $72$ caterpillars, and a similar approach can be used.
Again, this resulted in a single construction that is conjectured to be globally optimal.

The optimal configurations we found are depicted in~\cref{fig:EP1045n8-10}.
On the left, the conjecturally optimal octagon is presented, while on the right, the conjecturally optimal decagon is shown.
In particular, we note that the side lengths of these $n$-gons are not all equal; they are more irregular than one might have hoped.

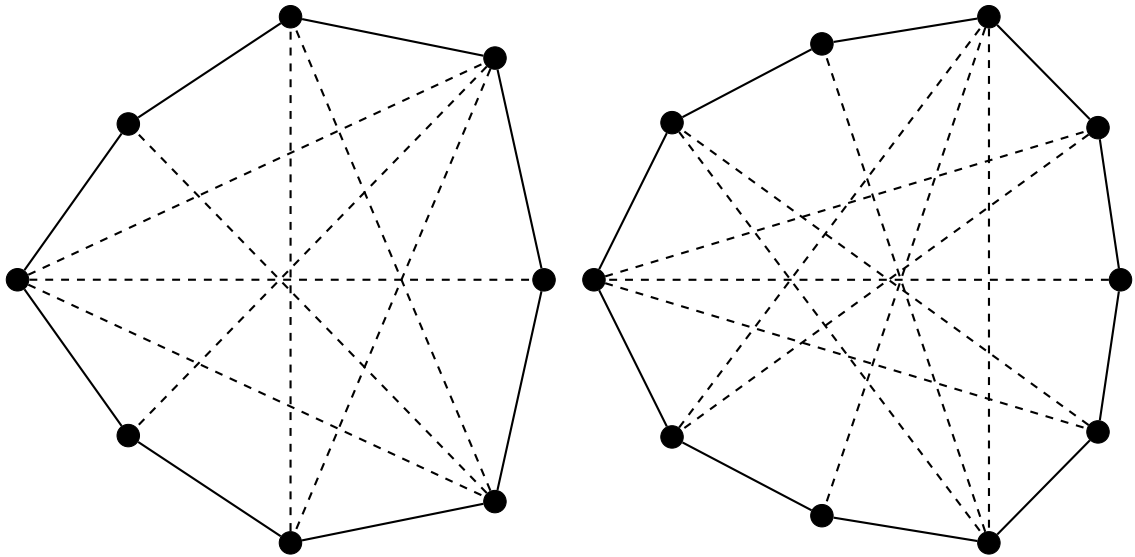
\begin{figure}[h]
    \centering
\begin{tikzpicture}[scale=3.5,rotate=-155.0763, main_node/.style={fill,draw,minimum size=0.25,circle,inner sep=3pt]}]

\node[main_node] (0) at (2.0, 0.0) {};
\node[main_node] (1) at (0.0, 0.0) {};
\node[main_node] (2) at (0.7107282904, 1.5289795482) {};
\node[main_node] (3) at (1.8680151738, 0.7145063403) {};
\node[main_node] (4) at (1.4811630304, 1.3439330628) {};
\node[main_node] (5) at (0.1863650561, 0.8430470274) {};
\node[main_node] (6) at (1.3686850581, -0.3596953555) {};
\node[main_node] (7) at (0.6381160031, -0.4697018811) {};

 \path[draw, thick]
\foreach \i/\j in {0/1,0/2,0/5,1/3,1/4,2/6,2/7,4/7}{(\i) edge[dashed] node {} (\j) }
\foreach \i/\j in {0/3,3/4,4/2,2/5,5/1,1/7,7/6,6/0}{(\i) edge node {} (\j) };

\end{tikzpicture}\quad \begin{tikzpicture}[scale=3.5,rotate=0, main_node/.style={fill,draw,minimum size=0.25,circle,inner sep=3pt}]

\node[main_node] (1) at (2.0, 0.0) {};
\node[main_node] (6) at (0.0, 0.0) {};
\node[main_node] (2) at  (1.9145726113151686, 0.5782834218632041) {};
\node[main_node] (10) at  (1.9145726113151686, -0.5782834218632041) {};
\node[main_node] (3) at  (1.5001169871403601, 1) {};
\node[main_node] (9) at  (1.5001169871403601, -1) {};
\node[main_node] (5) at  (0.2965753585155655, 0.5973376437582559) {};
\node[main_node] (7) at (0.2965753585155655, -0.5973376437582559) {};
\node[main_node] (4) at  (0.8658739665269451, 0.8967698307393999){};
\node[main_node] (8) at  (0.8658739665269451, -0.8967698307393999){};

 \path[draw, thick]
\foreach \i/\j in {1/6,6/2,6/10,10/5,2/7,3/9,5/9,7/3,3/8,9/4}{(\i) edge[dashed] node {} (\j) }
\foreach \i/\j in {10/1,1/2,2/3,3/4,4/5,5/6,6/7,7/8,8/9,9/10}{(\i) edge node {} (\j) };

\end{tikzpicture}
\caption{Illustration of the extremal construction for $n=8$ and $n=10$}\label{fig:EP1045n8-10}
\end{figure}

\subsection{The extremal dodecagon}

Based on numerical optimizations, it seems that for $n=6m$ an optimal configuration has $D_3$ dihedral symmetry (of order $6$).
Moreover, the diameter graph appears to consist of a cycle $C_{n-3}$ together with one pendant edge on each reflection axis.
Under these assumptions, maximizing $\Delta$ becomes more tractable.
Since $\Delta$ is invariant under rotations and translations, we may assume that one pendant edge lies on the real axis and that the center of rotation is at the origin.

Each pendant edge consists of two points, so there are six points on the three pendant edges in total.
For convenience, we place two points $z_0$ and $z_1=z_0+2$ on the real axis.
The group action then yields the remaining four points:
$e^{2\pi i/3}z_0$, $e^{2\pi i/3}z_1$, $e^{4\pi i/3}z_0$, and $e^{4\pi i/3}z_1$.

Each of the other $6m-6$ points lies in an orbit of size $6$ under the $D_3$ action.
Hence, it suffices to choose $m-1$ additional points, which we denote by $z_2,\dots,z_m$, and require that
\[
|z_{k+1}-z_k|=2 \qquad \text{for all } k<m.
\]
After applying the group action, these requirements already produce three connected components in the diameter graph.

To connect these components, we impose one additional constraint linking $z_m$ to an image of $z_m$ under the group action.
Since the dihedral group acts by isometries, this connects all components.
There are two natural choices; here we impose
\[
\bigl|e^{-2\pi i/3} z_m - e^{2\pi i/3}\,\overline{z_m}\bigr|=2,
\]
or equivalently,
\[
\bigl|z_m - e^{4\pi i/3}\,\overline{z_m}\bigr|=2.
\]
These edges are drawn in blue in~\cref{fig:n=12}.
Since $e^{-2\pi i/3} z_m$ and $e^{2\pi i/3}\overline{z_m}$ are complex conjugates, these edges are perpendicular to a reflection axis (and hence to the corresponding pendant edge).

In~\cref{fig:n=12} we illustrate the resulting constructions for $n=6$ and $n=12$, which we will examine in more detail later.

Let us define \(A_m=\{z_k: k\in\{0,\dots,m\}\}\ \cup\ \{\overline{z_k}: k\in\{2,\dots,m\}\}\), $\omega = e^{2\pi i/3}$ and $\widetilde{P}=\{\omega^t z:\ z\in A_m,\ t\in\{0,1,2\}\}$. We assume the vertices are distinct, so that \(|A_m|=2m\) and \(|\widetilde P|=6m\), and each \(u\in\widetilde P\) has a unique representation \(u=\omega^t z\) with \(z\in A_m\). Recall that the two points $z_0$ and $z_1$ are on the real axis. For \(\Delta(\widetilde{P})\) this means that
\begin{equation}\label{eq:DeltaDihedral}
\begin{aligned}
\Delta(\widetilde{P})=&\prod_{\substack{u,v\in \widetilde{P}\\u\neq v}}|u-v| \\
=&\left(\prod_{z\in A_m}\ \prod_{\substack{a,b\in\{1,\omega,\omega^2\}\\a\neq b}}|az-bz|\right)
\left(\prod_{z\in A_m}\ \prod_{\substack{y\in A_m\\y\neq z}}\ \prod_{a\in\{1,\omega,\omega^2\}}\ \prod_{b\in\{1,\omega,\omega^2\}}|az-by|\right)\\
=&\left(\prod_{z\in A_m} 27\,|z|^{6}\right)\ \prod_{z\in A_m}\ \prod_{\substack{y\in A_m\\y\neq z}} |z^3-y^3|^{3}\\
=&  \left(|z_0|^6|z_1|^6\prod_{k=2}^{m} |z_k|^{12}\right) 3^{6m} |z_0^3-z_1^3|^6  \left(\prod_{k=2}^{m} |z_0^3 - z_k^3|^{12} |z_1^3 -z_k^3|^{12} \right)\\
& \prod_{k=2}^{m} |z_k^3 -\overline{z_k}^3|^6 \prod_{\substack{ j \in \{2,...,m\}\\j\neq k}} |z_k^3 -z_j^3|^6  |z_k^3 - \overline{z_j}^3|^6 \\
=&\left(
3^{m}|z_0||z_1||z_0^3-z_1^3|\prod_{k=2}^{m} |z_k|^{2}  |z_0^3 - z_k^3|^2 |z_1^3 -z_k^3|^2 |z_k^3 -\overline{z_k}^3| \prod_{j=2}^{k-1} |z_k^3 -z_j^3|^2  |z_k^3 - \overline{z_j}^3|^2 
\right)^6.
\end{aligned}
\end{equation}

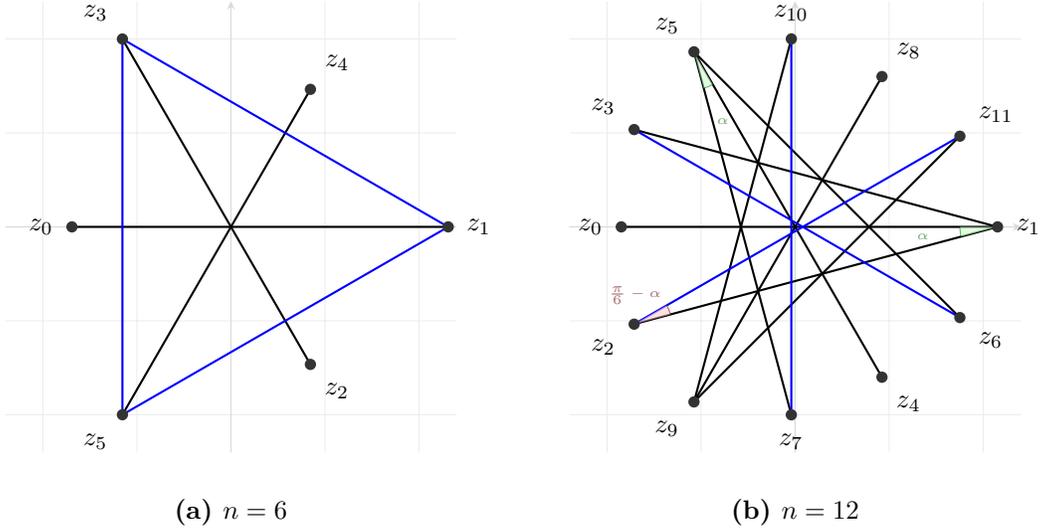
\begin{figure}
\centering
\begin{tikzpicture}[scale=2.5, >=stealth]

    % --- Subfigure (b): n=6 (Right side) ---
    % Shifted by 6 units to the right
    \begin{scope}[shift={(0,0)}, local bounding box=figB]
        % Grid
        \draw[step=0.5, gray!15, very thin] (-1.2,-1.2) grid (1.2,1.2);
        % Axes
        \draw[->, gray!30] (-1.2,0) -- (1.2,0);
        \draw[->, gray!30] (0,-1.2) -- (0,1.2);

        % Constants for n=6 (pi/12 = 15 deg)
        % z0 = tan(15)-1, z1 = tan(15)+1
        \pgfmathsetmacro{\xZeroSix}{tan(15)/sqrt(3)-1}
        \pgfmathsetmacro{\xOneSix}{tan(15)/sqrt(3)+1}

        % Generation Loop (Sets of 2 points, 3 rotations)
        \foreach \set/\rotAngle in {0/0, 1/120, 2/240} {
            \begin{scope}[rotate=\rotAngle]
                % Define base points on x-axis
                \coordinate (S0) at (\xZeroSix, 0);
                \coordinate (S1) at (\xOneSix, 0);

                % Draw Segment
                \draw[thick, black] (S0) -- (S1);

                % Save coordinates (0,1 -> 2,3 -> 4,5)
                \path (S0) coordinate (w\the\numexpr\set*2+0\relax);
                \path (S1) coordinate (w\the\numexpr\set*2+1\relax);
            \end{scope}
        }
        % Extra Connections
        \draw[thick, blue] (w1) -- (w3);
        \draw[thick, blue] (w5) -- (w1);
        \draw[thick, blue] (w3) -- (w5);

        % Labels for n=6
        \foreach \i in {0,...,5} {
            \fill[black!80] (w\i) circle (0.03);
            \pgfmathanglebetweenpoints{\pgfpoint{0}{0}}{\pgfpointanchor{w\i}{center}}
            \let\myangle\pgfmathresult
            \node[black!80, font=\tiny, label={\myangle:$z_{\i}$}] at (w\i) {};
        }
        \node[below, font=\bfseries] at (0,-1.4) {(a) $n=6$};
    \end{scope}

    % --- Subfigure (a): n=12 (Left side) ---
    \begin{scope}[shift={(3,0)},local bounding box=figA]
        % Grid
        \draw[step=0.5, gray!15, very thin] (-1.2,-1.2) grid (1.2,1.2);
        % Axes
        \draw[->, gray!30] (-1.2,0) -- (1.2,0);
        \draw[->, gray!30] (0,-1.2) -- (0,1.2);
        
        % Constants for n=12 (pi/24 = 7.5 deg)
        \pgfmathsetmacro{\xZero}{tan(7.5)/sqrt(3)-1}
        \pgfmathsetmacro{\xOne}{tan(7.5)/sqrt(3)+1}

        % Generation Loop
        \foreach \set/\rotAngle in {0/0, 1/120, 2/240} {
            \begin{scope}[rotate=\rotAngle]
                \coordinate (L0) at (\xZero, 0);
                \coordinate (L1) at (\xOne, 0);
                \coordinate (L2) at ($(L1) + (195:2)$); 
                \coordinate (L3) at ($(L1) + (165:2)$); 

                \draw[thick, black] (L0) -- (L1);
                \draw[thick, black] (L1) -- (L2);
                \draw[thick, black] (L1) -- (L3);
                
                \path (L0) coordinate (z\the\numexpr\set*4+0\relax);
                \path (L1) coordinate (z\the\numexpr\set*4+1\relax);
                \path (L2) coordinate (z\the\numexpr\set*4+2\relax);
                \path (L3) coordinate (z\the\numexpr\set*4+3\relax);
            \end{scope}
        }

        % Extra Connections
        \draw[thick, blue] (z2) -- (z11);
        \draw[thick, blue] (z3) -- (z6);
        \draw[thick, blue] (z10) -- (z7);
        
        % Angles
        \tikzset{
            angle style1/.style={draw=green!50!black, fill=green!20, opacity=0.6,angle eccentricity=2},
            angle style2/.style={draw=red!50!black, fill=red!20, opacity=0.6,angle eccentricity=2},
            label style1/.style={font=\tiny, text=green!40!black},
            label style2/.style={font=\tiny, text=red!40!black}
        }
        \pic [angle style1," $\alpha$"{label style1}] {angle = z0--z1--z2};
        \pic [angle style2, "$\frac{\pi}{6}-\alpha$"{label style2,shift={(-2.6em,0)}}] {angle = z1--z2--z11};
        \pic [angle style1, "$\alpha$"{label style1}] {angle = z7--z5--z4};

        % Labels
        \foreach \i in {0,...,11} {
            \fill[black!80] (z\i) circle (0.03);
            \pgfmathanglebetweenpoints{\pgfpoint{0}{0}}{\pgfpointanchor{z\i}{center}}
            \let\myangle\pgfmathresult
            \node[black!80, font=\tiny, label={\myangle:$z_{\i}$}] at (z\i) {};
        }
        \node[below, font=\bfseries] at (0,-1.4) {(b) $n=12$};
    \end{scope}

\end{tikzpicture}
\caption{Constructions for \(n=6\) and \(n=12\) assuming dihedral symmetry and 3 pendant edges along the symmetry axes.}\label{fig:n=12}
\end{figure}

For \(n=6\) and \(m=1\), no optimization needs to be done, as these assumptions lead to a single construction.
There is only \(z_0\) and \(z_1\), and the condition that \(|z_1-\overline{z_1} e^{\frac{4\pi}3 i}|=2\) means, since \(z_1\) is positive real, that \(z_1= \frac{2}{|1-e^{\frac{4\pi}3 i}|}=\frac{2}{\sqrt{3}} =\frac{2\sqrt{3}}{3}\).
In (a) of \cref{fig:n=12}, \(z_0= z_1-2 \), \(z_2=z_0 e^{\frac{2\pi}3 i}\), \(z_3=z_1 e^{\frac{2\pi}3 i}\), \(z_4=z_0 e^{\frac{4\pi}3 i}\), and \(z_5=z_1 e^{\frac{4\pi}3 i}\).

For \(n=12\) and \(m=2\), there are \(z_0\), \(z_1\) and \(z_2\) that we need to determine.
We know that \(z_0\) and \(z_1\) are real.
Furthermore, we have that \(z_0=z_1-2\), and since \(|z_1-z_2|=2\), we can write \(z_2=z_1 - 2 e^{i\alpha}\).

Since $\omega z_1,\omega z_2\in\widetilde P$ and $\diam(\widetilde P)\le 2$, we have
\[
|z_1-\omega z_1|=\sqrt3\,z_1\le 2,\qquad |z_2-\omega z_2|=\sqrt3\,|z_2|\le 2,
\]
hence $0<z_1\le 2/\sqrt3$ and $|z_2|\le 2/\sqrt3$.
Writing $z_2=z_1-2e^{i\alpha}$ gives
\[
|z_2|^2=z_1^2-4z_1\cos\alpha+4\le \frac{4}{3},
\]
so
\[
\cos\alpha \ge \frac{z_1^2+\frac{8}{3}}{4z_1}=\frac{z_1}{4}+\frac{2}{3z_1}
\ge \frac{\sqrt3}{2},
\]
where the last inequality uses $0<z_1\le 2/\sqrt3$.
Therefore $0\le \alpha\le \pi/6$, and $\alpha\neq 0$ since the vertices are distinct.

From \(|z_2-e^{\frac{4\pi}3 i} \overline{z_2}|=2\), it follows that 
\begin{align*}
2&=\left|z_1-2e^{i\alpha} - z_1 e^{\frac{4\pi}3 i} +2 e^{i\left(\frac{4\pi}3-\alpha\right)}\right|=\left|z_1 e^{\frac{\pi}3 i} -2e^{i\left(\frac{\pi}3+\alpha\right)} - z_1 e^{-\frac{\pi}3 i} +2 e^{-i\left(\frac{\pi}3+\alpha\right)}\right|\\
&=\left|z_1 2i \sin\left(\frac{\pi}3\right) - 4i  \sin \left(\frac{\pi}3+\alpha\right)\right|=\left|4 \sin \left(\frac{\pi}3+\alpha\right)-2 z_1 \sin\left(\frac{\pi}3\right)\right|.
\end{align*}

Hence
\begin{equation}\label{eq:pm-relation}
4\sin\Bigl(\frac{\pi}{3}+\alpha\Bigr)-2z_1\sin\Bigl(\frac{\pi}{3}\Bigr)=\pm 2.
\end{equation}
Solving~\eqref{eq:pm-relation} for $z_1$ gives
\begin{equation}\label{eq:z1-of-alpha-pm}
z_1(\alpha)=\frac{2\sin\bigl(\frac{\pi}{3}+\alpha\bigr)\mp 1}{\sin(\pi/3)}
=\frac{4\sin\bigl(\frac{\pi}{3}+\alpha\bigr)\mp 2}{\sqrt3}.
\end{equation}

We now determine the correct sign.
By construction, $z_1>0$ and the diameter constraint applied to the pair $\{z_1,\omega z_1\}$ gives
\[
|z_1-\omega z_1|=|1-\omega|\,z_1=\sqrt3\,z_1\le 2,
\]
hence $z_1\le 2/\sqrt3$.
If the ``$-$'' choice in~\eqref{eq:z1-of-alpha-pm} were taken (i.e.\ the right-hand side of~\eqref{eq:pm-relation} equals $-2$),
then
\[
z_1(\alpha)=\frac{4\sin\bigl(\frac{\pi}{3}+\alpha\bigr)+2}{\sqrt3}.
\]In our setting $0\le \alpha\le \pi/6$, and the admissible configurations satisfy $\sin(\frac{\pi}{3}+\alpha)>0$; thus the ``$-$'' choice would force
$z_1(\alpha)>2/\sqrt3$, contradicting $z_1\le 2/\sqrt3$.
Therefore the sign in~\eqref{eq:pm-relation} must be ``$+2$'', and we obtain
\begin{equation}\label{eq:z1-of-alpha}
z_1(\alpha)=\frac{2\sin\bigl(\frac{\pi}{3}+\alpha\bigr)-1}{\sin(\pi/3)}
=\frac{4\sin\bigl(\frac{\pi}{3}+\alpha\bigr)-2}{\sqrt3}.
\end{equation}

A visualization is shown in (b) of \cref{fig:n=12}.
Hence, looking at \(z_1\) as a function of \(\alpha\),
\[
\frac{\mathrm{d} z_1}{\mathrm{d} \alpha}= \frac{4 \cos \left(\frac{\pi}{3} +\alpha\right) }{\sqrt{3}}.
\]

Since $0<\alpha\le \pi/6$, we have $\cos\alpha\ge \sqrt3/2$ and $0<z_1\le 2/\sqrt3$. Writing $z_2=z_1-2e^{i\alpha}=x+iy$ with $x=z_1-2\cos\alpha$ and $y=-2\sin\alpha<0$, we obtain
$x\le 2/\sqrt3-\sqrt3=-1/\sqrt3$, hence $x^2\ge 1/3$, and also $y^2\le 1$.
Therefore
\[
\Im(z_2^3)=y(3x^2-y^2)\le 0,
\qquad\text{so}\qquad
|z_2^3-\overline{z_2}^3|=i\,(z_2^3-\overline{z_2}^3).
\]

Evaluating \cref{eq:DeltaDihedral} for \(m=2\) yields:
\begin{align*}
\frac{\Delta^{\frac16}}{3^2} =& |z_0||z_1||z_0^3-z_1^3| |z_2|^{2}  |z_0^3 - z_2^3|^2 |z_1^3 -z_2^3|^2 |z_2^3- \overline{z_2}^3| \\
=& (-z_0) z_1 (z_1^3-z_0^3) z_2 \overline{z_2} (i)(z_2^3- \overline{z_2}^3)(z_0^3 - z_2^3)(z_0^3 - \overline{z_2}^3) (z_1^3 -z_2^3)(z_1^3 -\overline{z_2}^3)\\
=& (2-z_1) z_1 (6z_1^2-12z_1+8) (z_1^2-4z_1 \cos(\alpha)+4) 4(3z_1^2 \sin(\alpha) -6 z_1 \sin(2\alpha)  +4\sin(3\alpha)  )\\
&((z_1-2)^3 - (z_1-2 e^{i\alpha})^3)((z_1-2)^3 - (z_1-2 e^{-i\alpha})^3)\\
&(z_1^3 -(z_1-2 e^{i\alpha})^3)(z_1^3 -(z_1-2 e^{-i\alpha})^3).
\end{align*}
Using Maple \cite[\texttt{n=12-calculation.mw}]{githubRepo}, we differentiate this expression with respect to $\alpha$ and set the derivative to zero. Let us use \(c=\cos(\alpha)\in[\sqrt3/2,1)\), one finds
\begin{multline}
    ((-61440\sqrt{3} - 61440)c^{13} + (598016\sqrt{3} + 942080)c^{12} + (-1358848\sqrt{3} - 2945024)c^{11} \\
    + (-1702912\sqrt{3} - 2994176)c^{10} + (12447488\sqrt{3} + 23366400)c^9 + (-14598912\sqrt{3} - 23974400)c^8 \\
    + (-14583936\sqrt{3} - 27066368)c^7 + (41161600\sqrt{3} + 68952832)c^6 + (-14880848\sqrt{3} - 25631840)c^5\\
    + (-26352512\sqrt{3} - 44362112)c^4 + (24621148\sqrt{3} + 43141088)c^3 + (-1331604\sqrt{3} - 2403864)c^2 \\
    + (-5793933\sqrt{3} - 10124442)c + 1834665\sqrt{3} + 3164778)\sin(\alpha)\\
    + (61440\sqrt{3} + 61440)c^{14} + (32768\sqrt{3} - 139264)c^{13} + (-1781760\sqrt{3} - 2925568)c^{12}\\
    + (5654528\sqrt{3} + 11060224)c^{11} + (-2075392\sqrt{3} - 3122432)c^{10} + (-20319744\sqrt{3} - 37831424)c^9\\
    + (31883136\sqrt{3} + 52888576)c^8 + (9866240\sqrt{3} + 19010304)c^7 + (-54395584\sqrt{3} - 91238480)c^6 \\
    + (26467744\sqrt{3} + 45939072)c^5 + (25917056\sqrt{3} + 43554644)c^4 + (-27518584\sqrt{3} - 48203492)c^3 \\
    + (2249010\sqrt{3} + 3986631)c^2 + (5793930\sqrt{3} + 10124325)c - 3164769 - 1834665\sqrt{3}=0.
\end{multline}
Using the fact that \(\sin(\alpha)^2=1-c^2\), we can get the following degree-$26$ polynomial equation for \(c\):
\begin{multline}
    3774873600 c^{26}+\left(-12079595520 \sqrt{3}-18874368000\right) c^{25}+\left(57780731904 \sqrt{3}+72746008576\right) c^{24}+\\\left(-100931731456 \sqrt{3}-186466172928\right) c^{23}+\left(31461474304 \sqrt{3}+143637086208\right) c^{22}+\\\left(209656479744 \sqrt{3}+389073076224\right) c^{21}+\left(-458848468992 \sqrt{3}-967476510720\right) c^{20}+\\\left(293161140224 \sqrt{3}+484148248576\right) c^{19}+\left(422923927552 \sqrt{3}+953012977664\right) c^{18}+\\\left(-912719085568 \sqrt{3}-1588930084864\right) c^{17}+\left(380269395968 \sqrt{3}+464902488064\right) c^{16}+\\\left(573757374464 \sqrt{3}+1038482817024\right) c^{15}+\left(-740008067072 \sqrt{3}-1163467247616\right) c^{14}+\\\left(118737051648 \sqrt{3}+151783239680\right) c^{13}+\left(335886705664 \sqrt{3}+532904293376\right) c^{12}+\\\left(-241433358336 \sqrt{3}-381667765248\right) c^{11}+\left(-3581119488 \sqrt{3}+6257104128\right) c^{10}\\
    +\left(76148838656 \sqrt{3}+116278888704\right) c^{9}+\left(-32570889408 \sqrt{3}-57731247296\right) c^{8}\\
    +\left(-1798975744 \sqrt{3}+1147734336\right) c^{7}+\left(6487432128 \sqrt{3}+11016204480\right) c^{6}\\+\left(-2864017968 \sqrt{3}-5670981840\right) c^{5}+\left(457206888 \sqrt{3}+880893600\right) c^{4}\\+\left(251077056 \sqrt{3}+498830580\right) c^{3}+\left(-183696321 \sqrt{3}-327921552\right) c^{2}\\+\left(43182972 \sqrt{3}+72512874\right) c-4953312-3027339 \sqrt{3}=0.
\end{multline}
Numerically, this yields
\[
c=0.9659364725201318915\ldots,\qquad
\alpha=0.26175825\ldots=14.99764302\ldots^\circ,
\]
and
\[
\frac{\Delta}{12^{12}}\approx 1.2901383629057280854\ldots.
\]

\subsection{Values obtained for more small values}
\begin{center}
\begin{longtable}{rrrr}
% \centering
% \begin{tabular}{rrrr}
\toprule
$n$ & $\log(\Delta_{\text{max found}}(n))$ & $\overline{\Delta}_{\text{max found}}(n)$ & $\overline{\Delta}(P_{\text{Section 4}})$  \\
\midrule
4 & 5.683852 & 1.148748 &  \\
6 & 11.021240 & 1.310854 & 1.310854 \\
8 & 16.859060 & 1.250472 &  \\
10 & 23.250608 & 1.252004 &  \\
12 & 30.073629 & 1.290138 & 1.290138 \\
14 & 37.182717 & 1.266036 &  \\
16 & 44.597546 & 1.266296 &  \\
18 & 52.276202 & 1.283347 & 1.283184 \\
20 & 60.154953 & 1.271579 &  \\
22 & 68.243701 & 1.272150 &  \\
24 & 76.521876 & 1.282119 & 1.281941 \\
26 & 84.953812 & 1.275350 &  \\
28 & 93.545549 & 1.275996 &  \\
30 & 102.284943 & 1.282629 & 1.282470 \\
32 & 111.149205 & 1.278300 &  \\
34 & 120.142414 & 1.278918 &  \\
36 & 129.256574 & 1.283683 & 1.283547 \\
38 & 138.475863 & 1.280706 &  \\
40 & 147.803223 & 1.281265 &  \\
42 & 157.232997 & 1.284867 & 1.284753 \\
44 & 166.753557 & 1.282710 &  \\
46 & 176.367127 & 1.283206 &  \\
48 & 186.069494 & 1.286031 & 1.285934 \\
50 & 195.851764 & 1.284400 &  \\
52 & 205.715643 & 1.284841 &  \\
54 & 215.657906 & 1.287120 & 1.287038 \\
56 & 225.671505 & 1.285852 &  \\
58 & 235.757835 & 1.286239 &  \\
60 & 245.914305 & 1.288116 & 1.288048 \\
62 & 256.135206 & 1.287106 &  \\
64 & 266.421697 & 1.287447 &  \\
66 & 276.771641 & 1.289026 & 1.288966 \\
68 & 287.180354 & 1.288199 &  \\
\bottomrule
% \end{tabular}
\caption{$\overline{\Delta}_{\text{max found}}(n)$ and $\overline{\Delta}_{\text{max found}}(n)$ are respectively lower bounds for $\log(\Delta_{\mathrm{max}}(n))$ and $\overline{\Delta}_{\mathrm{max}}(n)$, given here for various even $n$. When \(n\) is a multiple of six, we also include the estimate from the construction of \cref{sec:diam_graph_6k} as $\overline{\Delta}(P_{\text{Section 4}})$.}\label{tab:lowerboundsDelta}
\end{longtable}
\end{center}

The constructions for $n \in \{8,10,12\}$ were already hard to describe, and so do the larger constructions. 
Below, we give a table,~\cref{tab:lowerboundsDelta}, with the best lower bounds for $\Delta_{\mathrm{max}}$ and $\oD_{\mathrm{max}}$ obtained for some more values of even $n$ (which we think are near the true values).
For $6 \mid n$, we also compare these with the approximate construction from~\cref{sec:diam_graph_6k}.

\section{A construction for $6 \mid n$}\label{sec:diam_graph_6k}

In this section, we prove~\cref{thr:6mult}.

We first present the construction of our polygon $Y$ for $n=6k$.
Let $X = A_1 A_2 \ldots A_{6k}$ be a regular $n$-gon with unit diameter. 
Since $n$ is even, the diameter corresponds to the distance between opposite vertices. Thus, $X$ is inscribed in a circle of radius $1/2$, which implies its edge length is $\ell = |A_1 A_2| = \sin(\frac{\pi}{n})$. Let $\alpha = \frac{(n-2)\pi}{n}$ be its internal angle. 

The polygon $Y = B_1 B_2 \ldots B_{6k}$ is constructed as an equilateral polygon with a side length $\ell$, formed by concatenating six congruent arcs derived from $X$. Specifically, for each $j \in \{0, \ldots, 5\}$, the sequence of vertices $B_{jk} \ldots B_{(j+1)k}$ is congruent to the arc $A_{jk} \ldots A_{(j+1)k}$ of the regular polygon. Throughout, indices are taken modulo $n=6k$, and we set $A_0=A_n$, $B_0=B_n$. Consequently, all internal angles at $B_i$ are equal to $\alpha$, except at the ``junction'' vertices $B_{k}, B_{2k}, \ldots, B_{6k}$. While the polygon $B_k B_{2k} B_{3k} B_{4k} B_{5k} B_{6k}$ is a hexagon with six equal sides, 
$\angle B_{6k}B_k B_{2k} = \angle B_{2k} B_{3k} B_{4k} = \angle B_{4k} B_{5k} B_{6k} = \frac{2\pi}{3} + \frac{\pi}{n}$ 
and $\angle B_{k} B_{2k} B_{3k} = \angle B_{3k} B_{4k} B_{5k} = \angle B_{5k} B_{6k} B_k = \frac{2\pi}{3} - \frac{\pi}{n}$.
Therefore, at $B_{rk}$ with $r\in\{1,3,5\}$, we have that $\angle B_{rk-1}B_{rk}B_{rk+1}=\alpha+\pi/n$,
and at $B_{rk}$ with $r\in\{2,4,6\}$, we have that $\angle B_{rk-1}B_{rk}B_{rk+1}=\alpha-\pi/n$. 

\begin{figure}[H]
    \centering
    \includegraphics[width=0.5\linewidth]{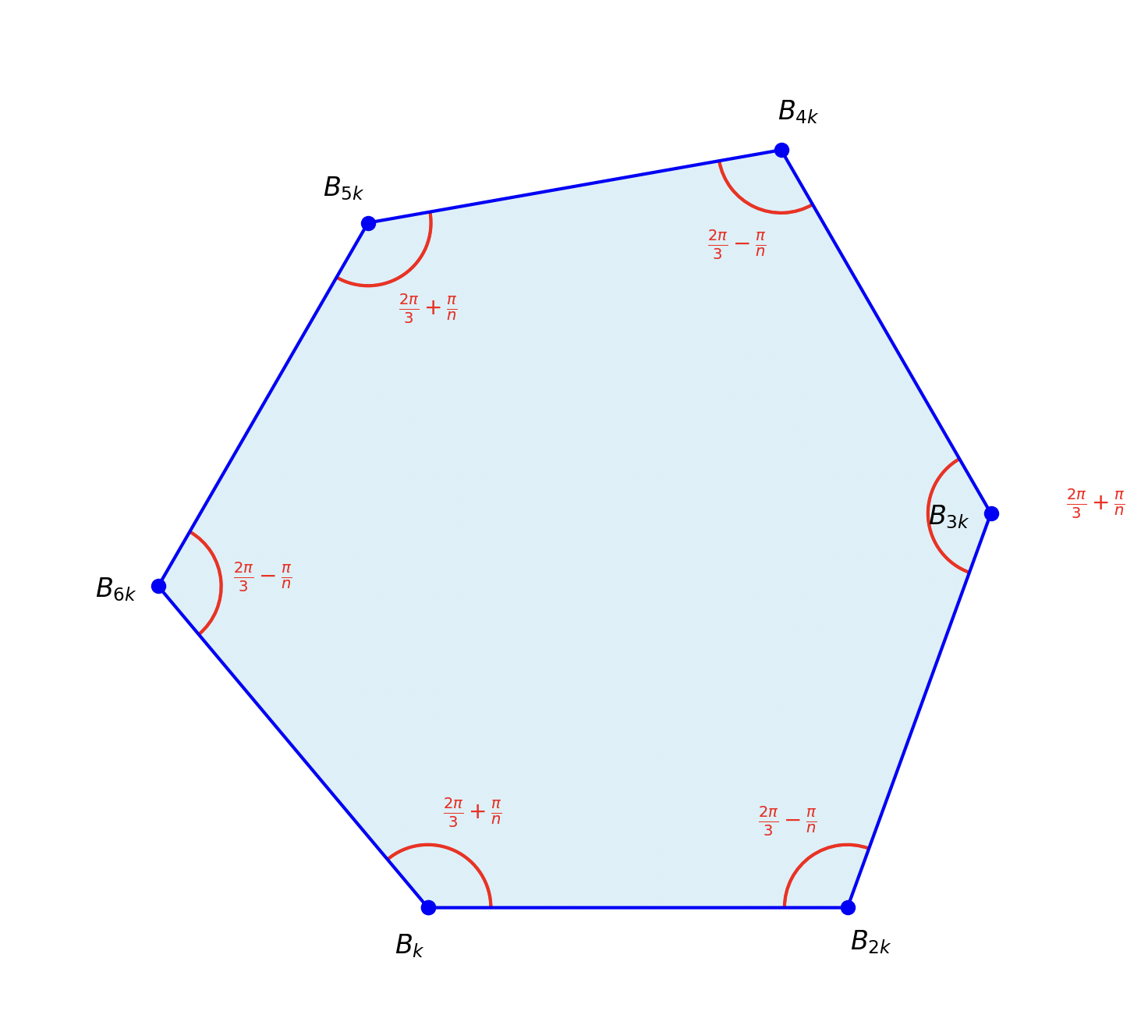}
    \caption{The shape of the polygon $B_k B_{2k} B_{3k} B_{4k} B_{5k} B_{6k}$.}
    \label{fig:placeholder}
\end{figure}

Finally, we obtain our polygon $P$ by rescaling $Y$ by the factor $2/\cos(\pi/2n)$ (so that $\diam(P)=2$).
\begin{lem}\label{lem:diamY}
    $\diam(Y)=\cos(\pi/2n).$
\end{lem}

\begin{proof}
Notice that $\abs{B_{ik}B_{(i+1)k}}=\abs{A_{ik}A_{(i+1)k}}=\frac{1}{2}$ for every $1 \le i \le 6,$ since $A_kA_{2k}A_{3k}\ldots A_{6k}$ is a regular hexagon with diameter $1$ and the arcs $B_{ik}B_{ik+1} \ldots B_{(i+1)k}$ are isometric to $A_{ik}A_{ik+1} \ldots A_{(i+1)k}$ for every $i.$
Since $\angle B_kB_{2k}B_{3k}=\frac{2\pi}{3}-\frac{\pi}{n}$ and $\angle B_{2k}B_{3k} B_{4k}=\frac{2\pi}{3}+\frac{\pi}{n}$, and
$$\left( \cos\left(\frac{\pi}{3}+\frac{\pi}{n} \right)+\cos\left(\frac{\pi}{3}-\frac{\pi}{n} \right)+1\right)^2+ \left( \sin\left(\frac{\pi}{3}+\frac{\pi}{n} \right)-\sin\left(\frac{\pi}{3}-\frac{\pi}{n} \right)\right)^2=4\cos^2\left( \frac{\pi}{2n} \right),$$
we conclude that $\abs{B_kB_{4k}}=\cos(\pi/2n)$.
Hence $\diam(Y)\geq \cos(\pi/2n)$.

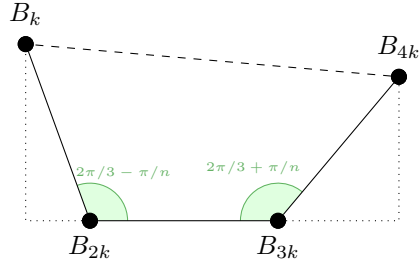
\begin{figure}[h]
    \centering
\begin{tikzpicture}[scale=1.25, main_node/.style={fill,draw,minimum size=0.1,circle,inner sep=2pt]}]

    \node[main_node] (1) at (-0.684,1.8794) {};
    \node[main_node] (2) at (0.0, 0.0) {};
    \node[main_node] (3) at (2.0, 0.0) {};
    \node[main_node] (4) at (3.2856, 1.532) {};

     \tikzset{
            angle style1/.style={draw=green!50!black, fill=green!20, opacity=0.6,angle eccentricity=1.55},
            label style1/.style={font=\tiny, text=green!50!black}
        }

     % \pic [angle style1, "$\pi/3-\pi/n$"{label style1}] {angle = (3.2856,1.532)--(2,0)--(3,0)};

     \pic [
    angle style1,
    "$2\pi/3+\pi/n$"{label style1}
] {angle = 4--3--2};

\pic [
    angle style1,
    "$2\pi/3-\pi/n$"{label style1}
] {angle = 3--2--1};

    \node[main_node] (1) at (-0.684,1.8794) {};
    \node[main_node] (2) at (0.0, 0.0) {};
    \node[main_node] (3) at (2.0, 0.0) {};
    \node[main_node] (4) at (3.2856, 1.532) {};

    \path[draw] (1)--(2)--(3)--(4);

    \path[draw, dotted] (2)--(-0.684,0)--(1);
    \path[draw, dotted] (3)--(3.2856,0)--(4);

     \path[draw, dashed] (4)--(1);

     \node[black!80, font=\tiny, label={$B_{k}$}] at (1) {};
     \node[black!80, font=\tiny, label={$B_{2k}$}] at (0,-0.6) {};
     \node[black!80, font=\tiny, label={$B_{3k}$}] at (2,-0.6) {};
     \node[black!80, font=\tiny, label={$B_{4k}$}] at (4) {};

\end{tikzpicture}
\caption{Quadrilateral $B_kB_{2k}B_{3k}B_{4k}$ and projections to compute $\abs{B_k B_{4k}}^2$}\label{fig:det_diam_Y}
\end{figure}

It remains to prove that $\diam(Y)\leq \cos(\pi/2n)$.

We sketch first one way to conclude so.
Set $\delta:=\pi/n$.
In triangle $B_k B_{4k} B_{k \pm 1}$, $\lvert B_k B_{k \pm 1} \rvert = \sin \delta, \lvert B_k B_{4k} \rvert= \cos( \delta /2)= \sin( \pi/2- \delta /2)$ and $\angle B_{4k} B_k B_{k \pm 1}=  \pi/2- \delta /2. $
By the sine rule, we conclude that $\abs{B_{k\pm 1}B_{4k}}=\abs{B_kB_{4k}}.$
Triangle $\triangle B_{k+1} B_{4k} B_{4k-1}$ is isomorphic to $\triangle B_k B_{4k} B_{k \pm 1}$ (by Side-Angle-Side (SAS)), and thus also $\lvert B_{k+1} B_{4k-1} \rvert= \cos( \delta /2)$.
This can be repeated, to notice that $n$ distances are equal to $\cos(\pi/2n).$
Among those, there are all line segments $B_i B_{i+3k}$ (for every $1 \le i \le 6k$).

Let $P$ be a regular $2n$-gon with center $B_i$ and radius $\cos( \delta /2)$, with one vertex equal to $B_{i+3k}$.
Then all side lengths of $Y$ and $P$ are equal to $\sin \delta.$
All internal angles of $P$ equal $\pi-\pi/n$, while the internal angles of $Y$ are bounded by $\pi-\pi/n$. From this, one can conclude that $Y$ lies completely within $P$.

Since the above is true for every $1 \le i \le 6k$, and the diameter of a convex polygon can be found among the distances between the vertices, we conclude that $\diam Y= \cos(\pi/2n)$.

A detailed proof is deferred to Appendix~\ref{app:diamY}.
\end{proof}

\begin{lem}\label{lem:2ndregime}
    The ratio $\prod_{ 1 \le i,j \le k } \frac{\abs{B_{k-i}B_{k+j}}^2\abs{B_{2k-i}B_{2k+j}}^2}{\abs{A_{k-i}A_{k+j}}^2\abs{A_{2k-i}A_{2k+j}}^2}$ converges to $$C_1=\exp\left(\frac{1}{4}-\frac{\pi\sqrt3}{24}-\frac{\ln (3)}{8}\right)$$ as $k \to \infty.$
\end{lem}

\begin{proof}
Assume throughout this lemma that $n=6k$ and set $\delta:=\pi/n$.
For $1\le i,j\le k$, define
\[
X_{i,j}:=1-\frac{|B_{k-i}B_{k+j}|^2\,|B_{2k-i}B_{2k+j}|^2}{|A_{k-i}A_{k+j}|^2\,|A_{2k-i}A_{2k+j}|^2}.
\]
In this regime we have $|A_{2k-i}A_{2k+j}|=|A_{k-i}A_{k+j}|$, hence
\[
1-X_{i,j}
=\frac{|B_{k-i}B_{k+j}|^2\,|B_{2k-i}B_{2k+j}|^2}{|A_{k-i}A_{k+j}|^4},
\qquad
P_k:=\prod_{1\le i,j\le k}(1-X_{i,j})
\]
is exactly the product appearing in the statement.

Write $x_i:=\pi i/n$ and $y_j:=\pi j/n$. Since $k=n/6$, we have
$x_i,y_j\in(0,\pi/6]$ and $x_i+y_j\in(0,\pi/3]$.
Moreover, $|A_kA_{k-i}|=|B_kB_{k-i}|=\sin(x_i)$ and similarly $|A_kA_{k+j}|=|B_kB_{k+j}|=\sin(y_j)$.

Consider the triangle $\triangle A_{k-i}A_kA_{k+j}$. The angle at $A_k$ equals
$\pi-(x_i+y_j)$, so by the cosine law and $\cos(\pi-\theta)=-\cos\theta$,
\[
|A_{k-i}A_{k+j}|^2
=\sin^2(x_i)+\sin^2(y_j)-2\sin(x_i)\sin(y_j)\cos\bigl(\pi-(x_i+y_j)\bigr)
=\sin^2(x_i+y_j).
\]
For the corresponding $B$-triangles, the same angle is perturbed by $\mp\delta$ at the vertices
$B_k$ and $B_{2k}$, hence
\begin{align*}
|B_{k-i}B_{k+j}|^2
&=\sin^2(x_i)+\sin^2(y_j)+2\sin(x_i)\sin(y_j)\cos\bigl((x_i+y_j)-\delta\bigr),\\
|B_{2k-i}B_{2k+j}|^2
&=\sin^2(x_i)+\sin^2(y_j)+2\sin(x_i)\sin(y_j)\cos\bigl((x_i+y_j)+\delta\bigr).
\end{align*}

A direct algebraic expansion using
\[
\cos(\alpha-\delta)+\cos(\alpha+\delta)=2\cos\alpha\cos\delta,
\qquad
\cos(\alpha-\delta)\cos(\alpha+\delta)=\cos^2\alpha-\sin^2\delta
\]
shows that
\begin{align*}
|A_{k-i}A_{k+j}|^4-|B_{k-i}B_{k+j}|^2\,|B_{2k-i}B_{2k+j}|^2
&=
4(\sin^2 x_i+\sin^2 y_j)\sin x_i\sin y_j\cos(x_i+y_j)\,(1-\cos\delta)\\
&\quad
+2\sin^2 x_i\sin^2 y_j\,(1-\cos 2\delta).
\end{align*}
Dividing by $|A_{k-i}A_{k+j}|^4=\sin^4(x_i+y_j)$ gives the exact formula
\[
X_{i,j}
=
\frac{4(\sin^2 x_i+\sin^2 y_j)\sin x_i\sin y_j\cos(x_i+y_j)\,(1-\cos\delta)
+2\sin^2 x_i\sin^2 y_j\,(1-\cos 2\delta)}{\sin^4(x_i+y_j)}.
\]

Now we expand in $\delta$. By Taylor's theorem with remainder, there is an absolute constant $C>0$
such that for all sufficiently small $\delta$,
\[
\Bigl|1-\cos\delta-\frac{\delta^2}{2}\Bigr|\le C\delta^4,
\qquad
\Bigl|1-\cos(2\delta)-2\delta^2\Bigr|\le C\delta^4.
\]
Substituting these into the exact expression yields
\[
X_{i,j}=\delta^2\,H(x_i,y_j)+\delta^4\,E(x_i,y_j;\delta),
\]
where
\[
H(x,y)
=
\frac{2(\sin^2 x+\sin^2 y)\sin x\sin y\cos(x+y)+4\sin^2 x\sin^2 y}{\sin^4(x+y)},
\]
and $E(\cdot,\cdot;\delta)$ is a function arising from the Taylor remainders.

We claim that $H$ is bounded on $(0,\pi/6]^2$ and that $E$ is uniformly bounded there for all
sufficiently small $\delta$. Indeed, for $t\in[0,\pi/3]$ one has
$\sin t\ge \frac{2}{\pi}t$ and $\sin t\le t$, $|\cos t|\le 1$. Hence for $x,y\in(0,\pi/6]$,
\[
\sin^4(x+y)\ge \Bigl(\frac{2}{\pi}(x+y)\Bigr)^4,
\]
while the numerator of $H$ is bounded in absolute value by a constant multiple of
$(x^2+y^2)xy+x^2y^2\le (x+y)^4$. This gives $|H(x,y)|\le M_1$ for some absolute $M_1$.
The same estimate applies to the coefficients multiplying the Taylor remainder terms, so
$|E(x,y;\delta)|\le M_2$ for some absolute $M_2$ independent of $i,j,k$ (for $k$ large).
Consequently,
\[
\max_{1\le i,j\le k}|X_{i,j}|
\le \delta^2 M_1+\delta^4 M_2
=O(\delta^2)\xrightarrow[k\to\infty]{}0.
\]

Summing the expansion gives
\[
\sum_{1\le i,j\le k}X_{i,j}
=
\delta^2\sum_{1\le i,j\le k}H(x_i,y_j)
+\delta^4\sum_{1\le i,j\le k}E(x_i,y_j;\delta).
\]
The second term is $O(k^2\delta^4)=O(k^2/n^4)=O(1/k^2)\to 0$ since $n=6k$.

For the first term, $\delta^2\sum_{1\le i,j\le k}H(x_i,y_j)$ is the (right-endpoint) Riemann sum on
$[\delta,\pi/6]^2$ with mesh size $\delta$.
Since $H$ is bounded on $(0,\pi/6]^2$, the missing strip
$[0,\pi/6]^2\setminus[\delta,\pi/6]^2$ has area $O(\delta)$ and hence contributes at most $O(\delta)$ to
$\iint_{[0,\pi/6]^2}H$.
Therefore
\[
\lim_{k\to\infty}\delta^2\sum_{1\le i,j\le k}H(x_i,y_j)
=\int_0^{\pi/6}\int_0^{\pi/6}H(x,y)\,dx\,dy.
\]
Consequently,
\[
\lim_{k\to\infty}\ \sum_{1\le i,j\le k}X_{i,j}
=
\int_0^{\pi/6}\int_0^{\pi/6}H(x,y)\,dx\,dy.
\]

By the computation in Section~\ref{subsec:integral_v1} or~\cite[\texttt{EP1045_1stregime}]{githubRepo}, the integral equals
\[
-\frac14+\frac{\pi\sqrt3}{24}+\frac{\ln 3}{8}=-\ln C_1.
\]

Finally we pass from sums to the product $P_k=\prod_{1\le i,j\le k}(1-X_{i,j})$.
Since $\max_{i,j}|X_{i,j}|\to 0$, for all sufficiently large $k$ we have $|X_{i,j}|\le \frac12$
for every $i,j$. On $[-\frac12,\frac12]$ we have \(|\ln(1-u)+u|\le u^2\). Therefore,
\[
\ln P_k=\sum_{i,j}\ln(1-X_{i,j})
=-\sum_{i,j}X_{i,j}+O\!\left(\sum_{i,j}X_{i,j}^2\right).
\]
Using the uniform bound $X_{i,j}=O(\delta^2)$, we get
\[
\sum_{i,j}X_{i,j}^2\le k^2\cdot O(\delta^4)=O(k^2\delta^4)=O(1/k^2)\to 0,
\]
and hence
\[
\ln P_k=-\sum_{i,j}X_{i,j}+o(1)\xrightarrow[k\to\infty]{}\ln C_1.
\]
Exponentiating yields $P_k\to C_1$, as claimed.
\end{proof}

The following two lemmas can be proven analogously. For the sake of completeness, the details of the computations in Maple are also presented in~\cref{app:AnIntegral}.

\begin{lem}\label{lem:3rdregime}
    The ratio $\prod_{ 1 \le i,j \le k } \frac{\abs{B_{k-i}B_{2k+j}}^2\abs{B_{2k-i}B_{3k+j}}^2}{\abs{A_{k-i}A_{2k+j}}^2\abs{A_{2k-i}A_{3k+j}}^2}$ converges to \[C_2=\exp\left(-\frac{1}{4}-\frac{\ln(2)}{2}-\frac{\pi\sqrt3}{24}+\frac{5 \ln (3)}{8}\right)\] as $k \to \infty.$
\end{lem}

\begin{proof}
    Analogous to the proof of Lemma~\ref{lem:2ndregime}, up to the formulas that have to be updated. See Section~\ref{subsec:integral_v2} or~\cite[\texttt{EP1045_2ndregime}]{githubRepo}.
    In this case, for $x_i:=\pi i/n$ and $y_j:=\pi j/n$, we have
    \begin{align*}\abs{B_{k-i}B_{2k+j}}^2
    &=
    \left(
    \frac12
    +\cos \left(\frac{\pi}{6}+x_i-\delta\right)\sin x_i
    +\cos \left(\frac{\pi}{6}+y_j+\delta\right)\sin y_j
    \right)^2 \\
    &\quad+
    \left(
    \sin \left(\frac{\pi}{6}+x_i-\delta\right)\sin x_i
    -\sin \left(\frac{\pi}{6}+y_j+\delta\right)\sin y_j
    \right)^2.\\
    \abs{B_{2k-i}B_{3k+j}}^2
    &=
    \left(
    \frac12
    +\cos \left(\frac{\pi}{6}+x_i+\delta\right)\sin x_i
    +\cos \left(\frac{\pi}{6}+y_j-\delta\right)\sin y_j
    \right)^2 \\
    &\quad+
    \left(
    \sin \left(\frac{\pi}{6}+x_i+\delta\right)\sin x_i
    -\sin \left(\frac{\pi}{6}+y_j-\delta\right)\sin y_j
    \right)^2.
    \qedhere\end{align*}
\end{proof}

\begin{lem}\label{lem:4thregime}
    The ratio $\prod_{ 1 \le i,j \le k } \frac{\abs{B_{k-i}B_{4k-j}}^4}{\abs{A_{k-i}A_{4k-j}}^4}$ converges to $C_3=\frac{\sqrt 3}{2}$ as $k \to \infty.$
\end{lem}

\begin{proof}
    Analogous to the proof of Lemma~\ref{lem:2ndregime}, up to the formulas that have to be updated. See Section~\ref{subsec:integral_v3} or~\cite[\texttt{EP1045_3rdregime}]{githubRepo}.
    In this case, for $x_i:=\pi i/n$ and $y_j:=\pi j/n$, we have
    \begin{align*}
    \abs{B_{k-i}B_{4k-j}}^2&=\cos^2 \left(x_i-y_j-\frac{\delta}{2}\right).  \qedhere\end{align*}
\end{proof}

Finally, recall that $X=A_1\cdots A_n$ denotes the regular $n$-gon of unit diameter,
$Y=B_1\cdots B_n$ is the equilateral polygon constructed from $X$ by modifying the junction angles,
and $P$ is the rescaling of $Y$ by the factor $2/\cos(\pi/2n)$ so that $\diam(P)=2$ (as ensured by Lemma~\ref{lem:diamY}).
Using $\oD(P)=\Delta(P)/n^n$, we write
\[
\oD(P)=\frac{\Delta(P)}{\Delta(Y)}\cdot\frac{\Delta(Y)}{\Delta(X)}\cdot\frac{\Delta(X)}{n^n}.
\]
Since $P=\frac{2}{\cos(\pi/2n)}\,Y$, we have
$\Delta(P)/\Delta(Y)=\bigl(\frac{2}{\cos(\pi/2n)}\bigr)^{n(n-1)}$.
Moreover, scaling the regular unit-diameter configuration $X$ by a factor $2$ yields a regular
$n$-gon of diameter $2$, whose discriminant equals $n^n$ (for even $n$); hence
$\Delta(X)=2^{-n(n-1)}n^n$ and therefore
\[
\frac{\Delta(P)}{\Delta(Y)}\cdot\frac{\Delta(X)}{n^n}
=\frac{1}{\cos(\pi/2n)^{n(n-1)}}
\sim \exp(\pi^2/8).
\]
Together with Lemmas~\ref{lem:2ndregime}--\ref{lem:4thregime}, which give
$\Delta(Y)/\Delta(X)\to C_1^3C_2^3C_3^{3/2}$, we conclude that
\[
\oD(P)\to \exp(\pi^2/8)\,C_1^3C_2^3C_3^{3/2}
=\frac{3^{9/4}}{2^3}\exp\!\left(\frac{\pi^2-2\sqrt3\,\pi}{8}\right),
\]
proving~\cref{thr:6mult}.

Regarding the last remark, when $n$ is even we may start from the above construction
$P'=V_1V_2\cdots V_{3n}$ with $3n$ vertices and then take every third vertex to obtain an
$n$-vertex polygon \(\widehat{P}=V_3V_6V_9\cdots V_{3n}\). The diameter graph of $\widehat{P}$ is disconnected, so this construction is not expected to be optimal.
Nevertheless, it yields \(\oD(\widehat{P})\longrightarrow C_*^{1/9}\). Indeed, the analysis is entirely analogous to the one above. The only changes are that passing from
$P'$ to $\widehat{P}$ replaces the angular mesh size by $\delta/3$, and the final normalization to diameter
$2$ uses the scaling factor $\frac{2}{\cos(\pi/6n)}$ in place of $\frac{2}{\cos(\pi/2n)}$.
Since the leading term in $\log \oD(\cdot)$ is governed by a double sum (equivalently, a double integral),
these modifications rescale the limiting constant in the exponent by a factor $1/9$, which explains the
$9$th-root relation.

\section{A uniform lower bound for even $n$}\label{sec:unif_lower_bnd}
In this section we prove Theorem~\ref{thm:main_unifom_even_n_1}.
For each sufficiently large even integer $n$ (i.e. \(n\geq8\)) we construct an explicit configuration of diameter~$2$
such that the normalized discriminant $\oD=\Delta/n^n$ converges to a constant strictly larger than~$1$ as $n\to\infty$.
The construction is a small radial perturbation of the regular $n$--gon.
It is designed so that antipodal pairs remain at distance~$2$,
while a $\pi$--antiperiodic symmetry ensures that the first-order term in the expansion of $\log \oD$ vanishes for even~$n$.

\subsection{Construction and diameter control}
\subsubsection{The triangular wave.}
Let $\mathrm{tri}:\mathbb R\to\mathbb R$ be the $2\pi$--periodic extension of
\[
x\mapsto 1-\frac{2}{\pi}|x|\qquad (x\in[-\pi,\pi]).
\]
Equivalently, $\mathrm{tri}(x)=1-\frac{2}{\pi}\arccos(\cos x)$.
We set \(g(\theta)=\mathrm{tri}(3\theta)\). For $x,y\in\mathbb R$ we write the circular distance
\[
d(x,y)=\min_{k\in\mathbb Z}|x-y-2\pi k|\in[0,\pi].
\]

\begin{lem}\label{lem:gprops}
The function $g(\theta)=\mathrm{tri}(3\theta)$ satisfies:
\begin{enumerate}
\item[(1)] $|g(\theta)|\le 1$ for all $\theta\in\mathbb R$;
\item[(2)] $g$ is even: $g(-\theta)=g(\theta)$;
\item[(3)] $g$ is $\pi$--antiperiodic: $g(\theta+\pi)=-g(\theta)$;
\item[(4)] $g$ is Lipschitz on the circle: for all $\theta,\phi\in\mathbb R$,
\[
|g(\theta)-g(\phi)|\le L\, d(\theta,\phi),
\qquad L=\frac{6}{\pi}.
\]
\end{enumerate}
\end{lem}

\begin{proof}
Items (1)--(3) are immediate from the definition of $\mathrm{tri}$.
For (4), note that $\mathrm{tri}$ is piecewise linear on $[-\pi,\pi]$ with slope bounded by $2/\pi$,
hence it is globally Lipschitz on $\mathbb R$ (with respect to $|\cdot|$) with constant $2/\pi$.
Therefore $g(\theta)=\mathrm{tri}(3\theta)$ is globally Lipschitz with constant $3\cdot(2/\pi)=6/\pi$:
\[
|g(u)-g(v)|\le \frac{6}{\pi}|u-v|\qquad (u,v\in\mathbb R).
\]
Given $\theta,\phi\in\mathbb R$, choose $k\in\mathbb Z$ with
$|\theta-(\phi+2\pi k)|=d(\theta,\phi)$. Using $2\pi$--periodicity of $g$,
\[
|g(\theta)-g(\phi)|
=
|g(\theta)-g(\phi+2\pi k)|
\le \frac{6}{\pi}\,|\theta-(\phi+2\pi k)|
= \frac{6}{\pi}\,d(\theta,\phi).
\qedhere
\]
\end{proof}

\subsubsection{The perturbed $n$-gon for even $n$}
Fix an even integer $n=2m$. Let
\[
\theta_k=\frac{2\pi k}{n}\qquad (k=0,1,\dots,n-1),\qquad \zeta_k=e^{i\theta_k}.
\]
Define the amplitude
\[
t_n=\frac{\pi^2}{12n}\left(1-\frac{1}{n}\right),
\]
and set
\[
z_k=\left(1+t_n g(\theta_k)\right) \zeta_k,\qquad k=0,1,\dots,n-1.
\]
Thus the points are a radial perturbation of the unit roots of unity, with size $t_n=O(1/n)$.

\begin{lem}\label{lem:diameter}
For every even $n\ge 8$, the configuration $\{z_k\}_{k=0}^{n-1}$ satisfies
\[
\max_{i,j}|z_i-z_j|\le 2.
\]
Moreover, for every $k$ we have $|z_k-z_{k+m}|=2$.
\end{lem}

\begin{proof}
Write $r_k=1+t_n g(\theta_k)$ so that $z_k=r_k e^{i\theta_k}$.

\medskip\noindent\underline{\emph{Antipodal pairs.}}
Since $\zeta_{k+m}=-\zeta_k$ and $g(\theta_{k+m})=-g(\theta_k)$ by Lemma~\ref{lem:gprops}(3),
\[
z_{k+m}=(1-t_n g(\theta_k)) e^{i(\theta_k+\pi)}=-(1-t_n g(\theta_k))e^{i\theta_k},
\]
hence $z_k-z_{k+m}=2e^{i\theta_k}$ and therefore $|z_k-z_{k+m}|=2$.

\medskip\noindent\underline{\emph{All other pairs.}}
Fix $i\neq j$ and $d(\theta_i,\theta_j)\in(0,\pi)$. Then
\[
|z_i-z_j|^2=r_i^2+r_j^2-2r_ir_j\cos(d(\theta_i,\theta_j)).
\]

If $d(\theta_i,\theta_j)\le \pi/2$ then $\cos(d(\theta_i,\theta_j))\ge0$, so by Lemma~\ref{lem:gprops}(1) we have $|z_i-z_j|^2\le r_i^2+r_j^2\le 2(1+t_n)^2<4$
(because $t_n\le \pi^2/24<\sqrt2-1$). Hence $|z_i-z_j|<2$.

Assume now $d(\theta_i,\theta_j)\in(\pi/2,\pi)$ and write $d(\theta_i,\theta_j)=\pi-\alpha$ with $\alpha\in(0,\pi/2)$.
Then $\cos(d(\theta_i,\theta_j))=-\cos\alpha$ and
\[
|z_i-z_j|^2=r_i^2+r_j^2+2r_ir_j\cos\alpha.
\]
There exists $\varepsilon\in\{\pm1\}$ such that
\[
\theta_j\equiv \theta_i+\pi-\varepsilon\alpha\pmod{2\pi},
\qquad \alpha=\frac{2\pi\ell}{n}
\quad\text{for some }\ell\in\Bigl\{1,\dots,\Bigl\lfloor\frac{n-1}{4}\Bigr\rfloor\Bigr\}.
\]
By Lemma~\ref{lem:gprops}(3),
$g(\theta_j)=g(\theta_i+\pi-\varepsilon\alpha)=-g(\theta_i-\varepsilon\alpha)$.
Set
\[
a=g(\theta_i),\qquad b=g(\theta_i-\varepsilon\alpha).
\]
Then $r_i=1+t_n a$ and $r_j=1-t_n b$. A direct expansion gives
\begin{align}
|z_i-z_j|^2-4
&= -4\sin^2\left(\frac{\alpha}{2}\right)
+2t_n(a-b)(1+\cos\alpha)
+t_n^2\Bigl[(a-b)^2+2(1-\cos\alpha)\,ab\Bigr].
\label{eq:star}
\end{align}
We bound the right-hand side from above.
By Lemma~\ref{lem:gprops}(4), $|a-b|\le L\alpha$ with $L=6/\pi$.
Also $|ab|\le1$ and $1-\cos\alpha\le \alpha^2/2$, hence
\[
(a-b)^2+2(1-\cos\alpha)ab\le (L^2+1)\alpha^2.
\]
Moreover, for all $\alpha\ge0$ one has $-4\sin^2(\alpha/2)\le -\alpha^2+\alpha^4/12$,
and $1+\cos\alpha\le2$ gives $2t_n(a-b)(1+\cos\alpha)\le 4t_nL\alpha$.
Substituting these estimates into \eqref{eq:star} yields
\begin{equation}\label{eq:hbound}
|z_i-z_j|^2-4
\le
-\alpha^2+\frac{\alpha^4}{12}+4t_nL\alpha+(L^2+1)t_n^2\alpha^2.
\end{equation}

Now substitute $\alpha=\frac{2\pi\ell}{n}$ and $t_n=\frac{\pi^2}{12n}(1-\frac1n)$.
Using $t_n\le \pi^2/(12n)$ and $L^2+1=(\pi^2+36)/\pi^2$, \eqref{eq:hbound} implies the explicit bound
\begin{equation}\label{eq:E-nl}
|z_i-z_j|^2-4
\le
-\frac{4\pi^2}{n^2}\Bigl(\ell(\ell-1)+\frac{\ell}{n}\Bigr)
+\frac{\pi^4}{n^4}\Bigl(\frac{4}{3}\ell^4+\frac{\pi^2+36}{36}\ell^2\Bigr)
=:E_{n,\ell}.
\end{equation}

We claim $E_{n,\ell}\le0$ for every even $n\ge8$ and $1\le \ell\le n/4$.
For $\ell=1$,
\[
E_{n,1}\le -\frac{4\pi^2}{n^3}
+\frac{\pi^4}{n^4}\Bigl(\frac{7}{3}+\frac{\pi^2}{36}\Bigr)\le 0
\qquad (n\ge 8).
\]
For $\ell\ge2$, use $\ell(\ell-1)\ge \ell^2/2$ and $\ell\le n/4$ in \eqref{eq:E-nl} to get
\[
E_{n,\ell}
\le
-\frac{2\pi^2\ell^2}{n^2}
+\frac{\pi^4}{n^4}\Bigl(\frac{4}{3}\ell^4+\frac{\pi^2+36}{36}\ell^2\Bigr)
\le
\frac{\ell^2}{n^4}\Bigl(n^2\Bigl(-2\pi^2+\frac{\pi^4}{12}\Bigr)+\frac{\pi^4(\pi^2+36)}{36}\Bigr).
\]
The bracket is negative for $n\ge 8$, hence $E_{n,\ell}\le0$.
Therefore $|z_i-z_j|^2\le 4$ in all cases, i.e.\ $|z_i-z_j|\le2$.
\end{proof}

\subsection{Factorization and second-order expansion}
We next compare $\Delta(z_0,\dots,z_{n-1})$ to the regular configuration $\{\zeta_k\}$.
For $i\neq j$ define
\[
\rho_{ij}:=\frac{g(\theta_i)\zeta_i-g(\theta_j)\zeta_j}{\zeta_i-\zeta_j}.
\]
Then
\[
z_i-z_j=(\zeta_i-\zeta_j)\bigl(1+t_n\rho_{ij}\bigr),
\]
and consequently
\begin{equation}\label{eq:factor}
\frac{\Delta(z_0,\dots,z_{n-1})}{\prod_{i\neq j}|\zeta_i-\zeta_j|}
=
\prod_{i\neq j}\bigl|1+t_n\rho_{ij}\bigr|.
\end{equation}

\begin{lem}\label{lem:rootsV}
For $\zeta_k=e^{2\pi i k/n}$ one has \(\prod_{i\neq j}|\zeta_i-\zeta_j|=n^n\).
\end{lem}

\begin{proof}
Let $p_n(z)=z^n-1=\prod_{j=0}^{n-1}(z-\zeta_j)$. For each root $\zeta_i$,
\(
p_n'(\zeta_i)=n\zeta_i^{n-1}=\prod_{j\neq i}(\zeta_i-\zeta_j).
\)
Taking absolute values and multiplying over $i$ gives
\[
\prod_{i}\prod_{j\neq i}|\zeta_i-\zeta_j|
=\prod_i |p'(\zeta_i)|
=\prod_i n = n^n.
\qedhere
\]
\end{proof}

Combining Lemma~\ref{lem:rootsV} with \eqref{eq:factor} yields
\begin{equation}\label{comb_lem_ro_fa_yi_1}
\frac{\Delta(z_0,\dots,z_{n-1})}{n^n}
=
\prod_{i\neq j}\bigl|1+t_n\rho_{ij}\bigr|.
\end{equation}

\begin{lem}\label{lem:rhoBound}
For all $i\neq j$, one has $|\rho_{ij}|\le 4$.
\end{lem}

\begin{proof}
Write $g_i=g(\theta_i)$. Then
\(g_i\zeta_i-g_j\zeta_j=g_i(\zeta_i-\zeta_j)+(g_i-g_j)\zeta_j\), hence
\[
\rho_{ij}=g_i+(g_i-g_j)\frac{\zeta_j}{\zeta_i-\zeta_j}.
\]
Since $|g_i|\le 1$ and
$|g_i-g_j|\le Ld(\theta_i,\theta_j)$ (Lemma~\ref{lem:gprops}(4)), while
$|\zeta_i-\zeta_j|=2\sin(d(\theta_i,\theta_j)/2)$, we obtain
\[
|\rho_{ij}|
\le 1+\frac{Ld(\theta_i,\theta_j)}{2\sin(d(\theta_i,\theta_j)/2)}.
\]
Since $\sin$ is concave on $[0,\pi/2]$, for $t\in[0,\pi/2]$ we have
$\sin t\ge 2t/\pi$; applying this to $t=d(\theta_i,\theta_j)/2$ gives
$\sin(d(\theta_i,\theta_j)/2)\ge d(\theta_i,\theta_j)/\pi$ and hence
\[
|\rho_{ij}|\le 1+\frac{L\pi}{2}=1+\frac{6}{\pi}\cdot\frac{\pi}{2}=4.
\qedhere
\]
\end{proof}

Taking logarithms in \eqref{comb_lem_ro_fa_yi_1} gives
\begin{equation}\label{eq:logsum}
\log\frac{\Delta(z_0,\dots,z_{n-1})}{n^n}
=\sum_{i\neq j}\log\bigl|1+t_n\rho_{ij}\bigr|.
\end{equation}
Let \( \Re z \) denote the real part of a complex number $z$. We next approximate each summand in \eqref{eq:logsum} by a second-order expansion.
\begin{lem}\label{lem:taylor}
If $|u|\le \tfrac12$, then
\[
\log|1+u|=\Re \left(u-\frac{u^2}{2}\right)+R(u),
\qquad |R(u)|\le \frac{2}{3}|u|^3.
\]
\end{lem}

\begin{proof}
For $|u|<1$ we have $\log(1+u)=\sum_{k\ge1}(-1)^{k+1}u^k/k$. Taking real parts yields $\log|1+u|=\Re\log(1+u)$. The tail is bounded by
\[
\left|\sum_{k\ge3}\frac{(-1)^{k+1}}{k}u^k\right|
\le \sum_{k\ge3}\frac{|u|^k}{k}
\le \frac{1}{3}\sum_{k\ge3}|u|^k
=\frac{|u|^3}{3(1-|u|)}
\le \frac{2}{3}|u|^3,
\]
for $|u|\le 1/2$.
\end{proof}
When $n$ is even, the $\pi$--antiperiodicity of the perturbation yields a global cancellation of the linear term.
\begin{lem}\label{lem:linearCancel}
Let $n=2m$ be even. Then $\sum_{i\neq j}\rho_{ij}=0$.
\end{lem}

\begin{proof}
We have $\zeta_{k+m}=-\zeta_k$ and, by Lemma~\ref{lem:gprops}(3),
$g(\theta_{k+m})=-g(\theta_k)$. Hence $g(\theta_{k+m})\zeta_{k+m}=g(\theta_k)\zeta_k$. Therefore for any $i\neq j$,
\[
\rho_{i+m,\;j+m}
=
\frac{g(\theta_{i+m})\zeta_{i+m}-g(\theta_{j+m})\zeta_{j+m}}{\zeta_{i+m}-\zeta_{j+m}}
=
\frac{g(\theta_i)\zeta_i-g(\theta_j)\zeta_j}{-(\zeta_i-\zeta_j)}
=
-\rho_{ij}.
\]
The map $(i,j)\mapsto(i+m,j+m)$ is a bijection on ordered pairs $i\neq j$,
so the sum cancels.
\end{proof}

Combining Lemmas~\ref{lem:taylor} and \ref{lem:linearCancel} now gives the desired asymptotic.

\begin{lem}\label{lem:logAsym}
Along even $n\to\infty$,
\[
\log\frac{\Delta(z_0,\dots,z_{n-1})}{n^n}
=
-\frac{t_n^2}{2}\sum_{i\neq j}\Re(\rho_{ij}^2)+o(1).
\]
\end{lem}

\begin{proof}
By Lemma~\ref{lem:rhoBound} and $t_n\le \pi^2/(12n)$, for all sufficiently large $n$ we have
$|t_n\rho_{ij}|\le 1/2$. Apply Lemma~\ref{lem:taylor} termwise to \eqref{eq:logsum}:
\[
\log\frac{\Delta(z_0,\dots,z_{n-1})}{n^n}
=\sum_{i\neq j}\Re \left(t_n\rho_{ij}-\frac{t_n^2}{2}\rho_{ij}^2\right)
+\sum_{i\neq j}R(t_n\rho_{ij}).
\]
By Lemma~\ref{lem:linearCancel} the linear term vanishes. For the remainder,
using $|R(u)|\le \frac23|u|^3$ and $|\rho_{ij}|\le 4$,
\[
\left|\sum_{i\neq j}R(t_n\rho_{ij})\right|
\le \frac{2}{3}\,n(n-1)\,(4t_n)^3
=O(n^2 t_n^3)=O(1/n)\to0.
\qedhere
\]
\end{proof}

\subsection{Limit of the quadratic term and evaluation of $J$}
To evaluate the right-hand side of Lemma~\ref{lem:logAsym} we pass to a continuum limit.
Define, for $x,y\in[0,2\pi]$,
\begin{equation}\label{eq:def_riemann_sum_limit_for_ss_1}
\xi(x)=e^{ix},\qquad f(x)=g(x)e^{ix},\qquad
\rho(x,y)=
\begin{cases}
\dfrac{f(x)-f(y)}{\xi(x)-\xi(y)},& d(x,y)>0,\\[1.2ex]
0,& d(x,y)=0,
\end{cases}
\end{equation}
and set
\[
F(x,y)=\Re(\rho(x,y)^2)\qquad (x,y\in[0,2\pi]).
\]
For $i\neq j$ we have $F(\theta_i,\theta_j)=\Re(\rho_{ij}^2)$.

\begin{lem}\label{lem:Riemann}
The function $F$ is bounded on $[0,2\pi]^2$ and continuous at every point $(x,y)$ with $d(x,y)>0$.
In particular, $F$ is Riemann integrable on $[0,2\pi]^2$ and
\[
\lim_{n\to\infty}\frac{1}{n^2}\sum_{i\neq j}\Re(\rho_{ij}^2)
=
\frac{1}{4\pi^2}\int_0^{2\pi}\int_0^{2\pi}F(x,y)\,dx\,dy
=:\,J.
\]
\end{lem}

\begin{proof}
Boundedness follows from the same estimate as in Lemma~\ref{lem:rhoBound}
(with $\theta_i,\theta_j$ replaced by $x,y$).
Continuity holds when $d(x,y)>0$ because $e^{ix}-e^{iy}\neq0$ there.
The set of points where $d(x,y)=0$ is
\[
D=\{(x,y)\in[0,2\pi]^2:\ e^{ix}=e^{iy}\}=\{(x,y)\in[0,2\pi]^2:\ x=y\}\cup\{(0,2\pi),(2\pi,0)\},
\]
which has Lebesgue measure $0$. We have shown that $F$ is bounded and continuous on $[0,2\pi]^2\setminus D$. By Lebesgue's criterion for Riemann integrability, it follows that $F$ is Riemann integrable on $[0,2\pi]^2$. 

Finally we relate the discrete sums to the integral.
Recall that $\theta_i=2\pi i/n$ for $i=0,1,\dots,n-1$. Consider the uniform
partition of $[0,2\pi]$ into subintervals of length $h=2\pi/n$ and the corresponding
two--dimensional Riemann sums. Since $F$ is Riemann integrable, we have
\[
h^2\sum_{i=0}^{n-1}\sum_{j=0}^{n-1}F(\theta_i,\theta_j)
\ \longrightarrow\
\int_0^{2\pi}\int_0^{2\pi}F(x,y)\,dx\,dy
\qquad (n\to\infty).
\]
Because $h^2=(2\pi/n)^2=4\pi^2/n^2$, dividing both sides by $4\pi^2$ yields
\[
\frac{1}{n^2}\sum_{i=0}^{n-1}\sum_{j=0}^{n-1}F(\theta_i,\theta_j)
\ \longrightarrow\
\frac{1}{4\pi^2}\int_0^{2\pi}\int_0^{2\pi}F(x,y)\,dx\,dy.
\]
Moreover, by our convention $F(\theta_i,\theta_i)=0$ for every $i$, hence
\[
\sum_{i=0}^{n-1}\sum_{j=0}^{n-1}F(\theta_i,\theta_j)
=
\sum_{i\neq j}F(\theta_i,\theta_j)
=
\sum_{i\neq j}\Re(\rho_{ij}^2).
\]
Combining the last two displays proves the claimed limit and completes the proof.\end{proof}

Now $J$ is an integral that can be computed, see~\cref{app:comp_J}, after which the proof can be finalised.

\begin{proof}[Proof of Theorem~\ref{thm:main_unifom_even_n_1}]
By Lemma~\ref{lem:logAsym} and Lemma~\ref{lem:Riemann},
\[
\log\frac{\Delta(z_0,\dots,z_{n-1})}{n^n}
=
-\frac{t_n^2}{2}\Bigl(n^2J+o(n^2)\Bigr)+o(1)
=
-\frac{(nt_n)^2}{2}\,J+o(1).
\]
Since $nt_n\to \pi^2/12$,
\[
\lim_{\substack{n\to\infty\\ n\ \mathrm{even}}}
\log\frac{\Delta(z_0,\dots,z_{n-1})}{n^n}
=
-\frac12\left(\frac{\pi^2}{12}\right)^2\left(\frac13-\frac{84\zeta(3)}{\pi^4}\right)
=
\frac{7}{24}\zeta(3)-\frac{\pi^4}{864}.
\]
Exponentiating,
\[
\lim_{\substack{n\to\infty\\ n\ \mathrm{even}}}
\frac{\Delta(z_0,\dots,z_{n-1})}{n^n}
=
\exp\left(\frac{7}{24}\zeta(3)-\frac{\pi^4}{864}\right).
\]
By Lemma~\ref{lem:diameter}, for all even $n\ge 8$ the configuration is feasible (diameter $\le2$),
so $\Delta_{\max}(n)\ge \Delta(z_0,\dots,z_{n-1})$ and therefore
\[
\liminf_{\substack{n\to\infty\\ n\ \mathrm{even}}}\oD_{\max}(n)
=
\liminf_{\substack{n\to\infty\\ n\ \mathrm{even}}}\frac{\Delta_{\max}(n)}{n^n}
\ge
\exp\left(\frac{7}{24}\zeta(3)-\frac{\pi^4}{864}\right).
\qedhere
\]\end{proof}

\begin{rem}[Other odd frequencies]
One can replace $g(\theta)=\mathrm{tri}(3\theta)$ by $\mathrm{tri}(m\theta)$ with $m$ odd and repeat the same
expansion-and-limit strategy, choosing the perturbation amplitude small enough to retain diameter~$\le2$.
In numerical experiments within this ``triangular-wave perturbation'' family, the choice $m=3$ appears to give the best
constant, while $m=1$ recovers Sothanaphan's construction in~\cite{Natso25}; we view this as further evidence for
Conjecture~\ref{conj:extrpolygonproperties}(iii).
\end{rem}

\bibliographystyle{abbrv}
\bibliography{ref}

\appendix

\section{Proof of Proposition~\ref{prop:n4}}\label{app:proof_n4}

Maximizing \(\Delta\) is equivalent to maximizing \(f= \sum_{1 \le j < k \le n} \log ( |z_k - z_j|^2 )\). Let \(\mathbf z=(z_1,z_2,z_3,z_4)\) be a maximizer for \(\Delta\) under the diameter constraint \(\max_{i,j}|z_i-z_j|\le 2\). By scaling we may assume the diameter equals \(2\). Let \(G\) be the diameter graph (equivalently, the graph of active constraints \(|z_i-z_j|=2\)).

By Lemma~\ref{lem:hopf_Pannwitz1}, \(G\) has at most \(4\) edges.
By Lemmas~\ref{lem:conn} and~\ref{lem:degge1}, \(G\) is connected and has minimum degree at least \(1\).
By Lemma~\ref{lem:NoC4}, \(G\) contains no even cycle (in particular, it is not \(C_4\)).
Hence, up to relabeling, \(G\) is one of:
\begin{align*}
&\text{(i) }K_{1,3} \text{ with edges } \{1,2\},\{2,3\},\{2,4\},\\
&\text{(ii) a triangle with a pendant edge, with edges } \{1,2\},\{1,3\},\{2,3\},\{2,4\},\\
&\text{(iii) }P_4 \text{ with edges } \{1,3\},\{2,3\},\{2,4\}.
\end{align*}

In all three cases we may assume \(\{2,4\}\) and \(\{2,3\}\) are diameter edges.
By translation and rotation we normalize
\[
z_2=0,\qquad z_4=2,\qquad z_3=2e^{i\beta}.
\]
Since \(\mathbf z\) is a (global hence local) maximizer, Theorem~\ref{thm:KKTNec} applies.

\smallskip
\noindent\underline{\emph{The \(k=4\) stationarity equation.}}
In each of the three cases above, the only active edge incident to \(4\) is \(\{2,4\}\). Thus Theorem~\ref{thm:KKTNec} with \(k=4\) gives
\begin{equation}\label{eq:k4}
\frac1{z_1-2}+\frac1{-2}+\frac1{2e^{i\beta}-2}=-2\lambda_{2,4},
\end{equation}
where \(\lambda_{2,4}\ge 0\).
In particular, the left-hand side is real.

\smallskip
\noindent\underline{\emph{Case (i): the star \(K_{1,3}\) cannot occur.}}
Here \(\{1,2\}\) is active, so \(z_1=2e^{i\alpha}\).
Taking imaginary parts in \eqref{eq:k4} and using
\[
\frac1{e^{it}-1}=-\frac12-\frac i2\cot\Bigl(\frac t2\Bigr)\qquad(t\not\equiv 0\!\!\!\pmod{2\pi}),
\]
we obtain
\[
0=\Im\!\left(\frac1{2(e^{i\alpha}-1)}+\frac1{2(e^{i\beta}-1)}\right)
=-\frac14\left(\cot\frac{\alpha}{2}+\cot\frac{\beta}{2}\right),
\]
hence \(\cot(\alpha/2)=-\cot(\beta/2)\), i.e.\ \(\beta\equiv-\alpha\pmod{2\pi}\).
Substituting \(\beta=-\alpha\) into \eqref{eq:k4} gives \(\lambda_{2,4}=\frac12\). Now apply Theorem~\ref{thm:KKTNec} with \(k=2\).
Since the active edges incident to \(2\) are \(\{1,2\},\{2,3\},\{2,4\}\), we have
\[
\frac1{z_1}+\frac1{z_3}+\frac1{z_4}
=\lambda_{1,2}(\overline z_1-\overline z_2)+\lambda_{2,3}(\overline z_3-\overline z_2)+\lambda_{2,4}(\overline z_4-\overline z_2).
\]
With \(z_1=2e^{i\alpha}\), \(z_3=2e^{-i\alpha}\), \(z_4=2\) and \(\lambda_{2,4}=\tfrac12\), the imaginary parts force
\(\lambda_{1,2}=\lambda_{2,3}\eqqcolon\lambda\).
Applying Theorem~\ref{thm:KKTNec} with \(k=1\) (the only active edge incident to \(1\) is \(\{1,2\}\)) yields
\[
\frac1{-2e^{i\alpha}}+\frac1{2e^{-i\alpha}-2e^{i\alpha}}+\frac1{2-2e^{i\alpha}}
=-2\lambda\,e^{-i\alpha}.
\]
Multiplying by \(2e^{i\alpha}\), the right-hand side becomes \(-4\lambda\in\mathbb R\), so the imaginary part of the left-hand side must vanish.
A direct computation gives
\[
0=\Im\!\left(\frac{e^{i\alpha}}{e^{-i\alpha}-e^{i\alpha}}+\frac{e^{i\alpha}}{1-e^{i\alpha}}\right)
=\frac{1+2\cos\alpha}{2\sin\alpha}.
\]
Since \(\sin\alpha\neq 0\) (otherwise points collide or violate the diameter bound), we get \(\cos\alpha=-\tfrac12\).
Then
\[
|z_1-z_4|=|2e^{i\alpha}-2|=2|e^{i\alpha}-1|=2\sqrt{2-2\cos\alpha}=2\sqrt3>2,
\]
contradicting feasibility. Hence case~(i) is impossible for a maximizer.

\smallskip
\noindent\underline{\emph{Case (ii): triangle with a pendant edge.}}
Here \(\{1,2\},\{2,3\},\{1,3\}\) are active, so
\(
z_1=2e^{i\alpha}
\)
and
\(
|z_1-z_3|=2
\),
i.e.
\[
|2e^{i\alpha}-2e^{i\beta}|=2
\quad\Longleftrightarrow\quad
4\Bigl|\sin\Bigl(\frac{\alpha-\beta}{2}\Bigr)\Bigr|=2
\quad\Longleftrightarrow\quad
\alpha-\beta\equiv \pm\frac{\pi}{3}\pmod{2\pi}.
\]
On the other hand, the imaginary-part argument from \eqref{eq:k4} (as above) yields again
\(\beta\equiv-\alpha\pmod{2\pi}\).
Combining these two relations gives
\(
2\alpha\equiv\pm\frac{\pi}{3}\pmod{2\pi}
\),
hence \(\alpha\equiv \pm\frac{\pi}{6},\,\pm\frac{5\pi}{6}\pmod{2\pi}\).

The remaining feasibility constraint \(|z_1-z_4|\le 2\) reads
\[
|2e^{i\alpha}-2|=4\Bigl|\sin\Bigl(\frac{\alpha}{2}\Bigr)\Bigr|\le 2
\quad\Longleftrightarrow\quad
\Bigl|\sin\Bigl(\frac{\alpha}{2}\Bigr)\Bigr|\le \frac12,
\]
which rules out \(\alpha\equiv \pm\frac{5\pi}{6}\).
Thus \(\alpha\equiv \pm\frac{\pi}{6}\), and (up to complex conjugation and relabeling \(z_1\leftrightarrow z_3\)) we may take \(\alpha=\frac{\pi}{6}\).
Therefore
\[
z_1=\sqrt3+i,\qquad z_2=0,\qquad z_3=\sqrt3-i,\qquad z_4=2.
\]
For this configuration,
\[
|z_1-z_2|=|z_2-z_3|=|z_2-z_4|=|z_1-z_3|=2,\qquad
|z_1-z_4|^2=|z_3-z_4|^2=(2-\sqrt3)^2+1=8-4\sqrt3.
\]
Hence
\[
\overline{\Delta}=\frac{\Delta}{4^4}
=\frac{4^4(8-4\sqrt3)^2}{4^4}
=(8-4\sqrt3)^2
=16(7-4\sqrt3).
\]

\smallskip
\noindent\underline{\emph{Case (iii): the path \(P_4\) yields \(\overline{\Delta}\le 1\).}}
Assume the diameter graph is $P_4$ with active set
\(
\{\{1,3\},\{2,3\},\{2,4\}\}.
\)
Normalize as before:
\[
z_2=0,\qquad z_4=2,\qquad z_3=2e^{i\beta},\qquad z_1=z_3+2e^{i\alpha}=2e^{i\beta}+2e^{i\alpha}.
\]
Let \(\lambda_{1,3},\lambda_{2,3},\lambda_{2,4}\ge 0\) be the corresponding multipliers.
Applying Theorem~\ref{thm:KKTNec} with \(k=4,2,1,3\) gives
\begin{align}
\frac1{z_1-2}+\frac1{-2}+\frac1{z_3-2} &= -2\lambda_{2,4}, \label{eq:P4-k4}\\
\frac1{z_1}+\frac1{z_3}+\frac1{2} &= 2\lambda_{2,3}e^{-i\beta}+2\lambda_{2,4}, \label{eq:P4-k2}\\
\frac1{-z_1}+\frac1{z_3-z_1}+\frac1{2-z_1} &= -2\lambda_{1,3}e^{-i\alpha}, \label{eq:P4-k1}\\
\frac1{z_1-z_3}+\frac1{-z_3}+\frac1{2-z_3} &= 2\lambda_{1,3}e^{-i\alpha}-2\lambda_{2,3}e^{-i\beta}. \label{eq:P4-k3}
\end{align}

\begin{claim}\label{claim:solve_equations_vv1}
Solving \eqref{eq:P4-k4}--\eqref{eq:P4-k3} yields uniquely (up to complex conjugation)
\[
e^{i\beta}=\frac34-\frac{\sqrt7}{4}i,\qquad
e^{i\alpha}=-\frac18+\frac{3\sqrt7}{8}i,\qquad
\lambda_{2,4}=\lambda_{1,3}=\frac34,\qquad
\lambda_{2,3}=0.
\]    
\end{claim}

\begin{proof}[Proof of Claim~\ref{claim:solve_equations_vv1}]
Set $A:=e^{i\alpha}$ and $B:=e^{i\beta}$ so that $z_3=2B$ and $z_1=2(A+B)$. From \eqref{eq:P4-k4} we obtain an explicit expression for \(\lambda_{2,4}\):
\begin{equation}\label{eq:lam24-explicit}
4\lambda_{2,4}
=1-\frac{1}{A+B-1}-\frac{1}{B-1}.
\end{equation}
From \eqref{eq:P4-k1} we similarly get
\begin{equation}\label{eq:lam13-explicit}
4\lambda_{1,3}
=1+\frac{A}{A+B}+\frac{A}{A+B-1}.
\end{equation}
Substituting \eqref{eq:lam24-explicit} into \eqref{eq:P4-k2} and simplifying yields
\begin{equation}\label{eq:lam23-explicit}
4\lambda_{2,3}
=\frac{(A+2B-1)\bigl(A(2B-1)+2B(B-1)\bigr)}{(A+B)(B-1)(A+B-1)}.
\end{equation}

Since the KKT multipliers are real, we must have
\(\lambda_{2,4},\lambda_{2,3}\in\mathbb R\).
Using \(\overline A=A^{-1}\), \(\overline B=B^{-1}\) (because \(|A|=|B|=1\)),
the condition \(\lambda_{2,4}=\overline{\lambda_{2,4}}\) is equivalent (after clearing denominators,
and using \(B\neq 1\) since \(z_3\neq z_4\)) to
\begin{equation}\label{eq:real-lam24-poly}
2A^2B^2 - A^2B - A^2 + 2AB^3 -3AB^2 -3AB +2A -B^3 -B^2 +2B=0.
\end{equation}
Likewise, \(\lambda_{2,3}=\overline{\lambda_{2,3}}\) and \eqref{eq:lam23-explicit} give
\begin{equation}\label{eq:real-lam23-factor}
(A+B^2)\bigl(3A^2B-3A^2+3AB^2-8AB+3A-3B^2+3B\bigr)=0,
\end{equation}
again after clearing denominators.

Assume for contradiction that the second factor in \eqref{eq:real-lam23-factor} vanishes, i.e.
\begin{equation}\label{eq:F-branch}
3A^2B-3A^2+3AB^2-8AB+3A-3B^2+3B=0.
\end{equation}
Rewrite \eqref{eq:F-branch} and \eqref{eq:real-lam24-poly} as quadratic equations in $A$:
\[
3(B-1)A^2+(3B^2-8B+3)A-3B(B-1)=0,
\]
and
\[
(B-1)(2B+1)A^2+(2B^3-3B^2-3B+2)A-B(B-1)(B+2)=0.
\]
Since $B\neq 1$, multiply the first equation by $(2B+1)$ and subtract $3$ times the second equation.
The $A^2$-terms cancel, and a short simplification gives
\[
(B-1)\bigl((4B-3)A+3B(B-1)\bigr)=0,
\]
hence
\begin{equation}\label{eq:A-rational}
A=-\frac{3B(B-1)}{4B-3}.
\end{equation}
Taking absolute values and using $|A|=|B|=1$ yields
\[
|4B-3|=3|B-1|.
\]
Now $|B-1|^2=2-2\cos\beta$ and $|4B-3|^2=25-24\cos\beta$, hence
\[
25-24\cos\beta=9(2-2\cos\beta)\quad\Rightarrow\quad \cos\beta=\frac{7}{6},
\]
a contradiction. Therefore the second factor in \eqref{eq:real-lam23-factor} cannot vanish, and so
\begin{equation}\label{eq:A=-B2}
A=-B^2.
\end{equation}

Substituting \eqref{eq:A=-B2} into \eqref{eq:real-lam24-poly} gives the factorization
\[
0=B(B+1)(B^2-B+1)(2B^2-3B+2).
\]
Now \(B\neq 0\). Also \(B=-1\) is impossible because then
\(|z_3-z_4|=|-2-2|=4>2\).
Finally, \(B^2-B+1=0\) implies \(B=e^{\pm i\pi/3}\), hence
\(|z_3-z_4|=|2B-2|=2\), i.e.\ the edge \(\{3,4\}\) would also be active, contradicting the path assumption.
Therefore
\begin{equation}\label{eq:B-quadratic}
2B^2-3B+2=0.
\end{equation}
Solving \eqref{eq:B-quadratic} gives \(B=\frac34\pm \frac{\sqrt7}{4}i\), and then \eqref{eq:A=-B2} gives
\(A=-B^2=-\frac18\mp \frac{3\sqrt7}{8}i\).
Substituting these values into \eqref{eq:lam24-explicit}--\eqref{eq:lam23-explicit} yields
\[
\lambda_{2,4}=\lambda_{1,3}=\frac34,\qquad \lambda_{2,3}=0,
\]
and the two choices of sign correspond to complex conjugation.

The values of \(A,B,\lambda_{1,3},\lambda_{2,3},\lambda_{2,4}\) obtained above were derived from
\eqref{eq:P4-k4}, \eqref{eq:P4-k2}, \eqref{eq:P4-k1} together with the reality constraints on the multipliers.
Substituting them into \eqref{eq:P4-k3} shows that \eqref{eq:P4-k3} is satisfied as well. This completes the proof.
\end{proof}

Thus, by Claim~\ref{claim:solve_equations_vv1}, after our normalization (translation and rotation) the KKT system has a unique solution up to complex conjugation. We may therefore assume
\[
z_2=0,\quad z_4=2,\quad z_3=\frac32-\frac{\sqrt7}{2}i,\quad z_1=\frac54+\frac{\sqrt7}{4}i.
\]
For this configuration,
\[
|z_1-z_2|=\sqrt2,\quad |z_3-z_4|=\sqrt2,\quad |z_1-z_4|=1,
\]
while the active ones satisfy \(|z_1-z_3|=|z_2-z_3|=|z_2-z_4|=2\).
Hence
\[
\Delta=\prod_{1\le i<j\le 4}|z_i-z_j|^2
=2 \cdot 4 \cdot 1 \cdot 4 \cdot 4 \cdot 2 =256,
\qquad
\overline{\Delta}=\frac{\Delta}{4^4}=1.
\]
Consequently, case~(iii) cannot achieve the global maximum since
\(16(7-4\sqrt3)>1\), and any maximizer in the path regime would either satisfy KKT
(and hence give \(\overline{\Delta}(4)=1\)) or else lie on the boundary where an additional
distance constraint becomes active, which reduces to case~(ii).

\smallskip
Combining the three cases, the unique maximizers (up to congruence and relabeling) are exactly the kites from case~(ii),
and the maximal value is \(16(7-4\sqrt3)\).\qed

\section{Detailed proof of Lemma~\ref{lem:diamY}}\label{app:diamY}

We only prove $\diam(Y)\le \cos(\pi/2n)$, since the reverse inequality was shown in the main text.

Set $\delta:=\pi/n$.
Since $Y$ is equilateral with side length $\ell=\sin\delta$, let
\[
e_r = \overrightarrow{B_r B_{r+1}}, \qquad 1\le r\le n,
\]
so that $|e_r|=\ell$ for all $r$. By construction $Y$ is a closed polygonal chain, hence $\sum_{r=1}^n e_r=0$. Moreover, by construction, the exterior angles of $Y$ equal:
\[
\delta \text{ at } B_k,B_{3k},B_{5k},\qquad
3\delta \text{ at } B_{2k},B_{4k},B_{6k},\qquad
2\delta \text{ otherwise}.
\]

\begin{claim}\label{clm:directions}
Taking $\arg(e_1)\equiv 0\pmod\pi$, the edge directions of $Y$ modulo $\pi$
are exactly $\{0,\delta,2\delta,\dots,(n-1)\delta\}$.
\end{claim}
\begin{proof}[Proof of Claim~\ref{clm:directions}]
The boundary of $Y$ splits into six blocks of $k$ consecutive edges.
Within each block all exterior angles equal $2\delta$, hence the $k$ edge directions in the block form an arithmetic progression
with step $2\delta$ modulo $\pi$.
Since $(k-1)\cdot 2\delta=\pi/3-2\delta$ and the junction exterior angles alternate between $\delta$ and $3\delta$,
the starting directions of successive blocks differ by $(2k-1)\delta$ and $(2k+1)\delta$ alternately.
Starting from $\arg(e_1)\equiv 0\pmod\pi$, this gives the following set of directions:
\[
\begin{aligned}
A =& \{2j\delta : 0 \le j \le k-1\}
\cup \{(2k - 1 + 2j)\delta : 0 \le j \le k - 1\}\\
&
\cup \{(4k + 2j)\delta : 0 \le j \le k - 1\}\cup \{(6k - 1 + 2j)\delta : 0 \le j \le k - 1\}\\
&
\cup \{(8k + 2j)\delta : 0 \le j \le k - 1\}
\cup \{(10k - 1 + 2j)\delta : 0 \le j \le k - 1\}.
\end{aligned}
\]
Since $n\delta=\pi$, working modulo $\pi$ is equivalent to reducing integer coefficients modulo $n=6k$.
Thus $6k\equiv 0$, $8k\equiv 2k$, and $10k\equiv 4k$ modulo $n$, so
\[
\begin{aligned}
A \equiv&
\{2j\delta : 0 \le j \le k-1\}
\cup \{(2k - 1 + 2j)\delta : 0 \le j \le k - 1\}\\
&
\cup \{(4k + 2j)\delta : 0 \le j \le k - 1\}\cup \{(-1 + 2j)\delta : 0 \le j \le k - 1\}\\
&
\cup \{(2k + 2j)\delta : 0 \le j \le k - 1\}
\cup \{(4k - 1 + 2j)\delta : 0 \le j \le k - 1\}
\pmod\pi.
\end{aligned}
\]
The three progressions with even coefficients cover all even residues modulo $n$ (each has size $k$ and they are disjoint),
and the three progressions with odd coefficients cover all odd residues modulo $n$.
Hence the union is exactly $\{0,1,\dots,n-1\}\delta$ modulo $\pi$, as claimed.
\end{proof}

By Claim~\ref{clm:directions}, for each $m\in\{0,\dots,n-1\}$,
exactly one of the two opposite vectors
\[
u_m = \ell\, e^{im\delta}, \qquad -u_m
\]
appears among $\{e_1,\dots,e_n\}$.

Fix arbitrary vertices $B_i,B_j$.
Let $I\subset\{1,\dots,n\}$ be the index set of edges along the boundary arc
from $B_i$ to $B_j$.
Then
\[
B_j - B_i = \sum_{r\in I} e_r.
\]
Using $\sum_{r=1}^n e_r=0$, we may write
\[
B_j - B_i
=
\frac12
\left(
\sum_{r\in I} e_r - \sum_{r\notin I} e_r
\right).
\]
Therefore there exist signs $\varepsilon_m\in\{\pm1\}$ such that
\[
B_j - B_i
=
\frac12
\sum_{m=0}^{n-1}
\varepsilon_m u_m.
\]
Consequently,
\[
|B_j - B_i|
\le
\frac12
\max_{\varepsilon_m\in\{\pm1\}}
\left|
\sum_{m=0}^{n-1}
\varepsilon_m u_m
\right|.
\]

For any direction $e^{it}$,
the projection of $\sum \varepsilon_m u_m$ onto this direction equals
\[
\ell \sum_{m=0}^{n-1}
\varepsilon_m \cos(t-m\delta).
\]
The maximum is obtained by taking
$\varepsilon_m = \operatorname{sgn}(\cos(t-m\delta))$.
Hence
\[
\max_{\varepsilon_m}
\left|
\sum_{m=0}^{n-1}
\varepsilon_m u_m
\right|
=
\ell
\max_{t\in\mathbb{R}}
\sum_{m=0}^{n-1}
|\cos(t-m\delta)|.
\]

We claim that
\[
\max_{t}
\sum_{m=0}^{n-1}
|\cos(t-m\delta)|
\leq 
\frac{1}{\sin(\delta/2)}.
\]
Indeed, set \(S(t)=\sum_{m=0}^{n-1}\bigl|\cos(t-m\delta)\bigr|\). Note that $S(t+\delta)=S(t)$. Since $|\cos x|$ has period $\pi$, we may assume $t\in[-\pi/2,\pi/2)$.
Choose the unique integer $M\in\{0,1,\dots,n-1\}$ such that
\[
t\in\Bigl[M\delta-\frac{\pi}{2},\ (M+1)\delta-\frac{\pi}{2}\Bigr).
\]
Then for $m\le M$ we have $t-m\delta\in[-\pi/2,\pi/2)$ and hence $\cos(t-m\delta)\ge0$,
while for $m\ge M+1$ we have $t-m\delta\in(-3\pi/2,-\pi/2)$ and hence $\cos(t-m\delta)\le0$. Therefore
\[
S(t)=\sum_{m=0}^{M}\cos(t-m\delta)-\sum_{m=M+1}^{n-1}\cos(t-m\delta)
=\Re\left(e^{it}\Bigl(\sum_{m=0}^{M}e^{-im\delta}-\sum_{m=M+1}^{n-1}e^{-im\delta}\Bigr)\right),
\]where $\Re(z)$ denotes the real part of the complex number $z$. Using geometric series and $n\delta=\pi$ (so $e^{-in\delta}=e^{-i\pi}=-1$), we compute
\begin{align*}
\sum_{m=0}^{M}e^{-im\delta}-\sum_{m=M+1}^{n-1}e^{-im\delta}
&=2\sum_{m=0}^{M}e^{-im\delta}-\sum_{m=0}^{n-1}e^{-im\delta}\\
&=\frac{2(1-e^{-i(M+1)\delta})-(1-e^{-in\delta})}{1-e^{-i\delta}}
\\
&=\frac{-2e^{-i(M+1)\delta}}{1-e^{-i\delta}}=\frac{i\,e^{-i(M+\frac12)\delta}}{\sin(\delta/2)},
\end{align*}
since $1-e^{-i\delta}=e^{-i\delta/2}\,2i\sin(\delta/2)$. Hence
\[
S(t)=\frac{1}{\sin(\delta/2)} \Re \bigl(i\,e^{i(t-(M+\frac12)\delta)}\bigr)
=-\frac{\sin \bigl(t-(M+\frac12)\delta\bigr)}{\sin(\delta/2)}
\le \frac{1}{\sin(\delta/2)}.
\]Thus \(\max_{t\in\mathbb R} \sum_{m=0}^{n-1}\bigl|\cos(t-m\delta)\bigr| \leq \frac{1}{\sin(\delta/2)}\). Therefore
\[
\max_{\varepsilon_m}
\left|
\sum_{m=0}^{n-1}
\varepsilon_m u_m
\right|
\leq
\frac{\ell}{\sin(\delta/2)}
=
\frac{\sin\delta}{\sin(\delta/2)}
=
2\cos(\delta/2).
\]
We conclude that for all $i,j$, \(|B_iB_j|
\le
\cos(\delta/2)
=
\cos(\pi/2n)\). Hence \(\diam(Y)\leq \cos(\pi/2n)\). This completes the proof.\qed

\section{Computation of the integrals in~\cref{sec:diam_graph_6k} }\label{app:AnIntegral}

In this section, we add the computations within Maple for the integrals appearing in the lemmas in~\cref{sec:diam_graph_6k}.

\subsection{Computation for the integral in Lemma~\ref{lem:2ndregime}}\label{subsec:integral_v1}

\mapleinput
{$ \displaystyle \texttt{>\,} f (b ,c ,r)\coloneqq \sin^{2}(b)+\sin^{2}(c)+2\cdot \sin (b)\cdot \sin (c)\cdot \cos (b +c +r)\, $}

% \mapleresult
\begin{dmath}\label{(1)}
f \coloneqq \left(b ,c ,r \right)\hiderel{\mapsto }\sin \! \left(b \right)^{2}+\sin \! \left(c \right)^{2}+2\cdot \sin \! \left(b \right)\cdot \sin \! \left(c \right)\cdot \cos \! \left(b +c +r \right)
\end{dmath}
\mapleinput
{$ \displaystyle \texttt{>\,} g (b ,c ,r)\coloneqq \mathit{simplify} (\mathit{expand} (f (b ,c ,0)^{2}-f (b ,c ,r)\cdot f (b ,c ,-r)))\, $}

% \mapleresult
\begin{dmath}\label{(2)}
g \coloneqq \left(b ,c ,r \right)\hiderel{\mapsto }\mathit{simplify} \! \left(\mathit{expand} \! \left(f \! \left(b ,c ,0\right)^{2}-f \! \left(b ,c ,r \right)\cdot f \! \left(b ,c ,-r \right)\right)\right)
\end{dmath}
\mapleinput
{$ \displaystyle \texttt{>\,} \mathit{collect} (\mathit{expand} (g (b ,c ,r)),\cos (r))\, $}

% \mapleresult
\begin{dmath}\label{(3)}
-4 \sin \! \left(c \right)^{2} \sin \! \left(b \right)^{2} \cos \! \left(r \right)^{2}+\left(4 \cos \! \left(b \right)^{3} \cos \! \left(c \right) \sin \! \left(c \right) \sin \! \left(b \right)-4 \cos \! \left(b \right)^{2} \sin \! \left(c \right)^{2} \sin \! \left(b \right)^{2}+4 \cos \! \left(b \right) \cos \! \left(c \right)^{3} \sin \! \left(c \right) \sin \! \left(b \right)-4 \cos \! \left(c \right)^{2} \sin \! \left(c \right)^{2} \sin \! \left(b \right)^{2}-8 \cos \! \left(b \right) \cos \! \left(c \right) \sin \! \left(c \right) \sin \! \left(b \right)+8 \sin \! \left(c \right)^{2} \sin \! \left(b \right)^{2}\right) \cos \! \left(r \right)-4 \cos \! \left(b \right)^{3} \cos \! \left(c \right) \sin \! \left(c \right) \sin \! \left(b \right)+4 \cos \! \left(b \right)^{2} \sin \! \left(c \right)^{2} \sin \! \left(b \right)^{2}-4 \cos \! \left(b \right) \cos \! \left(c \right)^{3} \sin \! \left(c \right) \sin \! \left(b \right)-4 \sin \! \left(c \right)^{4} \sin \! \left(b \right)^{2}+8 \cos \! \left(b \right) \cos \! \left(c \right) \sin \! \left(c \right) \sin \! \left(b \right)
\end{dmath}
% \mapleinput
% {$ \displaystyle \texttt{>\,} \,\mathit{simplify} (\frac{(-4 \sin (b)^{2} \sin (c)^{2} \cdot (1-r^{2})+(-4 \sin (b)^{2} \sin (c)^{2} \cos (c)^{2}-4 \sin (b)^{2} \sin (c)^{2} \cos (b)^{2}+4 \sin (b) \sin (c) \cos (c)^{3} \cos (b)+4 \sin (b) \sin (c) \cos (c) \cos (b)^{3}+8 \sin (b)^{2} \sin (c)^{2}-8 \sin (b) \sin (c) \cos (b) \cos (c)) \cdot (1-\frac{r^{2}}{2})-4\,\sin (c)^{4} \sin (b)^{2}+4 \sin (b)^{2} \sin (c)^{2} \cos (b)^{2}-4 \sin (b) \sin (c) \cos (c)^{3} \cos (b)-4 \sin (b) \sin (c) \cos (c) \cos (b)^{3}+8 \sin (b) \sin (c) \cos (b) \cos (c))}{r^{2}}) $}

\mapleinput{
$\displaystyle
\begin{aligned}
\texttt{>\,}\,\mathit{simplify}\Bigg(
\frac{1}{r^{2}}\Big(&
-4 \sin^{2}(b)\sin^{2}(c)(1-r^{2}) \\
&+ \big(
-4 \sin^{2}(b)\sin^{2}(c)\cos^{2}(c)
-4 \sin^{2}(b)\sin^{2}(c)\cos^{2}(b) \\
&\quad
+4 \sin(b)\sin(c)\cos^{3}(c)\cos(b)
+4 \sin(b)\sin(c)\cos(c)\cos^{3}(b) \\
&\quad
+8 \sin^{2}(b)\sin^{2}(c)
-8 \sin(b)\sin(c)\cos(b)\cos(c)
\big)\left(1-\frac{r^{2}}{2}\right) \\
&-4\sin^{4}(c)\sin^{2}(b)
+4\sin^{2}(b)\sin^{2}(c)\cos^{2}(b) \\
&-4\sin(b)\sin(c)\cos^{3}(c)\cos(b)
-4\sin(b)\sin(c)\cos(c)\cos^{3}(b) \\
&+8\sin(b)\sin(c)\cos(b)\cos(c)
\Big)
\Bigg)
\end{aligned}
$}

% \mapleresult
\begin{dmath}\label{(4)}
-2 \sin \! \left(b \right) \left(\cos \! \left(c \right) \cos \! \left(b \right)^{3}-\cos \! \left(b \right)^{2} \sin \! \left(c \right) \sin \! \left(b \right)+\left(\cos \! \left(c \right)^{3}-2 \cos \! \left(c \right)\right) \cos \! \left(b \right)-\cos \! \left(c \right)^{2} \sin \! \left(c \right) \sin \! \left(b \right)\right) \sin \! \left(c \right)
\end{dmath}
\mapleinput
{ \(\begin{aligned}  \texttt{>\,} \,h (b ,c)\coloneqq &-2 (\cos (c) \cos (b)^{3}-\sin (b) \sin (c) \cos (b)^{2}\\&+(\cos (c)^{3}-2 \cos (c)) \cos (b)-\sin (b) \sin (c) \cos (c)^{2}) \sin (b) \sin (c) \end{aligned}\) }

% \mapleresult
\begin{dmath}\label{(5)}
h \coloneqq \left(b ,c \right)\hiderel{\mapsto }-\left(2\cdot \cos \! \left(c \right)\cdot \cos \! \left(b \right)^{3}-2\cdot \sin \! \left(b \right)\cdot \sin \! \left(c \right)\cdot \cos \! \left(b \right)^{2}+2\cdot \left(\cos \! \left(c \right)^{3}-2\cdot \cos \! \left(c \right)\right)\cdot \cos \! \left(b \right)-2\cdot \sin \! \left(b \right)\cdot \sin \! \left(c \right)\cdot \cos \! \left(c \right)^{2}\right)\cdot \sin \! \left(b \right)\cdot \sin \! \left(c \right)
\end{dmath}
\mapleinput
{$ \displaystyle \texttt{>\,} \mathit{simplify} (h (b ,c)-2\cdot (\sin^{2}(b)+\sin^{2}(c))\cdot \cos (b +c)\cdot  \sin (b) \sin (c)) $}

% \mapleresult
\begin{dmath}\label{(6)}
4 \sin \! \left(c \right)^{2} \sin \! \left(b \right)^{2}
\end{dmath}
\mapleinput
{$ \displaystyle \texttt{>\,} \mathit{simplify} (\mathit{int} (\mathit{int} (\frac{h (b ,c)}{\sin (b +c)^{4}}\,,b =0..\frac{\pi}{6})\,,c =0..\frac{\pi}{6})) $}

% \mapleresult
\begin{dmath}\label{(7)}
-\frac{1}{4}+\frac{\pi  \sqrt{3}}{24}+\frac{\ln \! \left(3\right)}{8}
\end{dmath}
\mapleinput
{$ \displaystyle \texttt{>\,} \mathit{plot3d} (h (b ,c),b =0..\frac{ 3.15}{6},c =0..\frac{ 3.15}{6}) $}

% \mapleresult
\mapleplot{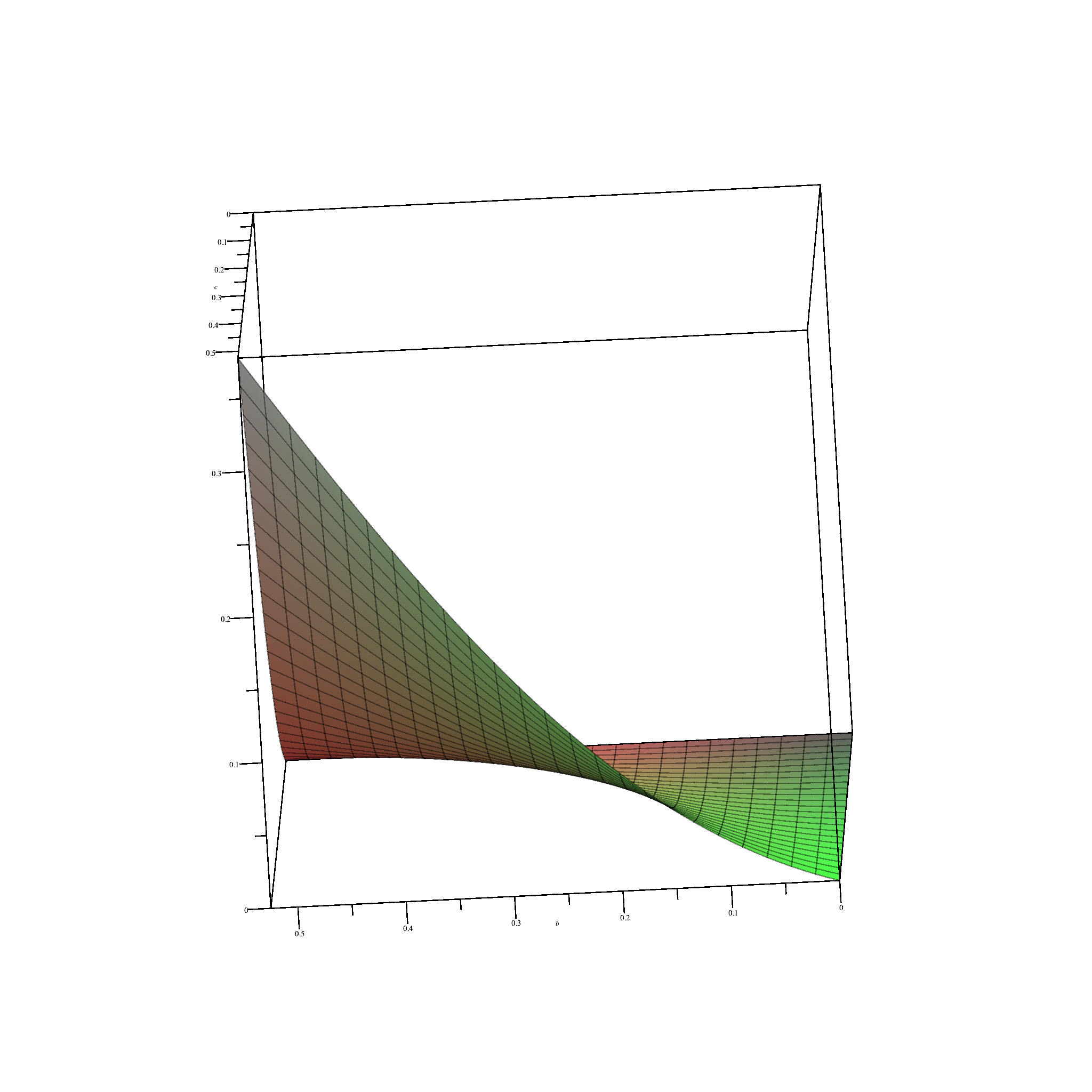}
% \mapleinput
% {$ \displaystyle \texttt{>\,}  $}

\subsection{Computation for the integral in Lemma~\ref{lem:3rdregime}}\label{subsec:integral_v2}

\mapleinput
{$ \displaystyle \texttt{>\,} f (b ,c ,r)\coloneqq (\frac{1}{2}-\cos (\frac{5}{6}\,\pi -c -r)\cdot \sin (c)-\cos (\frac{5}{6}\,\pi -b +r)\cdot \sin (b))^{2}+(\sin (\frac{5}{6}\,\pi -c -r)\cdot \sin (c)-\sin (\frac{5}{6}\,\pi -b +r)\cdot \sin (b))^{2}\, $}

% \mapleresult
\begin{dmath}\label{(1)'}
f \coloneqq \left(b ,c ,r \right)\hiderel{\mapsto }\left(\frac{1}{2}-\cos \! \left(\frac{5\cdot \pi}{6}-c -r \right)\cdot \sin \! \left(c \right)-\cos \! \left(\frac{5\cdot \pi}{6}-b +r \right)\cdot \sin \! \left(b \right)\right)^{2}+\left(\sin \! \left(\frac{5\cdot \pi}{6}-c -r \right)\cdot \sin \! \left(c \right)-\sin \! \left(\frac{5\cdot \pi}{6}-b +r \right)\cdot \sin \! \left(b \right)\right)^{2}
\end{dmath}
\mapleinput
{$ \displaystyle \texttt{>\,} g (b ,c ,r)\coloneqq \mathit{simplify} (\mathit{expand} (f (b ,c ,0)^{2}-f (b ,c ,r)\cdot f (b ,c ,-r)))\, $}

% \mapleresult
\begin{dmath}\label{(2)'}
g \coloneqq \left(b ,c ,r \right)\hiderel{\mapsto }\mathit{simplify} \! \left(\mathit{expand} \! \left(f \! \left(b ,c ,0\right)^{2}-f \! \left(b ,c ,r \right)\cdot f \! \left(b ,c ,-r \right)\right)\right)
\end{dmath}
\mapleinput
{$ \displaystyle \texttt{>\,} \,\mathit{collect} (\mathit{expand} (g (b ,c ,r)),\cos (r))\, $}

% \mapleresult
\begin{dmath}\label{(3)'}
\left(\sin \! \left(c \right) \sqrt{3}\, \cos \! \left(c \right) \sin \! \left(b \right)^{2}+\sin \! \left(c \right)^{2} \sin \! \left(b \right) \sqrt{3}\, \cos \! \left(b \right)+\cos \! \left(c \right)^{2} \cos \! \left(b \right)^{2}-\sin \! \left(c \right) \cos \! \left(c \right) \sin \! \left(b \right) \cos \! \left(b \right)-1\right) \cos \! \left(r \right)^{2}+\left(-3 \cos \! \left(b \right)^{4}+\frac{7 \cos \! \left(b \right)^{2}}{4}+4 \cos \! \left(b \right)^{4} \cos \! \left(c \right)^{2}+4 \cos \! \left(b \right)^{2} \cos \! \left(c \right)^{4}-8 \cos \! \left(c \right)^{2} \cos \! \left(b \right)^{2}+\sqrt{3}\, \sin \! \left(b \right) \cos \! \left(b \right)^{3}-\frac{13 \sin \! \left(b \right) \sqrt{3}\, \cos \! \left(b \right)}{4}+\sin \! \left(c \right) \sqrt{3}\, \cos \! \left(c \right)^{3}-\frac{13 \sin \! \left(c \right) \sqrt{3}\, \cos \! \left(c \right)}{4}+2 \sqrt{3}\, \sin \! \left(b \right) \cos \! \left(b \right) \cos \! \left(c \right)^{2}+8 \sin \! \left(c \right) \cos \! \left(c \right) \sin \! \left(b \right) \cos \! \left(b \right)-4 \sin \! \left(c \right) \cos \! \left(c \right) \sin \! \left(b \right) \cos \! \left(b \right)^{3}+2 \sin \! \left(c \right) \sqrt{3}\, \cos \! \left(c \right) \cos \! \left(b \right)^{2}-4 \sin \! \left(c \right) \cos \! \left(c \right)^{3} \sin \! \left(b \right) \cos \! \left(b \right)+\frac{5}{2}-3 \cos \! \left(c \right)^{4}+\frac{7 \cos \! \left(c \right)^{2}}{4}\right) \cos \! \left(r \right)+\frac{9 \sin \! \left(c \right) \sqrt{3}\, \cos \! \left(c \right)}{4}-\sqrt{3}\, \sin \! \left(b \right) \cos \! \left(b \right)^{3}+\frac{9 \sin \! \left(b \right) \sqrt{3}\, \cos \! \left(b \right)}{4}-\sin \! \left(c \right) \sqrt{3}\, \cos \! \left(c \right)^{3}+4 \sin \! \left(c \right) \cos \! \left(c \right) \sin \! \left(b \right) \cos \! \left(b \right)^{3}-\sin \! \left(c \right) \sqrt{3}\, \cos \! \left(c \right) \cos \! \left(b \right)^{2}+4 \sin \! \left(c \right) \cos \! \left(c \right)^{3} \sin \! \left(b \right) \cos \! \left(b \right)-\sqrt{3}\, \sin \! \left(b \right) \cos \! \left(b \right) \cos \! \left(c \right)^{2}-7 \sin \! \left(c \right) \cos \! \left(c \right) \sin \! \left(b \right) \cos \! \left(b \right)-\frac{7 \cos \! \left(c \right)^{2}}{4}+3 \cos \! \left(c \right)^{4}-4 \cos \! \left(b \right)^{2} \cos \! \left(c \right)^{4}+7 \cos \! \left(c \right)^{2} \cos \! \left(b \right)^{2}+3 \cos \! \left(b \right)^{4}-\frac{7 \cos \! \left(b \right)^{2}}{4}-4 \cos \! \left(b \right)^{4} \cos \! \left(c \right)^{2}-\frac{3}{2}
\end{dmath}

\mapleinput{
$\displaystyle
\begin{aligned}
\texttt{>\,}\,\mathit{simplify}\Bigg(
\frac{1}{r^{2}}\Big(&
\big(
 \sin^{2}(c)\sin(b)\sqrt{3}\cos(b)
+\sin(c)\sqrt{3}\cos(c)\sin^{2}(b) \\
&\quad
-\sin(c)\cos(c)\sin(b)\cos(b)
+\cos^{2}(b)\cos^{2}(c)-1
\big)(1-r^{2}) \\
&+
\big(
 -4\sin(c)\cos(c)\sin(b)\cos^{3}(b)
 +2\sin(c)\sqrt{3}\cos(c)\cos^{2}(b) \\
&\quad
 +2\sqrt{3}\sin(b)\cos(b)\cos^{2}(c)
 +8\sin(c)\cos(c)\sin(b)\cos(b) \\
&\quad
 -4\sin(c)\cos^{3}(c)\sin(b)\cos(b)
 -3\cos^{4}(c)
 +\tfrac{7}{4}\cos^{2}(c) \\
&\quad
 -3\cos^{4}(b)
 +\tfrac{7}{4}\cos^{2}(b)
 +\tfrac{5}{2}
 +\sin(c)\sqrt{3}\cos^{3}(c) \\
&\quad
 -\tfrac{13}{4}\sin(c)\sqrt{3}\cos(c)
 +\sqrt{3}\sin(b)\cos^{3}(b)
 -\tfrac{13}{4}\sin(b)\sqrt{3}\cos(b) \\
&\quad
 -8\cos^{2}(b)\cos^{2}(c)
 +4\cos^{2}(b)\cos^{4}(c)
 +4\cos^{2}(c)\cos^{4}(b)
\big)\Big(1-\tfrac{r^{2}}{2}\Big) \\
&+
 4\sin(c)\cos(c)\sin(b)\cos^{3}(b)
 -7\sin(c)\cos(c)\sin(b)\cos(b) \\
&-
 \sin(c)\sqrt{3}\cos(c)\cos^{2}(b)
 -\sqrt{3}\sin(b)\cos(b)\cos^{2}(c) \\
&+
 4\sin(c)\cos^{3}(c)\sin(b)\cos(b)
 -\tfrac{7}{4}\cos^{2}(b)
 -\tfrac{7}{4}\cos^{2}(c) \\
&+
 3\cos^{4}(c)
 +3\cos^{4}(b)
 -\sin(c)\sqrt{3}\cos^{3}(c)
 +\tfrac{9}{4}\sin(c)\sqrt{3}\cos(c) \\
&-
 \sqrt{3}\sin(b)\cos^{3}(b)
 +\tfrac{9}{4}\sin(b)\sqrt{3}\cos(b) \\
&-
 4\cos^{2}(b)\cos^{4}(c)
 -4\cos^{2}(c)\cos^{4}(b)
 +7\cos^{2}(b)\cos^{2}(c)
 -\tfrac{3}{2}
\Big)
\Bigg)
\end{aligned}
$}

% \mapleresult
\begin{dmath}\label{(4)'}
-\frac{1}{4}+\frac{\left(-4 \sin \! \left(b \right) \cos \! \left(b \right)^{3}-4 \cos \! \left(c \right)^{3} \sin \! \left(c \right)+5 \cos \! \left(b \right) \sin \! \left(b \right)+5 \sin \! \left(c \right) \cos \! \left(c \right)\right) \sqrt{3}}{8}+\frac{\left(3-4 \cos \! \left(c \right)^{2}\right) \cos \! \left(b \right)^{4}}{2}+2 \sin \! \left(c \right) \cos \! \left(c \right) \sin \! \left(b \right) \cos \! \left(b \right)^{3}+\frac{\left(-16 \cos \! \left(c \right)^{4}+24 \cos \! \left(c \right)^{2}-7\right) \cos \! \left(b \right)^{2}}{8}+2 \left(\cos \! \left(c \right)^{2}-\frac{3}{2}\right) \sin \! \left(c \right) \cos \! \left(c \right) \sin \! \left(b \right) \cos \! \left(b \right)+\frac{3 \cos \! \left(c \right)^{4}}{2}-\frac{7 \cos \! \left(c \right)^{2}}{8}
\end{dmath}
\mapleinput{
 $\displaystyle \texttt{>\,} \,\mathrm{h}(\mathrm{b},\mathrm{c})\coloneqq -\frac{1}{4}+\frac{(-4 \sin (b) \cos (b)^{3}-4 \sin (c) \cos (c)^{3}+5 \sin (b) \cos (b)+5 \cos (c) \sin (c)) \sqrt{3}}{8}+\break\frac{(-4 \cos (c)^{2}+3) \cos (b)^{4}}{2}+2 \sin (c) \cos (c) \sin (b) \cos (b)^{3}+\frac{(-16 \cos (c)^{4}+24 \cos (c)^{2}-7) \cos (b)^{2}}{8}+\break2 \sin (b) (\cos (c)^{2}-\frac{3}{2}) \cos (c) \sin (c) \cos (b)+\frac{3 \cos (c)^{4}}{2}-\frac{7 \cos (c)^{2}}{8} $
}
% \mapleresult
\begin{dmath}\label{(5)'}
h \coloneqq \left(b ,c \right)\hiderel{\mapsto }-\frac{1}{4}+\frac{\left(-4\cdot \sin \! \left(b \right)\cdot \cos \! \left(b \right)^{3}-4\cdot \sin \! \left(c \right)\cdot \cos \! \left(c \right)^{3}+5\cdot \sin \! \left(b \right)\cdot \cos \! \left(b \right)+5\cdot \cos \! \left(c \right)\cdot \sin \! \left(c \right)\right)\cdot \sqrt{3}}{8}+\frac{\left(-4\cdot \cos \! \left(c \right)^{2}+3\right)\cdot \cos \! \left(b \right)^{4}}{2}+2\cdot \cos \! \left(c \right)\cdot \sin \! \left(c \right)\cdot \sin \! \left(b \right)\cdot \cos \! \left(b \right)^{3}+\frac{\left(-16\cdot \cos \! \left(c \right)^{4}+24\cdot \cos \! \left(c \right)^{2}-7\right)\cdot \cos \! \left(b \right)^{2}}{8}+2\cdot \sin \! \left(b \right)\cdot \left(\cos \! \left(c \right)^{2}-\frac{3}{2}\right)\cdot \cos \! \left(c \right)\cdot \sin \! \left(c \right)\cdot \cos \! \left(b \right)+\frac{3\cdot \cos \! \left(c \right)^{4}}{2}-\frac{7\cdot \cos \! \left(c \right)^{2}}{8}
\end{dmath}
\mapleinput
{$ \displaystyle \texttt{>\,} \mathit{simplify} (\mathit{int} (\mathit{int} (\frac{h (b ,c)}{\sin (b +c +\frac{\pi}{6})^{4}}\,,b =0..\frac{\pi}{6})\,,c =0..\frac{\pi}{6}))\, $}

% \mapleresult
\begin{dmath}\label{(6)'}
\frac{1}{4}-\frac{5 \ln \! \left(3\right)}{8}+\frac{\pi  \sqrt{3}}{24}+\frac{\ln \! \left(2\right)}{2}
\end{dmath}
\mapleinput
{$ \displaystyle \texttt{>\,} \mathit{plot3d} (\frac{h (b ,c)}{f (c ,b ,0)^{2}},b =0..\frac{ 3.14}{6},c =0..\frac{ 3.14}{6})\, $}

% \mapleresult
\mapleplot{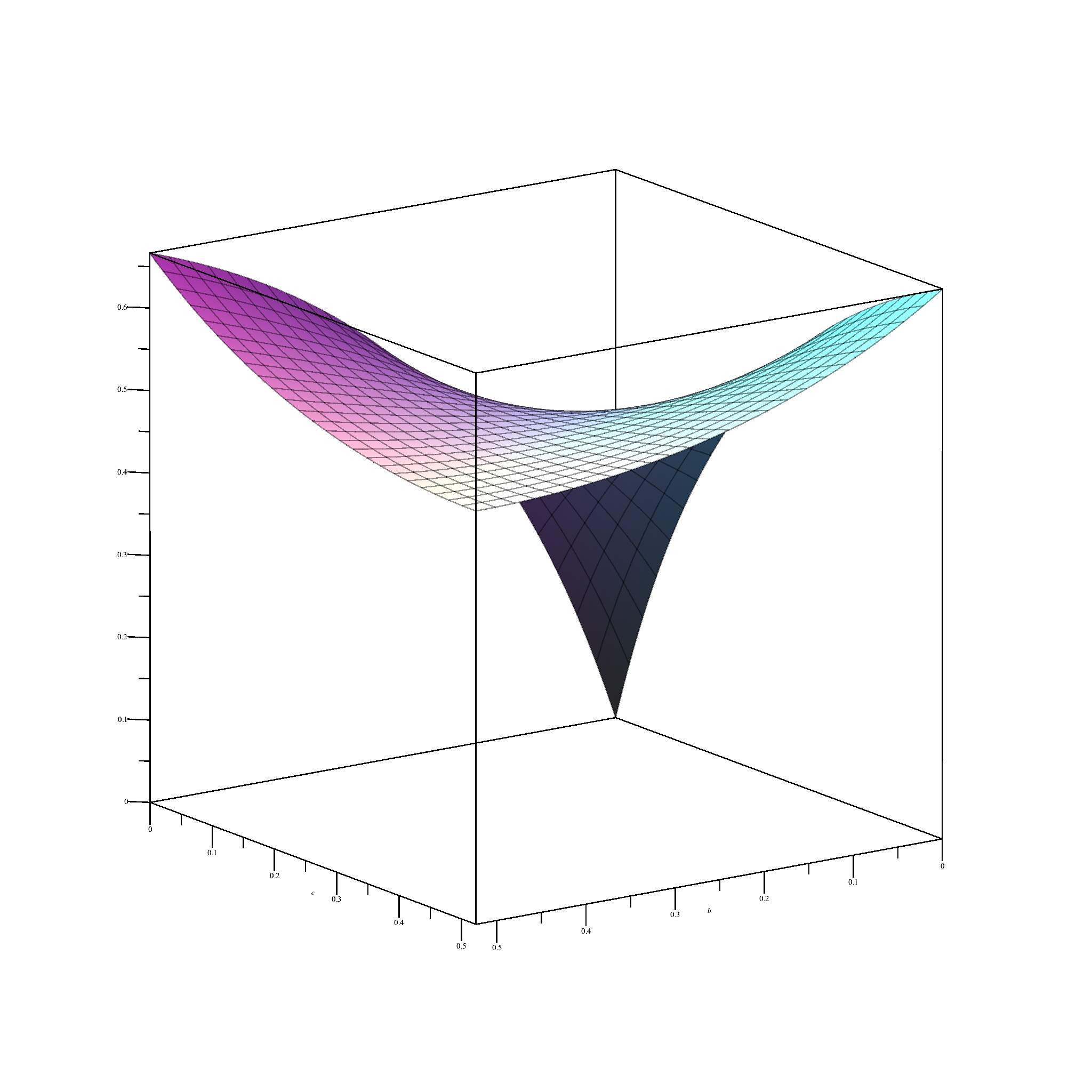}
\mapleinput
{$ \displaystyle \texttt{>\,} \mathit{int} (\mathit{int} (\frac{h (b ,c)}{f (c ,b ,0)^{2}}\,,b =0..\frac{ 3.14159265359}{6},\mathit{numeric} =\mathit{true})\,,c =0..\frac{ 3.14159265359}{6},\mathit{numeric} =\mathit{true}) $}

% \mapleresult
\begin{dmath}\label{(7)'}
 0.1366658305
\end{dmath}
% \mapleinput
% {$ \displaystyle \texttt{>\,} \, $}

\subsection{Computation for the integral in Lemma~\ref{lem:4thregime}}\label{subsec:integral_v3}

\mapleinput
{$ \displaystyle \texttt{>\,} f (b ,c ,r)\coloneqq (\frac{1}{2}-\cos (\frac{5}{6}\,\pi -c -r)\cdot \sin (c)-\frac{1}{2}\cdot \cos (\frac{2\cdot \,\pi}{3}+r)-\cos (\frac{\pi}{2}-b)\cdot \sin (b))^{2}+(\sin (\frac{5}{6}\,\pi -c -r)\cdot \sin (c)-(\frac{1}{2}\cdot \sin (\frac{2\cdot \,\pi}{3}+r)+\sin (\frac{\pi}{2}-b)\cdot \sin (b)))^{2}\, $}

% \mapleresult
\begin{dmath}\label{(1)''}
f \coloneqq \left(b ,c ,r \right)\hiderel{\mapsto }\left(\frac{1}{2}-\cos \! \left(\frac{5\cdot \pi}{6}-c -r \right)\cdot \sin \! \left(c \right)-\frac{\cos \! \left(\frac{2\cdot \pi}{3}+r \right)}{2}-\cos \! \left(-b +\frac{\pi}{2}\right)\cdot \sin \! \left(b \right)\right)^{2}+\left(\sin \! \left(\frac{5\cdot \pi}{6}-c -r \right)\cdot \sin \! \left(c \right)-\frac{\sin \! \left(\frac{2\cdot \pi}{3}+r \right)}{2}-\sin \! \left(-b +\frac{\pi}{2}\right)\cdot \sin \! \left(b \right)\right)^{2}
\end{dmath}
\mapleinput
{$ \displaystyle \texttt{>\,} g (b ,c ,r)\coloneqq \mathit{simplify} (\mathit{expand} (f (b ,c ,0)^{2}-f (b ,c ,r)\cdot f (b ,c ,-r)))\, $}

% \mapleresult
\begin{dmath}\label{(2)''}
g \coloneqq \left(b ,c ,r \right)\hiderel{\mapsto }\mathit{simplify} \! \left(\mathit{expand} \! \left(f \! \left(b ,c ,0\right)^{2}-f \! \left(b ,c ,r \right)\cdot f \! \left(b ,c ,-r \right)\right)\right)
\end{dmath}
\mapleinput
{$ \displaystyle \texttt{>\,} \mathit{collect} (\mathit{expand} (g (b ,c ,r)),\cos (r))\, $}

% \mapleresult
\begin{dmath}\label{(3)''}
-\frac{\cos \! \left(r \right)^{2}}{4}+\left(\sin \! \left(b \right) \cos \! \left(b \right) \sin \! \left(c \right) \cos \! \left(c \right)-\sin \! \left(b \right) \cos \! \left(b \right) \sqrt{3}\, \cos \! \left(c \right)^{2}+\frac{\sqrt{3}\, \cos \! \left(b \right) \sin \! \left(b \right)}{2}-\sin \! \left(c \right) \cos \! \left(c \right) \sqrt{3}\, \cos \! \left(b \right)^{2}-\cos \! \left(c \right)^{2} \cos \! \left(b \right)^{2}+\frac{\cos \! \left(b \right)^{2}}{2}+\frac{\sin \! \left(c \right) \sqrt{3}\, \cos \! \left(c \right)}{2}+\frac{\cos \! \left(c \right)^{2}}{2}-\frac{1}{4}\right) \cos \! \left(r \right)+\frac{\sin \! \left(b \right)^{2}}{2}-\sin \! \left(b \right) \cos \! \left(b \right) \sin \! \left(c \right) \cos \! \left(c \right)+\sin \! \left(b \right) \cos \! \left(b \right) \sqrt{3}\, \cos \! \left(c \right)^{2}-\frac{\sqrt{3}\, \cos \! \left(b \right) \sin \! \left(b \right)}{2}+\sin \! \left(c \right) \cos \! \left(c \right) \sqrt{3}\, \cos \! \left(b \right)^{2}+\cos \! \left(c \right)^{2} \cos \! \left(b \right)^{2}-\frac{\sin \! \left(c \right) \sqrt{3}\, \cos \! \left(c \right)}{2}-\frac{\cos \! \left(c \right)^{2}}{2}
\end{dmath}
% \mapleinput
% {$ \displaystyle \texttt{>\,} \mathit{simplify} (\frac{((-\frac{1}{4}) \cdot (1-r^{2})+(-\sqrt{3} \cos (c)^{2} \cos (b) \sin (b)-\cos (b)^{2} \cos (c)^{2}+\frac{\cos (c)^{2}}{2}-\sqrt{3} \sin (c) \cos (c) \cos (b)^{2}+\sin (c) \cos (c) \cos (b) \sin (b)+\frac{\sin (c) \sqrt{3} \cos (c)}{2}+\frac{\sqrt{3} \sin (b) \cos (b)}{2}+\frac{\cos (b)^{2}}{2}-\frac{1}{4}) \cdot (1-\frac{r^{2}}{2})+\sqrt{3} \cos (c)^{2} \cos (b) \sin (b)+\cos (b)^{2} \cos (c)^{2}+\sqrt{3} \sin (c) \cos (c) \cos (b)^{2}-\sin (c) \cos (c) \cos (b) \sin (b)-\frac{\sin (c) \sqrt{3} \cos (c)}{2}+\frac{\sin (c)^{2}}{2}-\frac{\sqrt{3} \sin (b) \cos (b)}{2}-\frac{\cos (b)^{2}}{2})}{r^{2}}) $}

\mapleinput{
$\displaystyle
\begin{aligned}
\texttt{>\,}\,\mathit{simplify}\Bigg(
\frac{1}{r^{2}}\Big(&
-\tfrac{1}{4}(1-r^{2}) \\
&+
\Big(
 -\sqrt{3}\cos^{2}(c)\cos(b)\sin(b)
 -\cos^{2}(b)\cos^{2}(c)
 +\tfrac{1}{2}\cos^{2}(c) \\
&\quad
 -\sqrt{3}\sin(c)\cos(c)\cos^{2}(b)
 +\sin(c)\cos(c)\cos(b)\sin(b) \\
&\quad
 +\tfrac{1}{2}\sin(c)\sqrt{3}\cos(c)
 +\tfrac{1}{2}\sqrt{3}\sin(b)\cos(b)
 +\tfrac{1}{2}\cos^{2}(b)
 -\tfrac{1}{4}
\Big) \Big(1-\tfrac{r^{2}}{2}\Big) \\
&+
 \sqrt{3}\cos^{2}(c)\cos(b)\sin(b)
 +\cos^{2}(b)\cos^{2}(c)
 +\sqrt{3}\sin(c)\cos(c)\cos^{2}(b) \\
&-
 \sin(c)\cos(c)\cos(b)\sin(b)
 -\tfrac{1}{2}\sin(c)\sqrt{3}\cos(c)
 +\tfrac{1}{2}\sin^{2}(c) \\
&-
 \tfrac{1}{2}\sqrt{3}\sin(b)\cos(b)
 -\tfrac{1}{2}\cos^{2}(b)
\Big)
\Bigg)
\end{aligned}
$}

% \mapleresult
\begin{dmath}\label{(4)''}
\frac{3}{8}+\frac{\left(2 \cos \! \left(b \right)^{2} \sin \! \left(c \right) \cos \! \left(c \right)+\sin \! \left(b \right) \left(-1+2 \cos \! \left(c \right)^{2}\right) \cos \! \left(b \right)-\sin \! \left(c \right) \cos \! \left(c \right)\right) \sqrt{3}}{4}+\frac{\left(-1+2 \cos \! \left(c \right)^{2}\right) \cos \! \left(b \right)^{2}}{4}-\frac{\sin \! \left(b \right) \cos \! \left(b \right) \sin \! \left(c \right) \cos \! \left(c \right)}{2}-\frac{\cos \! \left(c \right)^{2}}{4}
\end{dmath}
% \mapleinput
% {$ \displaystyle \texttt{>\,}  $}

\mapleinput
{$ \displaystyle \begin{aligned}\texttt{>\,} h (b ,c)\coloneqq &\frac{3}{8}+\frac{(2 \sin (c) \cos (c) \cos (b)^{2}+\sin (b) (2 \cos (c)^{2}-1) \cos (b)-\sin (c) \cos (c)) \sqrt{3}}{4}+ \\&\frac{(2 \cos (c)^{2}-1) \cos (b)^{2}}{4}-\frac{\sin (c) \cos (c) \cos (b) \sin (b)}{2}-\frac{\cos (c)^{2}}{4} \end{aligned}$}

% \mapleresult
\begin{dmath}\label{(5)''}
h \coloneqq \left(b ,c \right)\hiderel{\mapsto }\frac{3}{8}+\frac{\left(2\cdot \sin \! \left(c \right)\cdot \cos \! \left(c \right)\cdot \cos \! \left(b \right)^{2}+\sin \! \left(b \right)\cdot \left(2\cdot \cos \! \left(c \right)^{2}-1\right)\cdot \cos \! \left(b \right)-\sin \! \left(c \right)\cdot \cos \! \left(c \right)\right)\cdot \sqrt{3}}{4}+\frac{\left(2\cdot \cos \! \left(c \right)^{2}-1\right)\cdot \cos \! \left(b \right)^{2}}{4}-\frac{\sin \! \left(c \right)\cdot \cos \! \left(c \right)\cdot \cos \! \left(b \right)\cdot \sin \! \left(b \right)}{2}-\frac{\cos \! \left(c \right)^{2}}{4}
\end{dmath}
\mapleinput
{$ \displaystyle \texttt{>\,} \mathit{simplify} (\mathit{int} (\mathit{int} (\frac{h (b ,c)}{\sin (b +c +\frac{\pi}{3})^{4}}\,,b =0..\frac{\pi}{6})\,,c =0..\frac{\pi}{6})) $}

% \mapleresult
\begin{dmath}\label{(6)''}
-\frac{\ln \! \left(3\right)}{2}+\ln \! \left(2\right)
\end{dmath}
% \mapleinput
% {$ \displaystyle \texttt{>\,}  $}

\mapleinput
{$ \displaystyle \texttt{>\,} \,\mathit{int} (\mathit{int} (\frac{h (b ,c)}{f (c ,b ,0)^{2}}\,,b =0..\frac{ 3.1415926535}{6},\mathit{numeric} =\mathit{true})\,,c =0..\frac{ 3.1415926535}{6},\mathit{numeric} =\mathit{true}) $}

% \mapleresult
\begin{dmath}\label{(7)''}
 0.1438410363
\end{dmath}
% \mapleinput
% {$ \displaystyle \texttt{>\,} \, $}

% \mapleresult

\section{Computation of $J$ }\label{app:comp_J}

We compute $J$ by passing to Fourier series. The first ingredient is the classical expansion of $\mathrm{tri}$.

\begin{lem}\label{lem:triFourier}
The triangular wave has the Fourier expansion
\[
\mathrm{tri}(x)=\sum_{\substack{k\ge1\\ k\ \mathrm{odd}}}\frac{8}{\pi^2k^2}\cos(kx),
\]
with absolute (hence uniform) convergence.
\end{lem}

\begin{proof}
Since $\mathrm{tri}$ is even and has mean $0$, only cosine coefficients appear:
\[
a_k=\frac{1}{\pi}\int_{-\pi}^{\pi}\mathrm{tri}(x)\cos(kx)\,dx
=\frac{2}{\pi}\int_0^{\pi}\left(1-\frac{2x}{\pi}\right)\cos(kx)\,dx.
\]
The first term integrates to $0$, and integration by parts gives
$\int_0^\pi x\cos(kx)\,dx=\frac{(-1)^k-1}{k^2}$, hence
\[
a_k=\frac{4}{\pi^2}\cdot\frac{1-(-1)^k}{k^2}
=
\begin{cases}
\frac{8}{\pi^2k^2},& k\ \text{odd},\\
0,& k\ \text{even}.
\end{cases}
\]
Absolute convergence follows from $\sum_{k\ \mathrm{odd}}1/k^2<\infty$.
\end{proof}

For integers $k$ define, for $x\neq y$,
\[
R_k(x,y)=\frac{e^{ikx}-e^{iky}}{e^{ix}-e^{iy}}.
\]

\begin{lem}\label{lem:RkIntegral}
For integers $k,\ell$ one has
\[
\frac{1}{4\pi^2}\int_0^{2\pi}\int_0^{2\pi} R_k(x,y)\,R_\ell(x,y)\,dx\,dy
=
\begin{cases}
1-|k-1|,& k+\ell=2,\\
0,& k+\ell\neq 2.
\end{cases}
\]
\end{lem}

\begin{proof}
Let $z=e^{ix}$ and $w=e^{iy}$.
If $k\ge1$ then
\[
R_k(x,y)=\frac{z^k-w^k}{z-w}=\sum_{m=0}^{k-1} z^{k-1-m}w^m
=\sum_{m=0}^{k-1} e^{i((k-1-m)x+my)}.
\]
If $k\le0$ then, writing $k=-p$ with $p\ge0$,
\[
R_{-p}(x,y)=\frac{z^{-p}-w^{-p}}{z-w}
=\frac{w^p-z^p}{z^p w^p(z-w)}
=-\frac{z^p-w^p}{z^p w^p(z-w)}
=-\sum_{m=0}^{p-1} z^{-(m+1)}w^{-(p-m)}.
\]
In all cases, $R_kR_\ell$ is a finite sum of exponentials $e^{i(ax+by)}$.
Using
\[
\frac{1}{2\pi}\int_0^{2\pi} e^{iax}\,dx=
\begin{cases}
1,& a=0,\\
0,& a\neq 0,
\end{cases}
\]
one checks that a nonzero contribution can occur only when the total $x$-frequency and $y$-frequency both vanish,
which forces $k+\ell=2$.
When $k+\ell=2$, the number of surviving terms is $1-|k-1|$, and each contributes $1$.
\end{proof}

We use Fej\'er approximation to justify passing from the kernel identity to our Lipschitz $f$.

For $N\ge 0$ define the \emph{Fej\'er kernel}
\begin{equation}\label{eq:FejerKernelDef}
K_N(t)=\frac{1}{N+1}\left(\frac{\sin \bigl((N+1)t/2\bigr)}{\sin(t/2)}\right)^2,
\qquad t\in\mathbb R,
\end{equation}
and for a $2\pi$--periodic function $f:\mathbb R\to\mathbb C$ define its \emph{Fej\'er mean} by
\begin{equation}\label{eq:FejerMeanDef}
f^{(N)}(x)=\frac{1}{2\pi}\int_{0}^{2\pi} f(x-t)\,K_N(t)\,dt.
\end{equation}
A $2\pi$--periodic function $f$ is \emph{Lipschitz} (on the circle) if there exists $L\ge 0$ such that
\begin{equation}\label{eq:LipDef}
|f(x)-f(y)|\le L\,d(x,y)\qquad \text{for all }x,y\in\mathbb R,
\end{equation}
and we write $\Lip(f)$ for the smallest such $L$.

\begin{lem}\label{lem:FejerLip}
If $f$ is $2\pi$--periodic and Lipschitz on the circle, then for every $N\ge 0$,
\[
\Lip\bigl(f^{(N)}\bigr)\le \Lip(f).
\]
Moreover, for all $x,y\in\mathbb R$ with $d(x,y)>0$,
\[
\left|\frac{f^{(N)}(x)-f^{(N)}(y)}{e^{ix}-e^{iy}}\right|
\le \frac{\pi}{2}\,\Lip(f),
\]
uniformly in $N$.
\end{lem}

\begin{proof}
Since $K_N\ge 0$ and $\frac{1}{2\pi}\int_0^{2\pi}K_N(t)\,dt=1$, and since $d(x-t,y-t)=d(x,y)$, we have
\[
|f^{(N)}(x)-f^{(N)}(y)|
\le \frac{1}{2\pi}\int_0^{2\pi}|f(x-t)-f(y-t)|\,K_N(t)\,dt
\le \Lip(f)\,d(x,y),
\]
which implies $\Lip(f^{(N)})\le \Lip(f)$.
For $x\neq y$ we have $|e^{ix}-e^{iy}|=2\sin(d(x,y)/2)\ge 2d(x,y)/\pi$, and the quotient bound follows.
\end{proof}

\begin{lem}\label{lem:JFourier}
Let $\rho(x,y)=\frac{f(x)-f(y)}{e^{ix}-e^{iy}}$ as in~\eqref{eq:def_riemann_sum_limit_for_ss_1} and define
\[
B(f)=\frac{1}{4\pi^2}\int_0^{2\pi}\int_0^{2\pi}\rho(x,y)^2\,dx\,dy.
\]
Then $B(f)$ is real and equals
\[
B(f)=\sum_{k\in\mathbb Z}\bigl(1-|k-1|\bigr)\,a_k\,a_{2-k},
\]
where $f(x)=\sum_{k\in\mathbb Z} a_k e^{ikx}$ is the Fourier series of $f$.
\end{lem}

\begin{proof}
Let $f^{(N)}$ be Fej\'er means. Define
\[
\rho_N(x,y)=\frac{f^{(N)}(x)-f^{(N)}(y)}{e^{ix}-e^{iy}}.
\]
By uniform convergence $f^{(N)}\to f$, we have $\rho_N(x,y)\to \rho(x,y)$ for every $x\neq y$. For our specific $f(x)=g(x)e^{ix}$, Lemma~\ref{lem:gprops} implies $|g|\le1$ and $\Lip(g)\le L$.
Hence, for all $x,y\in\mathbb R$,
\[
|f(x)-f(y)|
\le |g(x)-g(y)|+|g(y)|\,|e^{ix}-e^{iy}|
\le L\,d(x,y)+2\sin(d(x,y)/2)
\le (L+1)\,d(x,y),
\]
so $\Lip(f)\le L+1$. By Lemma~\ref{lem:FejerLip}, $|\rho_N(x,y)|\le \frac{\pi}{2}\Lip(f)$ uniformly in $N$ and $(x,y)$.
Hence dominated convergence applies and
\[
B(f)=\lim_{N\to\infty}\frac{1}{4\pi^2}\int_0^{2\pi}\int_0^{2\pi} \rho_N(x,y)^2\,dx\,dy.
\]
Now $f^{(N)}$ is a trigonometric polynomial, say
$f^{(N)}(x)=\sum_{|k|\le N} a_k^{(N)} e^{ikx}$.
Then for $x\neq y$,
\[
\rho_N(x,y)=\sum_{|k|\le N} a_k^{(N)}\,R_k(x,y),
\]
a finite sum. Therefore,
\[
\frac{1}{4\pi^2}\int_0^{2\pi}\int_0^{2\pi} \rho_N(x,y)^2\,dx\,dy
=\sum_{|k|,|\ell|\le N} a_k^{(N)}a_\ell^{(N)}
\cdot \frac{1}{4\pi^2}\int_0^{2\pi}\int_0^{2\pi} R_k R_\ell\,dx\,dy.
\]
By Lemma~\ref{lem:RkIntegral}, only terms with $k+\ell=2$ survive, giving
\[
\frac{1}{4\pi^2}\int_0^{2\pi}\int_0^{2\pi} \rho_N(x,y)^2\,dx\,dy
=\sum_{k\in\mathbb Z}\bigl(1-|k-1|\bigr)a_k^{(N)}a_{2-k}^{(N)}.
\]
For Fej\'er means, $a_k^{(N)}=(1-|k|/(N+1))a_k$ for $|k|\le N$ and $0$ otherwise, hence $a_k^{(N)}\to a_k$
for each fixed $k$.
In our application the resulting infinite series is absolutely convergent (indeed it will reduce to $\sum_{r\ \mathrm{odd}}O(1/r^3)$),
so we may pass $N\to\infty$ termwise.
This yields the desired series for $B(f)$.
Since all $a_k$ are real, $B(f)\in\mathbb R$.
\end{proof}

We now compute $J$ explicitly for $f(\theta)=g(\theta)e^{i\theta}$ by combining the Fourier expansion of $g$
with Lemma~\ref{lem:JFourier}.

\begin{lem}\label{lem:Jvalue}
Let $J$ be the constant defined in Lemma~\ref{lem:Riemann}. Then
\[
J=\frac{1}{3}-\frac{84\,\zeta(3)}{\pi^4}<0.
\]
\end{lem}

\begin{proof}
By Lemma~\ref{lem:triFourier} we have the absolutely convergent Fourier expansion
\[
g(\theta)=\mathrm{tri}(3\theta)
=\sum_{\substack{r\ge1\\ r\ \mathrm{odd}}}\frac{8}{\pi^2r^2}\cos(3r\theta)
=\sum_{\substack{r\ge1\\ r\ \mathrm{odd}}}\frac{4}{\pi^2r^2}\Bigl(e^{i3r\theta}+e^{-i3r\theta}\Bigr).
\]
Multiplying by $e^{i\theta}$ gives
\[
f(\theta)=g(\theta)e^{i\theta}
=\sum_{\substack{r\ge1\\ r\ \mathrm{odd}}}\frac{4}{\pi^2r^2}
\Bigl(e^{i(1+3r)\theta}+e^{i(1-3r)\theta}\Bigr).
\]
Hence the Fourier coefficients of $f(\theta)=\sum_{k\in\mathbb Z}a_ke^{ik\theta}$ are
\[
a_{1+3r}=a_{1-3r}=\frac{4}{\pi^2r^2}\qquad(r\ge1,\ r\ \mathrm{odd}),
\qquad a_k=0\ \text{otherwise}.
\]
Since all $a_k$ are real, Lemma~\ref{lem:JFourier} yields $J=B(f)$ and
\[
J=\sum_{k\in\mathbb Z}\bigl(1-|k-1|\bigr)\,a_k\,a_{2-k}.
\]
A term is nonzero only if both $a_k$ and $a_{2-k}$ are nonzero. This forces
$k=1+3r$ and $2-k=1-3r$ for some odd $r\ge1$ (or the same pair in the reversed order).
For such $k$ we have $1-|k-1|=1-3r$, and therefore
\[
J=\sum_{\substack{r\ge1\\ r\ \mathrm{odd}}}2(1-3r)\left(\frac{4}{\pi^2r^2}\right)^2
=\frac{32}{\pi^4}\sum_{\substack{r\ge1\\ r\ \mathrm{odd}}}\frac{1-3r}{r^4}.
\]
Using
\[
\sum_{\substack{r\ge1\\ r\ \mathrm{odd}}}\frac1{r^4}
=\left(1-\frac1{2^4}\right)\zeta(4)=\frac{\pi^4}{96},
\qquad
\sum_{\substack{r\ge1\\ r\ \mathrm{odd}}}\frac1{r^3}
=\left(1-\frac1{2^3}\right)\zeta(3)=\frac{7}{8}\zeta(3),
\]
we obtain
\[
J=\frac{32}{\pi^4}\left(\frac{\pi^4}{96}-3\cdot\frac{7}{8}\zeta(3)\right)
=\frac{1}{3}-\frac{84\,\zeta(3)}{\pi^4}<0,
\]
as claimed.
\end{proof}

\end{document}